\documentclass[12pt,a4paper,twoside,reqno]{amsart}
\title[(0,2) mirror symmetry on homogeneous Hopf surfaces]
{(0,2) mirror symmetry on homogeneous Hopf surfaces}

\date{9 May 2023}

\author[L. \'Alvarez-C\'onsul]{Luis \'Alvarez-C\'onsul}
\address{Instituto de Ciencias Matem\'aticas (CSIC-UAM-UC3M-UCM)\\ Nicol\'as Cabrera 13--15, Cantoblanco\\ 28049 Madrid, Spain}
  \email{l.alvarez-consul@icmat.es}
\author[A. De Arriba de La Hera]{Andoni De Arriba de La Hera}
\address{Dep. \'Algebra, Geometr\'ia y Topolog\'ia\\ Universidad Complutense de Madrid\\ and Instituto de Ciencias Matem\'aticas (CSIC-UAM-UC3M-UCM)\\ Nicol\'as Cabrera 13--15, Cantoblanco\\ 28049 Madrid, Spain}
  \email{andoni.dearriba@icmat.es}
  \author[M. Garcia-Fernandez]{Mario Garcia-Fernandez}
\address{Dep. Matem\'aticas\\ Universidad Aut\'onoma de Madrid\\ and
  Instituto de Ciencias Matem\'ati\-cas (CSIC-UAM-UC3M-UCM)\\ Ciudad
  Universitaria de Cantoblanco\\ 28049 Madrid, Spain}
\email{mario.garcia@icmat.es}

\thanks{
Partially supported by the Spanish Ministry of Science and Innovation, through the `Severo Ochoa Programme for Centres of Excellence in R\&D' (SEV-2015-0554 and CEX2019-000904-S), and grant MTM2016-81048-P. The first author is partially supported by MICINN under grant PID2019-109339GB-C31. The second author is partially supported by MICINN under grant BES-2017-080578. The second and third authors are partially supported by MICINN under grants PID2019-109339GA-C32 and EUR2020-112265.
}

\usepackage{hyperref,url}
\usepackage{amssymb,latexsym}
\usepackage{amsmath,amsthm}
\usepackage[latin1]{inputenc}
\usepackage{enumerate}
\usepackage[all]{xy} \CompileMatrices
\SelectTips{cm}{12} 
\usepackage{tikz-cd}

\usepackage{verbatim} 
\usepackage{wrapfig}
\usepackage{graphicx}
\usepackage{slashed}
\usepackage{multicol}

\theoremstyle{plain}
\newtheorem{theorem}{Theorem}[section]
\newtheorem{lemma}[theorem]{Lemma}

\newtheorem{proposition}[theorem]{Proposition}

\newtheorem*{theorem*}{Theorem}
\theoremstyle{definition}
\newtheorem{definition}[theorem]{Definition}
\newtheorem{definition-theorem}[theorem]{Definition-Theorem}
\newtheorem{example}[theorem]{Example}
\newtheorem*{acknowledgements}{Acknowledgements}
\theoremstyle{remark}
\newtheorem{remark}[theorem]{Remark}

\numberwithin{equation}{section} 

\newcommand{\qf}[2]{{\left({#1}\middle|{#2}\right)}} 

\newcommand{\tr}{\operatorname{tr}}
\newcommand{\Id}{\operatorname{Id}}

\newcommand{\End}{\operatorname{End}}

\newcommand{\CC}{{\mathbb C}}

\newcommand{\RR}{{\mathbb R}}
\newcommand{\ZZ}{{\mathbb Z}}

\renewcommand{\)}{\right)}

\newcommand{\defeq}{\mathrel{\mathop:}=} 

\newcommand{\hra}{\hookrightarrow}
\newcommand{\lto}{\longrightarrow}

\newcommand{\cD}{\mathcal{D}}

\newcommand{\g}{\mathfrak{g}}
\newcommand{\cL}{\mathcal{L}}

\newcommand{\cR}{\mathcal{R}}

\newcommand{\Lie}{\operatorname{Lie}}

\newcommand{\cH}{\mathcal{H}} 

\newcommand{\U}{\operatorname{U}}
\newcommand{\SU}{\operatorname{SU}}

\newcommand{\glg}{\mathfrak{g}}

\newcommand{\C}{{\mathbb{C}}}
\newcommand{\R}{{\mathbb{R}}}

\newcommand{\Z}{{\mathbb{Z}}}
\newcommand{\N}{{\mathbb{N}}}

\newcommand{\la}{\langle}
\newcommand{\ra}{\rangle}

\begin{document}

\begin{abstract}
In this work we find the first examples of (0,2) mirror symmetry on compact non-K\"ahler complex manifolds. For this we follow Borisov's approach to mirror symmetry using vertex algebras and the chiral de Rham complex. Our examples of (0,2) mirrors are given by pairs of Hopf surfaces endowed with a Bismut-flat pluriclosed metric. Requiring that the geometry is homogeneous, we reduce the problem to the study of Killing spinors on a quadratic Lie algebra and the construction of embeddings of the $N=2$ superconformal vertex algebra in the superaffine vertex algebra, combined with topological T-duality.
\end{abstract}

\maketitle

\setlength{\parskip}{5pt}
\setlength{\parindent}{0pt}


\section{Introduction}
\label{sec:intro}

Shortly after the discovery of mirror symmetry in string theory, physicists realized that this duality extends to a class of superconformal field theories dubbed \emph{(0,2)-models} \cite{WittenN2}. Unlike the more familiar story for algebraic Calabi--Yau manifolds, the geometric content of these models is given by a pair $(X,E)$, where $X$ is a Calabi--Yau threefold and $E$ is a stable holomorphic vector bundle over $X$ satisfying 
$$
c_1(E) = 0, \qquad ch_2(E) = ch_2(X).
$$
Such pairs are supposed to be exchanged under \emph{(0,2)-mirror symmetry}, in a way that some of the features of the standard picture are recovered when $E = TX$ \cite{McOrist,MelSeShe}. For a general $E$, however, $(0,2)$ mirror symmetry should be very different \cite{GGG,MelBook} to the exchange of the A-model with the B-model in the homological or SYZ versions of mirror symmetry \cite{KontsevichHMS,SYZ}.

At this point a natural question is forced upon us: what is $(0,2)$ mirror symmetry?
An important motivation for this question came with the observation that $(0,2)$ models have associated \emph{chiral rings} \cite{ABS} which generalize the quantum cohomology ring of a symplectic manifold.
There has been some direct attempts to understand this picture in the mathematics literature with the aid of toric geometry \cite{Borisov02,Donagi}, but the story is far from over. 
These mathematical advances provide underpinnings for the \emph{half-twisted} model associated to a $(0,2)$-theory \cite{WittenMS} -- a closely-related but simpler theory which plays a similar role to that of the A and B models in mirror symmetry. According to Witten \cite{WittenMS}, $(0,2)$ mirror symmetry exchanges the half-twisted model of $(X,E)$ with the half-twisted model of its mirror and therefore it has the potential to provide an indirect method to understand the chiral rings \cite{MelPle}. This alternative should run parallel to the classical prediction on the number of fixed-degree rational curves on the quintic Calabi--Yau threefold, by Candelas, De la Ossa, Green and Parkes \cite{COGP}.
   
The construction of the chiral ring in \cite{Borisov02} is based on a general approach to mirror symmetry using vertex algebras by Borisov \cite{Borisov}, 
which is well-suited for understanding certain aspects of $(0,2)$ mirror symmetry. The idea is that, even though the notion of superconformal field theory which underlies mirror symmetry is not yet properly axiomatized, part of the structure of such a theory
is a vertex algebra with an $N=2$ superconformal structure, which does have a rigorous mathematical meaning \cite{Kac} (see Section \ref{ssec:background}). The most basic features of this algebraic object are a vector superspace $V$ and four \emph{quantum fields} 
\begin{equation*}
L, J, G^+, G^- \in \operatorname{End} V[[z^{\pm 1}]],
\end{equation*}
whose Fourier modes satisfy suitable commutation relations (see Example \ref{exam:N2}). In this algebraic set-up, $(0,2)$ mirror symmetry is expected to be an isomorphism of the vertex algebras corresponding to the $(0,2)$ models of $(X,E)$ and its mirror $(\hat X, \hat E)$, which flips the $N=2$ superconformal structures via the \emph{mirror involution}
\begin{equation}\label{eq:mirrorinvintro}
L \longleftrightarrow \hat L, \qquad   J \longleftrightarrow - \hat J, \qquad  G_\pm \longleftrightarrow \hat G_\mp.
\end{equation}
The mathematical definition of the $N=2$ superconformal structure associated to a pair $(X,E)$ is expected to involve certain global sections of the (twisted) chiral de Rham complex \cite{MSV,GSV}, a sheaf of vertex algebras on a smooth manifold introduced by Malikov, Schechtman and Vaintrob. Despite the accumulated physical knowledge on $(0,2)$ models, a construction of these global sections is not known at present. In relation to the half-twisted model, it was proved by Witten that the cohomology of the holomorphic chiral de Rham complex is related in perturbation theory to the chiral algebra of the half-twisted theory \cite{Witten02} (see also \cite{Kapustin}).

There has been some progress in Borisov's Programme in recent years, mainly due to the work of Heluani, Kac, and their collaborators \cite{HeluaniRev}, but this proposal is far from being complete: even for the case of algebraic Calabi--Yau manifolds, the construction of $N=2$ structures on the chiral de Rham complex often involves the $\SU(n)$-holonomy metric provided by Yau's theorem, which we do not know explicitly. For the case of pairs $(X,E)$ the replacement of the metric is expected to be a solution of the Hull--Strominger system \cite{HullTurin,Strom}, a fully non-linear system of partial differential equations about which we know very little at present. A novel feature of the solutions of the Hull--Strominger system is that they involve non-K\"ahler metrics which often admit continuous isometry groups, due to the presence of torsion. Over the years, this has led to the construction of large classes of examples on non-K\"ahler Calabi--Yau manifolds with an explicit description of the metric (see, e.g., \cite{FeiYau,OUVi}). Thus non-K\"ahler geometry provides a promising scenario for understanding the vertex algebra approach to $(0,2)$ mirror symmetry. Using the methods from \cite{GF3}, first candidates of $(0,2)$ mirrors on compact non-K\"ahler manifolds have been recently constructed by the third author in \cite{GF4} via T-duality.

Even though the bundle $E$ allows for symmetries, it also carries all sorts of additional difficulties compared to the standard Calabi--Yau case \cite{grst,PPZ}. According to physicists, one way to have non-vanishing torsion in the absence of $E$ is to add singularities to the solutions caused by the back-reaction of the geometry when a \emph{brane} is in place. A famous example is the \emph{heterotic string soliton} constructed by Callan, Harvey, and Strominger \cite{CHS} on the non-compact manifold
$$
\widetilde X = \mathbb{C}^2 \backslash \{0\},
$$ 
corresponding to a NS5-brane located at the origin. Near an asymptotic $S^3$ with \emph{flux} $\ell \in \mathbb{N}$ surrounding the fivebrane (the throat of a semi-wormhole) we find a familiar non-K\"ahler metric: a Bismut-flat pluriclosed hermitian metric. Taking the quotient by a radial $\mathbb{Z}$-action we obtain a compact Hopf surface $X$ with a preferred geometry (see Remark \ref{rem:Zquotient}).

Inspired by this, in the present work we find the first examples of $(0,2)$ mirror symmetry on compact non-K\"ahler complex manifolds using vertex algebras. Our examples of $(0,2)$ mirrors in Theorem \ref{thm:02} are given by pairs of Hopf surfaces endowed with a Bismut-flat pluriclosed metric, with fixed topology $S^3 \times S^1$. The quantity $\ell$ is realized as the cohomology class of the torsion of the hermitian metric on the underlying smooth manifold, which remains fixed:
\begin{equation}\label{eq:Severaclassintro}
[H_\ell] = \ell \in H^3(S^3 \times S^1,\mathbb{R}) \cong \mathbb{R}.
\end{equation}
We will show that, in our examples, (0,2)-mirror symmetry exchanges the complex structure on $X$ with the Aeppli class of the pluriclosed metric of the mirror $\hat X$, and vice versa (see Proposition \ref{prop:Tdual}). This phenomenon is strongly reminiscent of the rotation of the Hodge diamond on mirror symmetry for algebraic Calabi--Yau manifolds. More importantly, in Theorem \ref{thm:halftwist} we consider an $S^1$-equivariant version of the \emph{half-twisted model} $\Omega^{ch}_{X}(\mathcal{Q}_{\ell})^{T^1}$ following \cite{StructuresHeluani,LinshawMathai}, and prove that $(0,2)$ mirror symmetry yields an isomorphism
\begin{equation}\label{eq:halftwistmirrorintro}
\Omega^{ch}_{X}(\mathcal{Q}_{\ell})^{T^1} \cong \Omega^{ch}_{\overline{\hat X}}(\overline{\hat{\mathcal{Q}}}_{\ell})^{T^1}
\end{equation}
in our examples. The identity is remarkable, as it gives an isomorphism of holomorphic sheaves of SUSY vertex algebras over different complex manifolds. It shall be compared with a fundamental result of Borisov and Libgober \cite{BorLib}, which matches the elliptic genera of a Calabi--Yau hypersurface on a Fano toric manifold with the one of its mirror. 

The study of $(0,2)$ models on Hopf surfaces goes back to Witten \cite{Witten02}. Our result confirms physical expectations about mirror symmetry for these models \cite{Tan}. For the latest developments on $(0,2)$ mirror symmetry in relation to Landau-Ginzburg models, we refer to \cite{GGSharpe} and references therein. Alternative proposals to extend mirror symmetry to non-K\"ahler complex manifolds are due to Lau, Tseng, and Yau \cite{LTY}, Aldi and Heluani \cite{AldiHel}, Popovici \cite{Popovici} and Ward \cite{Ward}.

\subsection*{Outline of the paper}

Our construction is a round trip, from geometry to algebra and back, and has independent interest in the field of vertex algebras. We start by regarding a Bismut-flat pluriclosed hermitian metric as a solution of the \emph{Killing spinor equations} 
\begin{equation}\label{eq:killingEintro}
D^{S_+}_- \eta = 0, \qquad \slashed D^+ \eta = 0,
\end{equation}
on an exact Courant algebroid over $S^3 \times S^1$ with \v Severa class \eqref{eq:Severaclassintro} (see Proposition \ref{lemma:KillingevenE}). These are universal equations for a generalized metric, a divergence operator, and a spinor on a Courant algebroid, introduced in \cite{GF3,grt}. Their solutions encompass special holonomy metrics with solutions of the Hull--Strominger system, and are motivated by similar equations in supergravity \cite{CHS,CSW}. Identifying $S^3 \times S^1$ with the compact Lie group
\begin{equation}\label{eq:Kintro}
K = \SU(2) \times \U(1),
\end{equation}
and requiring that the solutions are left-invariant, we identify the Killing spinor equations with analogue equations on a Courant algebroid \emph{over a point}, that is, a quadratic Lie algebra $\g$ (see Proposition \ref{prop:geomalg}). 

Basic features of the Killing spinor equations \eqref{eq:killingEintro} on a quadratic Lie algebra $\g$ are studied in Section \ref{sec:KSeq}. In Section \ref{sec:SCVA}, we prove two main results concerning the construction of $N=2$ superconformal structures from Killing spinors in the universal superaffine vertex algebra associated to a quadratic Lie algebra (see Example \ref{exam:superafin}). For this, we assume that the rank of the generalized metric on $\g$ is even ($\dim V_+=2n$), and characterize the Killing spinor equations in terms of natural conditions for a decomposition 
$$
\g = V_+ \oplus V_- = l \oplus \overline{l} \oplus V_-
$$
and a \emph{divergence} $\varepsilon \in l \oplus \overline{l}$ (see Proposition \ref{prop:Killingeven}). Here, $l, \overline{l}$ are isotropic and $(l \oplus \overline{l})^\perp = V_-$. The resulting \emph{$F$-term} and \emph{$D$-term} equations
\begin{equation}\label{eq:FDeqintro}
[l,l] \subset l,\quad [\overline{l},\overline{l}] \subset \overline{l}, \qquad	 \frac12[\epsilon_j, \overline{\epsilon}_j] = \varepsilon_{\overline l} - \varepsilon_l
\end{equation}
make sense over an arbitrary closed field, and are reminiscent of similar conditions which appear in supersymmetric field theories in theoretical physics \cite{Dine}. Here, $\epsilon_j, \overline{\epsilon}_j$ form an isotropic basis of  $l \oplus \overline{l}$. In this generality, in Theorem \ref{th:N=2} we prove that a solution of \eqref{eq:FDeqintro} satisfying a technical condition \eqref{eq:orthogonalderived} 
induces an embbeding of the $N=2$ superconformal vertex algebra with central charge $c = 3 \dim l$ into the superaffine vertex algebra associated to $\g$ with level $0 \neq k \in \mathbb{C}$. This result generalizes a construction by Getzler for Manin triples \cite{Getzler} (see Remark \ref{rem:Getzler}). In our second main result (Theorem \ref{th:N=2dil}), we construct an analogue embedding with central charge 
\begin{equation}\label{eq:centralchargedilatonintro}
		c = 3\dim l+\frac{12}{k}\qf{\varepsilon}{\varepsilon}
	\end{equation}
for solutions of \eqref{eq:FDeqintro} such that $\varepsilon$ is \emph{holomorphic} (Definition \ref{def:holoiso}). We expect similar results for the cases $\dim V_+ \in \{ 4k, 7,8 \}$, in relation to other interesting vertex algebras in the literature \cite{HeluaniRev} (see Remark \ref{rem:G2} and Section \ref{sec:N=4}).

Back to the geometry, in Section \ref{sec:KillingHopf} we construct a family of left-invariant solutions of the Killing spinor equations on the group manifold $K$ (see Lemma \ref{lem:solutionsKsEq}). Applying Theorem \ref{th:N=2dil}, in Proposition \ref{prop:superaffineembed} we show that any element in the family induces an embedding of the $N=2$ superconformal vertex algebra with central charge $c = 6 + 6/\ell$ on the global sections of the (twisted) chiral de Rham complex. The construction of $(0,2)$ mirrors in Theorem \ref{thm:02} requires the application of topological T-duality \cite{BEM}, combined with Linshaw--Mathai's chiral version of the Cavalcanti--Gualtieri isomorphism \cite{CaGu,LinshawMathai}, which realizes the mirror involution \eqref{eq:mirrorinvintro}. The identification of the equivariant half-twisted models for $(0,2)$ mirrors is established in Theorem \ref{thm:halftwist}, following closely \cite{StructuresHeluani}. Finally, in Section \ref{sec:furtherex} we discuss two applications of Theorem \ref{th:N=2}. In Section \ref{sec:N=4} we construct an explicit two-parameter family of $N=4$ superconformal vertex algebras with central charge $c=6$. Our construction seems to be related to a $N=4$ superconformal algebra originally discovered by Sevrin, Troost and Van Proeyen \cite{STVan} (see also \cite{CHS}), but we have not been able to find a precise match. In Section \ref{sec:HS} we construct an infinite family of solutions of the Killing spinor equations for which Theorem \ref{th:N=2} applies, and such that $l \oplus \overline{l}$ is not a Manin triple (see Remark \ref{rem:Getzler}). Interestingly, this arise from invariant solutions of the Hull-Strominger system on a Lie group recently studied in \cite{GFGM}.

\begin{acknowledgements}
The authors wish to thank David Alfaya, Dan Fox, Reimundo He\-lua\-ni, Jock McOrist, Pavol \v Severa, Carlos Shahbazi, and Fridrich Valach for useful discussions. 
Special thanks to Jock McOrist for his generosity and patience during the long lectures on $(0,2)$ gauged linear sigma models that he gave at ICMAT.
The second author is grateful to the Department of Mathematics at the University of Toronto for the hospitality. 
\end{acknowledgements} 

\section{Killing spinors on quadratic Lie algebras}
\label{sec:KSeq}

\subsection{Generalized metrics, torsion, and divergence}\label{sec:gmetrics}

Let $\mathfrak{g}$ be a Lie algebra over the real numbers $\mathbb{R}$, endowed with an invariant inner product $\qf{\cdot}{\cdot}$ of arbitrary signature.
The aim of this section is to introduce a notion of \emph{metric with parallel spinor} on the quadratic Lie algebra $(\mathfrak{g},\qf{\cdot}{\cdot})$, which we will use for our applications to vertex algebras in Section \ref{sec:SCVA}.
To define this notion, we will view the quadratic Lie algebra as a Courant algebroid over a point, and then follow closely~\cite{GF3}.

We start by considering the relevant notions of metric and connection in our setup.

\begin{definition}\label{def:Gmetric}
	A \emph{generalized metric} on $\mathfrak{g}$ is an orthogonal decomposition
	$$
	\mathfrak{g} = V_+ \oplus V_-,
	$$
	so that the restriction of $\qf{\cdot}{\cdot}$ to $V_\pm$ is non-degenerate. We say that a generalized metric is \emph{Riemannian} if $\qf{\cdot}{\cdot}_{|V_+}$ is positive definite and $\qf{\cdot}{\cdot}_{|V_-}$ is negative definite.
\end{definition}

In the sequel we will only consider Riemannian metrics, but a similar analysis can be carried out in other signatures. Observe that a generalized metric is uniquely determined by a choice of positive-definite subspace $V_+ \subset \mathfrak{g}$. Given such a generalized metric, the associated orthogonal projections will be denoted
$$
\pi_+ \colon \mathfrak{g} \lto V_+, \qquad \pi_- \colon \mathfrak{g} \lto V_-.
$$
When there is no possibility of confusion, we will use the shorter notation
$$
a_\pm = \pi_\pm a, \quad \;
\textrm{ for } a \in \mathfrak{g}.
$$


\begin{definition}\label{def:Gconnection}
	A \emph{generalized connection} on $\mathfrak{g}$ is a linear map
	$$
	D \colon \mathfrak{g} \to \mathfrak{g}^* \otimes \mathfrak{g},
	$$
	satisfying the following compatibility condition for any $a,b,c \in \mathfrak{g}$:
	$$
	\qf{D_a b}{c} + \qf{b}{D_a c} = 0.
	$$
\end{definition}

In Definition \ref{def:Gconnection}, the notation $D_a b$ stands for the element in $\mathfrak{g}$ obtained from $a$ and $Db$ via the natural duality pairing $\mathfrak{g} \otimes \mathfrak{g}^* \to \mathbb{R}$.

The space of generalized connections on $\mathfrak{g}$ can be canonically identified with
$$
\mathfrak{g}^* \otimes \Lambda^2 \mathfrak{g},
$$
via the pairing $\qf{\cdot}{\cdot}$ on $\mathfrak{g}$. As in the case of differential geometry, given a generalized metric $V_+ \subset \mathfrak{g}$, we can define a natural space of $V_+$-compatible generalized connections.

\begin{definition}\label{def:Gconnectionmetric}
	Let $V_+ \subset \mathfrak{g}$ be a generalized metric. A generalized connection $D$ is \emph{$V_+$-compatible} if
	$$
	D(V_\pm) \subset \mathfrak{g}^* \otimes V_\pm.
	$$
	The space of $V_+$-compatible connections will be denoted by $\cD(V_+)$.
\end{definition}

As with generalized connections, there is a canonical identification
$$
\cD(V_+) = \mathfrak{g}^* \otimes (\Lambda^2 V_+ \oplus \Lambda^2 V_-),
$$
so any $D \in \cD(V_+)$ splits into four operators
\begin{equation}\label{eq:4operators}
	\begin{split}
		D^+_- & \in V_-^* \otimes \Lambda^2 V_+, \qquad D^-_+ \in V_+^* \otimes \Lambda^2 V_-,\\
		D^+_+ & \in V_+^* \otimes \Lambda^2 V_+, \qquad D^-_- \in V_-^* \otimes \Lambda^2 V_-.
	\end{split}
\end{equation}
In order to construct operators canonically associated to a generalized metric, we recall next the definition of torsion for generalized connections \cite{AXu,G3}.

\begin{definition}\label{def:torsion}
Given a generalized connection $D$, its \emph{torsion} $T_D \in \Lambda^3 \mathfrak{g}^*$ is defined by
	\begin{equation}\label{eq:torsionG}
		T_D(a,b,c) = \qf{D_{a}b - D_{b}a - [a,b]}{c} + \qf{D_{c}a}{b},
	\end{equation}
	for any $a,b,c \in \mathfrak{g}$.
\end{definition}

The fact that $T_D$ is totally skew-symmetric follows from the axioms of a quadratic Lie algebra and the compatibility between $D$ and $\qf{\cdot}{\cdot}$ (Definition \ref{def:Gconnection}).

Given a generalized metric $V_+ \subset \mathfrak{g}$, we will denote by 
$$
\cD^0(V_+) = \{ D \in \cD(V_+) \; | \; T_D = 0\}
$$
the space of torsion-free compatible connections. This is an affine space modelled on 
$$
\Sigma_+ \oplus \Sigma_-
$$
where
\begin{equation}\label{eq:sigma}\hspace{-1ex}
\Sigma_\pm = \left \{ \phi \in (V_\pm^*)^{\otimes 3}  : \phi(a,b,c) = - \phi(a,c,b), \ \sum_{\sigma(a,b,c)} \ \phi(a,b,c) = 0 , \;
	\textrm{ for } a,b,c \in V_\pm \right\}.
\end{equation}
The next result shows that the mixed-type operators $D^\pm_\mp$ in \eqref{eq:4operators} are uniquely determined by the generalized metric $V_+$ for any element $D \in \cD^0(V_+)$.

\begin{lemma}\label{lem:mixedfixed}
	Let $V_+ \subset \mathfrak{g}$ be a generalized metric. Then, for any $D \in \cD^0(V_+)$, we have
	\begin{equation}\label{eq:mixedcanonical}
		D_{a_-}b_+ = [a_-,b_+]_+, \qquad D_{a_+}b_- = [a_+,b_-]_-.
	\end{equation}
\end{lemma}
\begin{proof}
  This result follows because given $a,b,c \in \mathfrak{g}$ and  $D \in \cD(V_+)$, we have 
	\begin{equation*}\label{eq:TDmixed}
		T_D(a_-,b_+,c_+) = \qf{D_{a_-}b_+ - [a_-,b_+]}{c_+}.
        \qedhere
        \end{equation*}
\end{proof}

\begin{remark}
	The previous lemma can be regarded as a weak analogue of the Koszul formula in Riemannian geometry. In the context of generalized geometry, as introduced by Hitchin \cite{Hit1}, the canonical mixed-type operators \eqref{eq:mixedcanonical} 
relate to classical connections with skew-symmetric torsion (see Section \ref{ssec:KsCourant}).
\end{remark}

As we will see shortly, $\cD^0(V_+)$ is positive-dimensional provided that $\dim V_\pm \neq 1$, that is, in general, a generalized metric does not determine a torsion-free compatible connection uniquely \cite{GF3}. This situation can be compared with the construction of the Levi-Civita connection in Riemannian geometry. To proceed, we introduce `divergences' as a device to further constrain the degrees of freedom in $\cD^0(V_+)$.

\begin{definition}
	A \emph{divergence} on $\mathfrak{g}$ is an element $\alpha \in \mathfrak{g}^*$.
\end{definition}

Observe that any connection $D \in \cD(\g)$ determines a divergence by the formula
\begin{equation}\label{eq:alpha-D}
\alpha_D(a) = - \tr Da.
\end{equation}
In the sequel, we will identify divergences with elements $\varepsilon \in \g$ via the isomorphism $\g \cong \g^*$ provided by $\qf{\cdot}{\cdot}$.  Given $\varepsilon \in \mathfrak{g}$, we define elements
$$
\phi^{\varepsilon_+} \in V_+^* \otimes \Lambda^2 V_+, \qquad \phi^{\varepsilon_-} \in V_-^* \otimes \Lambda^2 V_-,
$$
by the formula
\begin{equation}\label{eq:chipm}
	\phi^{\varepsilon_\pm}_{a_\pm}b_\pm := \qf{a_\pm}{b_\pm} \varepsilon_\pm - \qf{\varepsilon_\pm}{b_\pm} a_\pm.
\end{equation}
Then we obtain a canonical splitting
\begin{equation}\label{eq:splittingyoungpure}
	\Sigma_\pm = \Sigma^0_\pm \oplus V_\pm
\end{equation}
via the direct sum decomposition 
\begin{equation}\label{eq:decompositionchipm}
	\phi^\pm = \phi_0^\pm + \phi^{\varepsilon_\pm} \in \Sigma_\pm,
\end{equation}
where 
\begin{equation}\label{eq:sigma0}
\phi_0^\pm \in	\Sigma^0_\pm = \Big{\{}\phi\in \Sigma_\pm : \sum_{i=1}^{\dim V_\pm} \phi(e_i^\pm,e_i^\pm,\cdot) = 0\Big{\}}.
\end{equation}
and
$$
\varepsilon_\pm := \frac{1}{\dim V_\pm -1}\sum_{i=1}^{n} \phi_{e_i^\pm} e^i_\pm.
$$
Here, $\{e_i^\pm\}$ is an orthogonal basis of $V_\pm$ with metric dual basis $\{e^i_\pm\}$, that is, $\la e_i^\pm, e^j_\pm \ra = \delta_{ij}$.

Given a generalized metric $V_+$ and a divergence $\varepsilon \in \g$, we denote by  
$$
\cD^0(V_+,\varepsilon) = \{D \in \cD^0(V_+) \; | \; \tr D = - \qf{\varepsilon}{\cdot}\}
$$
the space of torsion-free $V_+$-compatible generalized connections with fixed divergence. This is an affine space modelled on $\Sigma^0_+ \oplus \Sigma^0_-$.

\begin{lemma}\label{lem:Ddivergence}
Given a generalized metric $V_+$ and a divergence $\varepsilon \in \g$, the space $\cD^0(V_+,\varepsilon)$ is non-empty provided that $\dim V_\pm \neq 1$. Furthermore, the following formula defines an element $D \in \cD^0(V_+,\varepsilon)$:
\begin{equation}\label{eq:Ddiv}
\begin{split}\hspace*{-.8ex}
D_{a}b & = [a_-,b_+]_+ + [a_+,b_-]_- + \frac{1}{3} [a_+,b_+]_+ + \frac{1}{3} [a_-,b_-]_- + \frac{\phi^{\varepsilon_+}_{a_+}b_+}{\dim V_+ -1} + \frac{\phi^{\varepsilon_-}_{a_-}b_-}{\dim V_- -1} .
\end{split}
\end{equation}
\end{lemma}

\begin{proof}
	Consider $\tilde D \in \cD(V_+)$ defined by
	\begin{equation}\label{eq:Dpuretype}
		\tilde D_{a}b = [a_-,b_+]_+ + [a_+,b_-]_-.
	\end{equation}
	This is a compatible connection with $T_D \in \Lambda^3 V_+ \oplus \Lambda^3 V_-$, and therefore
	$$
	D^0 = \tilde D - \frac{1}{3}T_{\tilde D} \in \cD^0(V_+).
	$$
	Finally, it is not difficult to see that $D^0$ has zero divergence. The statement follows now from the identity
	\begin{equation*}
	\begin{split}
	D_{a}b & = D^0_{a}b + \frac{\phi^{\varepsilon_+}_{a_+}b_+}{\dim V_+ -1} + \frac{\phi^{\varepsilon_-}_{a_-}b_-}{\dim V_- -1}.
	\qedhere\end{split}
    \end{equation*}
\end{proof}

\begin{remark}\label{rem:divergence}
The $V_+^*$-valued endomorphism $\phi^{\varepsilon_+}$ in \eqref{eq:decompositionchipm} is reminiscent of the `1-form valued Weyl endomorphisms' in conformal geometry,
	which appear in the variation of the Levi-Civita connection upon a
	conformal change of the metric. 
	Thus geometrically, divergences keep track of the `conformal geometry' on the quadratic Lie algebra $\g$.
\end{remark}

To finish this section, we present a natural compatibility condition for pairs $(V_+,\varepsilon)$ which plays an important role in our applications in Section \ref{sec:SCVA}. This condition leads to desirable structural properties of the `generalized Ricci tensor' on a quadratic Lie algebra \cite{GF3,SV2}, but we will not need this more general piece of the theory for the present work.

\begin{definition}\label{def:isometry}
	Let $V_+$ be a generalized metric on $\g$. We say that $\varepsilon \in \g$ is an \emph{infinitesimal isometry} of $V_+$ if 
	\begin{equation}\label{eq:isometry}
		[\varepsilon,V_\pm] \subset V_\pm. 
	\end{equation}
	In this case, we will say that $(V_+,\varepsilon)$ is a \emph{compatible pair}.
\end{definition}

\subsection{The Killing spinor equations}\label{ssec:KSeq}

In order to introduce the Killing spinor equations on the quadratic Lie algebra $\g$, we define next a pair of Dirac-type operators canonically associated to a pair $(V_+,\varepsilon)$ as in Lemma \ref{lem:Ddivergence}.

Let $V_+$ be a Riemannian generalized metric on $\g$. In the sequel, we fix an orthogonal basis $\{e_i^\pm\}_{i=1}^n$ of $V_\pm$ with dual basis $\{e_\pm^i\}_{i=1}^n$, i.e., such that $\qf{e_\pm^i}{e_j^{\pm}} = \delta_{ij}$. It is not difficult to see that the following construction is independent of the choice of basis.

Let $Cl(V_+)$ and $Cl(V_-)$ denote the complex Clifford algebras of $V_+$ and $V_-$, respectively, defined by the relation
\begin{equation}\label{eq:Cliffordrel}
	v \cdot v = \qf{v}{v} 
\end{equation}
for $v \in V_\pm$. We fix irreducible representations $S_\pm$ of $Cl(V_\pm)$ (see, e.g., \cite{MicLaw}, and note the different convention for the relation \eqref{eq:Cliffordrel} in the Clifford algebra).

For any choice of metric connection $D \in \cD(V_+)$, we have associated \emph{spin connections}
\begin{equation}\label{eq:LCspinpure}
	D^{S_+}_+ \in V_+^* \otimes \End(S_+), \qquad D^{S_-}_- \in V_-^* \otimes \End(S_-).
\end{equation}

\begin{definition}\label{d:scalardirac} 
	Let $(V_+,\varepsilon)$ be a pair given by a generalized metric and a divergence. Given $D \in \cD^0(V_+,\varepsilon)$, define a pair of Dirac-type operators
	\begin{align*}
		\slashed D^+ \in \End(S_+), \qquad \slashed D^- \in \End(S_-),
	\end{align*}
	given explicitly by the following formula, for $\eta \in S_\pm$:
	$$
	\slashed D^\pm \eta = \sum_{i=1}^{n} e_\pm^i \cdot D^{S_\pm}_{e_i^\pm} \eta.
	$$
\end{definition}

\begin{lemma}\label{lem:dDpm}
	The Dirac operators $\slashed D^+$ and $\slashed D^-$ are independent of the choice of connection $D \in \cD^0(V_+,\varepsilon)$.
\end{lemma}

\begin{proof}
	Let $D \in \cD(V_+,\varepsilon)$. Consider a torsion-free $V_+$-compatible connection $D' = D + \phi \in \cD^0(V_+)$, with
	$$
	\phi = \phi_+ + \phi_- \in \Sigma_+ \oplus \Sigma_-.
	$$
	Using the decomposition in \eqref{eq:decompositionchipm}, there exists $e_\pm \in V_\pm$ such that $\phi_\pm = \phi_\pm^0 + \phi^{e_\pm}$. Then the divergence of $D'$, defined by~\eqref{eq:alpha-D}, is 
	$$
	\alpha_{D'} = \qf{\varepsilon}{\cdot} - (\dim V_+ - 1)\qf{e_+}{\cdot} - (\dim V_- - 1) \qf{e_-}{\cdot},
	$$
	and therefore $\alpha_{D'} = \qf{\varepsilon}{\cdot}$ implies $e_+ = e_- = 0$. Define $\sigma \in \Lambda^3 \g^*$ by
	$$
	\sigma(e_1,e_2,e_3) = \qf{\phi_{e_1}e_2}{e_3},
	$$
	and note that the total skew-symmetrization $c.p. (\sigma)$ vanishes, since $D'$ and $D$ are torsion-free. Decomposing $\sigma = \sigma_+ + \sigma_-$ in pure types, the result follows from
	\[
	\slashed D'^\pm = \slashed D^\pm - \tfrac{1}{2} c.p.(\sigma_\pm) = \slashed D^\pm.
	\qedhere
        \]
\end{proof}

We are ready to introduce the equations of our main interest. Observe that, from Lemma \ref{lem:Ddivergence} and Lemma \ref{lem:dDpm}, the operator $\slashed D^+$ (respectively $\slashed D^-$) depends only on $\varepsilon_+ = \pi_+ \varepsilon$ (respectively $\varepsilon_- = \pi_- \varepsilon$). Notice also that the canonical operators $D^\pm_\mp$ associated to a generalized metric $V_+ \subset \g$ (see Lemma \ref{lem:mixedfixed}) induce \emph{spin connections}
\begin{equation}\label{eq:LCspinmix}
	D^{S_+}_- \in V_-^* \otimes \End(S_+), \qquad D^{S_-}_+ \in V_+^* \otimes \End(S_-).
\end{equation}

\begin{definition}\label{def:killing}
	A triple $(V_+,\varepsilon_\pm,\eta)$, given by a generalized metric $V_+$, $\varepsilon_\pm \in V_\pm$, and a non-vanishing spinor $\eta \in S_\pm$, is a solution of the \emph{Killing spinor equations}, if
	\begin{equation}\label{eq:killing}
		\begin{split}
			D^{S_\pm}_\mp \eta &= 0,\\
			\slashed D^\pm \eta &= 0,
		\end{split}
	\end{equation}
	where $D^{S_\pm}_\mp$ and $\slashed D^\pm$ are defined via Lemma \ref{lem:mixedfixed} and Lemma \ref{lem:dDpm}.
\end{definition}

Motivated by the analogy with supergravity~\cite{CSW}, we will call the first and the second equations in \eqref{eq:killing} the \emph{gravitino equation} and the \emph{dilatino equation}, respectively. Unlike the dilatino equation, the gravitino equation only depends on the pair $(V_+,\eta)$ (by Lemma \ref{lem:mixedfixed}). Our next goal is to give a more amenable characterization of these equations.

\begin{lemma}\label{lem:gravitino}
	A pair $(V_+,\eta)$, given by a generalized metric $V_+$ and a non-vanishing spinor $\eta \in S_\pm$, is a solution of the gravitino equation
	\begin{equation}\label{eq:gravitino}
		D^{S_\pm}_\mp \eta = 0,
	\end{equation}
	if and only if 
	\begin{equation}\label{eq:isotropy}
		D^\pm_\mp \in V_\mp^* \otimes \Lie G_\eta,
	\end{equation}
	where $G_\eta \subset Spin(V_\pm)$ is the stabilizer of $\eta$. More explicitly, \eqref{eq:gravitino} is equivalent to
	$$
	\sum_{i,j} \qf{[\pi_\mp,e_i^\pm]}{e_\pm^j}e^j_\pm e_i^\pm \cdot \eta = 0.
	$$
\end{lemma}

\begin{proof}
	The first part of the statement follows simply from the identity~\cite{MicLaw}
	$$
	\Lie G_\eta = \{ B \in \Lambda^2 V_\pm \; | \; B \cdot \eta = 0\}.
	$$
	As for the second part, an endomorphism $A\in \End(V_\pm)$ satisfies
	$$
	A = \sum_{i,j=1}^{\dim V_\pm} \qf{A e_i^\pm}{e_\pm^j}\qf{e^i_\pm}{\cdot} \otimes e_j^\pm.
	$$
	Since $\qf{e^i_\pm}{\cdot} \otimes e_j^\pm - \qf{e^j_\pm}{\cdot} \otimes
	e_i^\pm \in \mathfrak{so}(V_\pm)$ embeds as $\frac{1}{2}e^j_\pm e_i^\pm$ in the Clifford
	algebra $Cl(V_\pm)$, an endomorphism $A\in \mathfrak{so}(V_\pm)$ corresponds to
	\begin{equation}\label{eq:A-as-element-in-Cl}
		A=\frac{1}{4}\sum_{i,j=1}^{\dim V_\pm} \qf{A e_i^\pm}{e_\pm^j} e^j_\pm e_i^\pm \in Cl(V_\pm).
	\end{equation}
	Then, given a spinor $\eta \in S_\pm$, we have
	\begin{align*}
		D^{S_\pm}_\mp \eta &{} =  \frac{1}{4}\sum_{i,j} \qf{[\pi_\mp,e_i^\pm]}{e_\pm^j}e^j_\pm e_i^\pm \cdot \eta.
        \qedhere
	\end{align*}
\end{proof}

\begin{remark}\label{rem:G2}
The case $\dim V_\pm = 2n_\pm$ and $G_\eta = \SU(n_\pm)$ is studied in detail in Section \ref{sec:Killingeven}. When $\dim V_\pm = 7$, it is interesting to consider the equations \eqref{eq:killing} for a non-vanishing real spinor, as in this case $G_\eta = G_2$ in Lemma \ref{lem:gravitino} (see \cite[Prop. 10.2]{MicLaw}). Similarly, when $\dim V_\pm = 8$ and we take $S_\pm = \mathbb{R}^8$, one has $G_\eta = \operatorname{Spin}(7)$ (see \cite[Prop. 10.4]{MicLaw}).
\end{remark}

Next we obtain a similar characterization of the dilatino equation in \eqref{eq:killing}.

\begin{lemma}\label{lem:dilatino}
	A triple $(V_+,\varepsilon_\pm,\eta)$ as in Definition \ref{def:killing} is a solution of the dilatino equation
	\begin{equation}\label{eq:dilatino}
		\slashed D^\pm \eta = 0,
	\end{equation}
	if and only if 
	$$
	\frac{1}{6}\sum_{i,j,k} \qf{[e_k^\pm,e_i^\pm]}{e_\pm^j}e^k_\pm e^j_\pm e_i^\pm \cdot \eta = \varepsilon_\pm \cdot \eta.
	$$
\end{lemma}

\begin{proof}
	Consider the connection $D \in \cD(V_+,\varepsilon)$ in \eqref{eq:Ddiv}. Arguing as in the proof of Lemma \ref{lem:gravitino}, for any $a_\pm \in V_\pm$ we obtain
	\begin{align*}
		D^{S_\pm}_{a_\pm} \eta &{} =  \frac{1}{12}\sum_{i,j} \qf{[a_\pm,e_i^\pm]}{e_\pm^j} e^j_\pm e_i^\pm \cdot \eta\\
		& + \frac{1}{4(\dim V_\pm -1)} \sum_{i,j} \Big{(}\qf{a_\pm}{e_i^\pm}\qf{\varepsilon_\pm}{e^j_\pm} - \qf{\varepsilon_+}{e_i^\pm} \qf{a_\pm}{e^j_\pm} \Big{)} e^j_\pm e_i^\pm \cdot \eta\\
		&{} =  \frac{1}{12}\sum_{i,j} \qf{[a_\pm,e_i^\pm]}{e_\pm^j} e^j_\pm e_i^\pm \cdot \eta +  \frac{1}{4(\dim V_\pm -1)} (\varepsilon_\pm a_\pm - a_\pm \varepsilon_\pm) \cdot \eta.
	\end{align*}
	Hence, setting 
	$$
	C := \frac{1}{12}\sum_{i,j,k} \qf{[e_k^\pm,e_i^\pm]}{e_\pm^j} e^k_\pm e^j_\pm e_i^\pm \cdot \eta,
	$$
	we have
	\begin{align*}
		\slashed D^\pm \eta & = C + \frac{1}{4(\dim V_\pm -1)} \sum_{k} e^k_\pm  (\varepsilon_\pm e_k^\pm - e_k^\pm \varepsilon_\pm) \cdot \eta \\
		& = C + \frac{1}{4(\dim V_\pm -1)} \sum_{k} \Big{(}2\qf{\varepsilon_\pm}{e^k_\pm} e_k^\pm - 2 \varepsilon_\pm\Big{)} \cdot \eta \\
		& = C - \frac{1}{2}\varepsilon_\pm \eta.
        \qedhere
	\end{align*}
	
\end{proof}

\subsection{Killing spinors in even dimensions}\label{sec:Killingeven}

In this section we analyse further the Killing spinor equations \eqref{eq:killing} under the assumption that $\dim V_\pm$ is even. In the sequel, we fix a generalized metric $\g = V_+ \oplus V_-$ on $\g$ and orientations on $V_\pm$. Assuming that $\dim V_\pm = 2n_\pm$, for $n_\pm \in \mathbb{N}$, $S_\pm$ split as irreducible $\operatorname{Spin}(2n_\pm)$-representations 
\begin{equation}\label{eq:chiralspinor}
	S_\pm = S_\pm^+ \oplus S_\pm^-,
\end{equation}
corresponding to the $\pm 1$-eigenspaces for the action of the complex volume
$$
\nu_\CC^\pm = i^{n_\pm}e_1^\pm \ldots e_{2n_\pm}^\pm
$$ 
for a choice of oriented basis $\{e_j^\pm\} \subset V_\pm$ such that $\qf{e_i^\pm}{e_j^\pm} = \pm \delta_{ij}$. A pure spinor $\eta \in S_\pm$ must have definite chirality, that is, either $\eta \in S_\pm^+$ or $\eta \in S_\pm^-$ (see \cite[Lemma 9.6, p. 337]{MicLaw}). For $\eta \in S_\pm^+$ pure, the isotropy group of $\eta$ in $\operatorname{Spin}(2n_\pm)$ is given by (see \cite[Lemma 9.15, p. 343]{MicLaw})
$$
G_\eta = \SU(n_\pm).
$$
In particular, $\eta$ determines an almost complex structure $I$ on $V_\pm$ compatible with $\qf{\cdot}{\cdot}_{|V_\pm}$ and the given orientation, such that the decomposition 
$$
V_\pm^{\C} := V_\pm \otimes_{\R} \C = V_\pm^{1,0} \oplus V_\pm^{0,1}
$$
in $\pm i$-eigenspaces is determined by
$$
V_+^{1,0} = \{a_+ \in V_+ \otimes \CC \; | \; a_+ \cdot \eta = 0\}, \qquad V_-^{0,1} = \{a_- \in V_- \otimes \CC \; | \; a_- \cdot \eta = 0\},
$$
where $\C$ is the field of complex numbers. 
Our aim in this section is to characterize \eqref{eq:killing} in terms of this $\SU(n_\pm)$-structure.

We fix a pure spinor $\eta \in S_+^+$ (the case $\eta \in S_-^+$ is analogue). Using the almost complex structure $I$ on $V_+$ determined by $\eta$, we have a model 
$$
S_+ = \Lambda^* V_+^{0,1}
$$
with Clifford action
$$
a_+ \cdot \sigma = \sqrt{2} i_{\qf{a_+^{1,0}}{\cdot}} \sigma + \sqrt{2}
a_+^{0,1} \wedge \sigma.
$$
Here, $\qf{\cdot}{\cdot}$ denotes the $\CC$-linear extension of the pairing to $V_+^{\C}$, which is a symmetric tensor of type $(1,1)$. With this identification, the decomposition \eqref{eq:chiralspinor} corresponds to 
\begin{equation}\label{eq:chiralspinorexp}
\Lambda^* V_+^{0,1} = \Lambda^{even} V_+^{0,1} \oplus \Lambda^{odd} V_+^{0,1},
\end{equation}
and, by \cite[Proposition 9.7, p. 337]{MicLaw}, $\eta = \lambda \in \CC^*$ in this model. For the following calculations, it will be convenient to fix an oriented orthonormal basis for $\qf{\cdot}{\cdot}_{|V_+}$ of the form
$$
\{e_1^+, Ie_1^+, \ldots , e_{n_+}^+, I e_{n_+}^+\},
$$
with associated basis
$$
\{\epsilon_j^+ \}_{j=1}^{n_+} \subset V_+^{1,0},\qquad \{ \overline{\epsilon}_j^+ \}_{j=1}^{n_+}  \subset \overline{V_+^{1,0}} = V_+^{0,1},
$$
defined by
$$
\epsilon_j^+ = \tfrac{1}{\sqrt{2}}(e_j^+ - i Ie_j^+), \qquad \overline{\epsilon}_j^+ = \overline{\epsilon_j^+} = \tfrac{1}{\sqrt{2}}(e_j^+ + i Ie_j^+).
$$
Notice that the $\C$-linear extension of $\qf{\cdot}{\cdot}_{|V_+}$ satisfies
\begin{equation}\label{eq:isotropybasis}
\qf{\epsilon_j^+}{\epsilon_k^+} = 0,  \quad \qf{\epsilon_j^+}{\overline{\epsilon}_k^+} = \delta_{jk}, \quad \qf{\overline{\epsilon}_j^+}{\overline{\epsilon}_k^+} = 0.
\end{equation}
With the previous notation, we have the following.

\begin{lemma}\label{lem:gravitinoeven}
	A pair $(V_+,\eta)$, with $\dim V_\pm = 2n_\pm$ and $\eta \in S_\pm^+$ pure, is a solution of the gravitino equation \eqref{eq:gravitino} if and only if the following conditions are satisfied
	\begin{equation}\label{eq:gravitinoeven}
		1) \; [V_\pm^{0,1},V_\pm^{0,1}] \subset V_\pm^\C, \qquad \qquad  2) \; \sum_{j=1}^{n_\pm} [\epsilon_j^\pm, \overline{\epsilon}_j^\pm] \in V_\pm^\C.
	\end{equation}
\end{lemma}

\begin{proof}
Assume first $\eta \in S_+^+$. Given $a_- \in V_-$, define
	$$
	\tau \in \Lambda^2 V_+
	$$
	by the formula 
	$$
	\tau(b_+,c_+) = \qf{\left[a_-,b_+\right]}{c_+},
	$$
	for $b_+,c_+ \in V_+$. Here we identify $V_+ \cong V_+^*$ with the isomorphism given by the induced metric on $V_+$. Then, by Lemma \ref{lem:gravitino}, there exists $\lambda \in \CC^*$ such that the gravitino equation is equivalent to 
	$$
	\tau \cdot \eta = \tau \cdot \lambda = 0
	$$
	for all $a_- \in V_-$. Decompose $\tau$ as
	$$
	\tau = \tau^{2,0}+\tau^{1,1}+\tau^{0,2},
	$$
	where $\tau^{0,2}=\overline{\tau^{2,0}}$. Using the identities
	\begin{align*}
		\epsilon_j \cdot 1 &= 0,\\
		\overline{\epsilon}_j \cdot 1 & = \sqrt{2} \overline{\epsilon}_j,\\
		\epsilon_j \overline{\epsilon}_k \cdot 1 & = 2 \delta_{jk},
	\end{align*}
	we obtain
	\begin{align*}
		\tau \cdot \eta = 2 \lambda \tau^{0,2} + 2 \lambda \sum_j \tau^{1,1}(\epsilon_j , \overline{\epsilon}_j).
	\end{align*}
	Thus $\tau \cdot \eta = 0$ holds if and only if
	\begin{align*}
		0 & = \tau(b_+^{0,1},c_+^{0,1}) = \qf{[a_-,b_+^{0,1}]}{c_+^{0,1}} = \qf{a_-}{[b_+^{0,1},c_+^{0,1}]},\\
		0 & = \sum_j \qf{a_-}{[\epsilon_j ,\overline{\epsilon}_j]},
	\end{align*}
	for all $b_+,c_+ \in V_+^\C$. The statement follows from the fact that $a_-$ can be chosen arbitrarily.
	
	Similarly, $\eta \in S^+_-$ pure is a solution of the gravitino equation if and only if
$$
1) \; [V_-^{1,0},V_-^{1,0}] \subset V_-^\C, \qquad \qquad  2) \; \sum_{j=1}^{n_-} [\overline{\epsilon}_j^-, \epsilon_j^-] \in V_-^\C,
$$
which is equivalent to \eqref{eq:gravitinoeven} by conjugation on the two equations.
\end{proof}

We state next our characterization of the dilatino equation \eqref{eq:dilatino} in the present setup.

\begin{lemma}\label{lem:dilatinoeven}
A triple $(V_+,\varepsilon_\pm,\eta)$, with $\dim V_\pm = 2n_\pm$, $\eta \in S_\pm^+$ pure, and $\varepsilon_\pm \in V_\pm$, is a solution of the dilatino equation \eqref{eq:dilatino} if and only if the following two conditions hold
\begin{equation}\label{eq:dilatinoeven}
1) \; [V_\pm^{0,1},V_\pm^{0,1}]_+ \subset V_\pm^{0,1}, \qquad \qquad 2) \; \frac{i}{2}\sum_{j=1}^{n_\pm} [\epsilon_j^\pm, \overline{\epsilon}_j^\pm]_+ = \mp  I\varepsilon_\pm.
\end{equation}
\end{lemma}

\begin{proof}
Assume first $\eta \in S_+^+$. Define
	$$
	H \in \Lambda^3 V_+
	$$
	by the formula 
	$$
	H(a_+,b_+,c_+) = \qf{\left[a_+,b_+\right]}{c_+}
	$$
	for $a_+,b_+,c_+ \in V_+$. Then, by Lemma \ref{lem:dilatino}, the dilatino equation is equivalent to 
	$$
	\frac{1}{6}H \cdot \lambda = \varepsilon_+ \cdot \lambda.
	$$
	Decompose $H$ as
	$$
	H = H^{3,0} + H^{2,1} + H^{1,2} + H^{0,3},
	$$
	where $H^{3,0}=\overline{H^{0,3}}$ and $H^{2,1}=\overline{H^{1,2}}$. Using the identities 
	\begin{align*}
		\overline{\epsilon}_j^+ \epsilon_k^+ \overline{\epsilon}_l^+ \cdot 1 & = 2
		\sqrt{2}\delta_{kl} \overline{\epsilon}_j^+,\\
		\epsilon_j^+ \overline{\epsilon}_k^+ \overline{\epsilon}_l^+ \cdot 1 & =2
		\sqrt{2}(\delta_{jk} \overline{\epsilon}_l^+ - \delta_{jl}\overline{\epsilon}_k^+),
	\end{align*}
	we obtain
	\begin{align*}
		(H - 6 \varepsilon_+) \cdot \eta = \lambda 2 \sqrt{2} \(H^{0,3} + \tfrac{3}{2} \sum_{j} H^{1,2}(\epsilon_j^+, \overline{\epsilon}_j^+) - 3 \varepsilon_+^{0,1}\).
	\end{align*}
	Thus $(H - 6 \varepsilon_+) \cdot \eta = 0$ holds if and only if
	\begin{align*}
		0 & = H(a_+^{0,1},b_+^{0,1},c_+^{0,1}) = \qf{[a_+^{0,1},b_+^{0,1}]}{c_+^{0,1}}, \\
		0 & = \sum_j \qf{a_+^{1,0}}{[\epsilon_j^+ ,\overline{\epsilon}_j^+]} - 2\qf{a_+^{1,0}}{\varepsilon_+^{0,1}},
	\end{align*}
	for all $a_+,b_+,c_+ \in V_+^\C$. Using the orthogonal decomposition
$$
\g \otimes_\RR \C = V_+^\C \oplus V_-^\C,
$$
and the fact that $\qf{\cdot}{\cdot}_{|V_+}$ is of type $(1,1)$, we see that the dilatino equation is equivalent to
\begin{equation*}
[V_+^{0,1},V_+^{0,1}]_+ \subset V_+^{0,1}, \qquad \qquad  \frac{1}{2}\sum_{j=1}^{n_+} [\epsilon_j^+, \overline{\epsilon}_j^+]^{0,1}_+ = \varepsilon_+^{0,1}.
	\end{equation*}
Using now
$$
\overline{i [\epsilon_j^+, \overline{\epsilon}_j^+]} = - i [\overline{\epsilon}_j^+,\epsilon_j^+] = i [\epsilon_j^+, \overline{\epsilon}_j^+],
$$
it follows that
$$
- I \varepsilon_+ = i \varepsilon_+^{0,1} + \overline{i\varepsilon_+^{0,1}} = \frac{i}{2}\sum_{j=1}^{n_+}\left([\epsilon_j^+, \overline{\epsilon}_j^+]^{0,1}_+ + [\epsilon_j^+, \overline{\epsilon}_j^+]^{1,0}_+\right) = \frac{i}{2}\sum_{j=1}^{n_+} [\epsilon_j^+, \overline{\epsilon}_j^+]_+.
$$
Similarly, a pure spinor $\eta \in S^+_-$ is a solution of the dilatino equation if and only if
$$
[V_-^{1,0},V_-^{1,0}]_- \subset V_-^{1,0}, \qquad \qquad  2) \; \frac{1}{2}\sum_{j=1}^{n_-} [\overline{\epsilon}_j^-, \epsilon_j^-]_-^{1,0} = - \varepsilon_-^{1,0},
$$
which is equivalent to \eqref{eq:dilatinoeven} by conjugation on the two equations.
\end{proof}

With the previous results at hand, we give a characterization of the Killing spinor equations in the present setup.

\begin{proposition}\label{prop:Killingeven}
Let $(V_+,\varepsilon_\pm,\eta)$ be a triple, with $\dim V_\pm = 2n_\pm$, $\eta \in S_\pm^+$ pure, and $\varepsilon_\pm \in V_\pm$. Then, $(V_+,\varepsilon_\pm,\eta)$ is a solution of the Killing spinor equations \eqref{eq:killing} if and only if the following two conditions hold:
\begin{equation}\label{eq:FtermDterm}
1) \; [V_\pm^{0,1},V_\pm^{0,1}] \subset V_\pm^{0,1}, \qquad \qquad 2) \; \frac{i}{2} \sum_{j=1}^{n_\pm} [\epsilon_j^\pm, \overline{\epsilon}_j^\pm] = \mp I\varepsilon_\pm.
\end{equation}
\end{proposition}

\begin{proof}
The proof is straightforward from Lemma \ref{lem:gravitinoeven} and Lemma \ref{lem:dilatinoeven}.
\end{proof}

\begin{remark}\label{rem:flipI}
Let $(V_+,\varepsilon_\pm,\eta)$ be a solution of the Killing spinor equations as in Proposition \ref{prop:Killingeven}, and denote by $\eta' \in S_\pm^-$ a pure spinor in the line corresponding to $-I$. Then, it follows from \eqref{eq:FtermDterm} that $(V_+,\varepsilon_\pm,\eta')$ is also a solution of the Killing spinor equations.
\end{remark}

\subsection{The F-term and D-term conditions}\label{sec:FDterms}

Proposition \ref{prop:Killingeven} suggests a weaker version of the equations \eqref{eq:killing}, which forgets about the real structure on the complex quadratic Lie algebra $\g \otimes_\R \C$, and takes as fundamental object the isotropic subspace
$$
V_+^{1,0} \subset \g \otimes_\RR \CC.
$$ 
This alternative point of view is more flexible, as it allows us to work over an arbitrary field of characteristic zero. This is the approach that we take for the applications to vertex algebras in Section \ref{sec:SCVA}.

Let $\g^c$ be a complex quadratic Lie algebra with pairing $\qf{\cdot}{\cdot}$. We consider the following space of non-degenerate isotropic subspaces:
$$
\cL = \{l \oplus \overline{l} \subset \g^c \; | \; l,\overline{l} \textrm{ are isotropic and } \qf{\cdot}{\cdot}_{| l \oplus \overline{l} } \textrm{ is non-degenerate} \}.
$$
By definition, given an element $l \oplus \overline{l} \in \cL$, we have a canonical identification
$$
l^* \cong \overline{l}.
$$
We denote the orthogonal complement of $l \oplus \overline{l}$ by $V_-^\CC = (l \oplus \overline{l})^\perp$. We will use the following notation for the orthogonal projections:
$$
\pi_+ \colon \g^c \to l \oplus \overline{l}, \qquad \pi_{l} \colon \g^c \to l, \qquad \pi_{\overline{l}} \colon \g^c \to \overline{l}, \qquad \pi_- \colon \g^c \to V_-^\CC,
$$
which exist by assumption. Given $a \in \g^c$, when there is no possible confussion we will use the simplified notation
\begin{equation}\label{eq:projections}
a_+ = \pi_+ a, \qquad  a_{l} = \pi_{l}a, \qquad a_{\overline{l}} = \pi_{\overline{l}}a, \qquad a_- = \pi_-a.
\end{equation}

To formulate the following definition, given $l \oplus \overline{l}  \in \cL$, we fix dual isotropic bases $\epsilon_j, \overline{\epsilon}_j$ of $l\oplus\overline{l}$ as in~\eqref{eq:isotropybasis}. In other words, $\{\epsilon_j\}_{j=1}^{\dim l}$ is a basis of $l$, $\{\overline{\epsilon}_j\}_{j=1}^{\dim l}$ is a basis of $\overline{l}$, and 
\begin{equation}\label{eq:isotropybasis-weaker}
\qf{\epsilon_j}{\epsilon_k} = 0,  \quad \qf{\epsilon_j}{\overline{\epsilon}_k} = \delta_{jk}, \quad \qf{\overline{\epsilon}_j}{\overline{\epsilon}_k} = 0.
\end{equation}

\begin{definition}\label{def:FtermDterm}
We will say that an element $l \oplus \overline{l}  \in \cL$ satisfies 
	\begin{enumerate}
	\item the \emph{F-term gravitino equation} if
	\begin{equation}\label{eq:Ftermgrav}
	[l,l] \subset l \oplus \overline{l}, \qquad [\overline{l},\overline{l}] \subset l \oplus \overline{l}.
	\end{equation}
		
	\item the \emph{F-term dilatino equation} if
	\begin{equation}\label{eq:Ftermdil}
	[l,l]_+ \subset l, \qquad [\overline{l},\overline{l}]_+ \subset \overline{l}.
	\end{equation}
	
	\item the \emph{F-term equation} if
	\begin{equation}\label{eq:Fterm}
	[l,l] \subset l, \qquad [\overline{l},\overline{l}] \subset \overline{l}.
	\end{equation}
	
	\item the \emph{D-term gravitino equation}  if, for any basis $\epsilon_j, \overline{\epsilon}_j$ of $l \oplus \overline{l}$ as in \eqref{eq:isotropybasis-weaker},
	\begin{equation}\label{eq:Dtermgrav}
	\sum_{j=1}^{\dim l} [\epsilon_j, \overline{\epsilon}_j]  \in l \oplus \overline{l}.
	\end{equation}

	\item the \emph{D-term dilatino equation} with divergence $\varepsilon \in l \oplus \overline{l}$ if, for any basis $\epsilon_j, \overline{\epsilon}_j$ of $l \oplus \overline{l}$ as in \eqref{eq:isotropybasis-weaker},
	\begin{equation}\label{eq:Dtermdil}
	\frac{1}{2} \sum_{j=1}^{\dim l} [\epsilon_j, \overline{\epsilon}_j]_+  = \varepsilon_{\overline l} - \varepsilon_l.
	\end{equation}
	
	\item the \emph{D-term equation} with divergence $\varepsilon \in l \oplus \overline{l}$ if, for any basis $\epsilon_j, \overline{\epsilon}_j$ of $l \oplus \overline{l}$ as in \eqref{eq:isotropybasis-weaker},
	\begin{equation}\label{eq:Dterm}
	\frac{1}{2}\sum_{j=1}^{\dim l} [\epsilon_j, \overline{\epsilon}_j] = \varepsilon_{\overline l} - \varepsilon_l.
	\end{equation}
	\end{enumerate}
\end{definition}

It is easy to see that \eqref{eq:Dtermgrav}, \eqref{eq:Dtermdil} and \eqref{eq:Dterm} are independent of the choice of basis. The names of the equations \eqref{eq:Fterm} and \eqref{eq:Dterm} are borrowed from similar conditions appearing in supersymmetric field theories in theoretical physics. In this setup, the structure of the effective potential typically distinguishes two types of conditions for a minimal energy configuration \cite{Dine}, and \eqref{eq:Fterm} and \eqref{eq:Dterm} are strongly reminiscent of those.

\begin{remark}\label{rem:involutivity}
Note that the $F$-term gravitino equation \eqref{eq:Ftermgrav} is equivalent to 
\begin{equation*}
\left[a_-,e_j\right]_{\overline{l}} = 0, \qquad \left[a_-,e^j\right]_{l} = 0,
	\end{equation*}
for all $j \in \left\lbrace1,\ldots,n\right\rbrace$, for any given $a_- \in V_-^\CC := (l \oplus \overline{l})^\perp$.
\end{remark}

In the case that $\g^c$ is the complexification of a real quadratic Lie algebra $\g$, any solution of the gravitino equation \eqref{eq:gravitino} determines a solution of the F-term and D-term gravitino equations \eqref{eq:Ftermgrav} and \eqref{eq:Dtermgrav}, given by
\begin{equation}\label{eq:lreal}
l \oplus \overline{l} = V_+^{1,0} \oplus V_+^{0,1} \subset \g \otimes_\R \C.
\end{equation}
This is a straightforward consequence of Proposition \ref{prop:Killingeven}. The following provides a converse for this result.

\begin{lemma}\label{lem:gravDFterm}
Let $\g$ be a real quadratic Lie algebra and consider $\g^c = \g \otimes_\R \C$. Denote by $\tau \colon \g^c \to \g^c$ the involution induced by conjugation.  Then any solution $(V_+,\eta)$ of the gravitino equation \eqref{eq:gravitino}, with $\dim V_+ = 2n_+$ and $\eta \in S_+^+$ pure, determines a solution of the F-term and D-term gravitino equations \eqref{eq:Ftermgrav} and \eqref{eq:Dtermgrav}, given by \eqref{eq:lreal}, such that
\begin{equation}\label{eq:realcompatiblegrav}
1) \; \tau(l) = \overline{l}, \qquad 2) \; \qf{a}{\tau(a)} > 0 \textrm{ for any } 0 \neq a \in l.
\end{equation}
Conversely, any solution of \eqref{eq:Ftermgrav} and \eqref{eq:Dtermgrav} satisfying \eqref{eq:realcompatiblegrav} determines a solution $(V_+,\eta)$ of the gravitino equation, with $\dim V_+ = 2n_+$ and $\eta \in S_+^+$ pure, uniquely up to rescaling of the spinor.
\end{lemma}

\begin{proof}
Given a solution $l \oplus \overline{l} \in \cL$ of \eqref{eq:Ftermgrav} and \eqref{eq:Dtermgrav} that satisfies \eqref{eq:realcompatiblegrav}, define $V_+^\C := l \oplus \overline{l}$. Then condition $1)$ in \eqref{eq:realcompatiblegrav} implies that $V_+^\C$ is the complexification of a subspace $V_+ \subset \g$, while condition $2)$ implies that $\qf{\cdot}{\cdot}_{|V_+} >0$. Using \eqref{eq:realcompatiblegrav} again, we see that $l$ induces an almost complex structure $I \colon V_+ \to V_+$ compatible with $\qf{\cdot}{\cdot}_{|V_+}$. Therefore, by \cite[Prop. 9.7, p. 337]{MicLaw}, $I$ determines a spinor line $\langle \eta \rangle \subset S_+^+$. Applying now Lemma \ref{lem:gravitinoeven}, it follows that $(V_+,\lambda \eta)$ is a solution of the gravitino equation \eqref{eq:gravitino}, for any $\lambda \in \C^*$, satisfying the hypotheses in the statement of the lemma. The converse follows from Proposition \ref{prop:Killingeven}.
\end{proof}

Under the above assumption $\g^c = \g \otimes_\R \C$, we have an analogue relation between the dilatino equation \eqref{eq:dilatino}, and the F-term and D-term dilatino equations \eqref{eq:Ftermdil} and \eqref{eq:Dtermdil}. Similarly, any solution of the Killing spinor equations determines a solution of the F-term and D-term equations, \eqref{eq:Fterm} and \eqref{eq:Dterm}, given by
\begin{equation}\label{eq:lrealdiv}
l \oplus \overline{l} = V_+^{1,0} \oplus V_+^{0,1} \subset \g \otimes_\R \C, \qquad \varepsilon = \varepsilon_+
\end{equation}
This follows by direct application of Proposition \ref{prop:Killingeven}. We give next the converse result.

\begin{proposition}\label{prop:FDreal}
Any solution $(V_+,\varepsilon_+,\eta)$ of the Killing spinor equations \eqref{eq:killing} on $\g$, with $\dim V_+ = 2n_+$ and $\eta \in S_+^+$ pure, determines a solution of the F-term and D-term equations on $\g \otimes_\R \C$, \eqref{eq:Fterm} and \eqref{eq:Dterm}, given by \eqref{eq:lrealdiv}, such that
\begin{equation}\label{eq:realcompatible}
1) \; \tau(l) = \overline{l}, \qquad 2) \; \qf{a}{\tau(a)} > 0 \textrm{ for any } 0 \neq a \in l, \qquad 3) \; \tau(\varepsilon) = \varepsilon .
\end{equation}
Conversely, any solution of the F-term and D-term equations satisfying \eqref{eq:realcompatible} determines a solution $(V_+,\varepsilon,\eta)$ of the Killing spinor equations, with $\dim V_+ = 2n_+$ and $\eta \in S_+^+$ pure, uniquely up to rescaling of the spinor.
\end{proposition}

\begin{proof}
Given a solution $l \oplus \overline{l} \in \cL$ of \eqref{eq:Fterm} and \eqref{eq:Dterm} that satisfies \eqref{eq:realcompatible}, define $V_+^\C := l \oplus \overline{l}$. Arguing as in the proof of Lemma \ref{lem:gravDFterm}, $(V_+,\varepsilon,\lambda \eta)$ is a solution of the Killing spinor equations \eqref{eq:killing}, for any $\lambda \in \C^*$, satisfying the hypothesis in the statement. Notice that condition $3)$ in \eqref{eq:realcompatible} implies that $\varepsilon$ is real, that is, $\varepsilon \in V_+$. The converse follows directly from Proposition \ref{prop:Killingeven} and Lemma \ref{lem:gravDFterm}. 
\end{proof}

Let $(l \oplus \overline{l},\varepsilon)$ be a pair as in Definition \ref{def:FtermDterm}. To finish this section we introduce a notion of compatibility for such a pair, which provides a `holomorphic counterpart' to the notion of isometry in Definition \ref{def:isometry} (cf. Lemma \ref{lem:examhol}).

\begin{definition}\label{def:holoiso}
Let $\g^c$ be a complex quadratic Lie algebra and $l \oplus \overline{l} \in \cL$. We say that $\varepsilon \in \g^c$ is
\begin{enumerate}

\item an \emph{infinitesimal isometry} if 
	\begin{equation}\label{eq:isometryabs}
		[\varepsilon,l \oplus \overline{l}] \subset l \oplus \overline{l}.
	\end{equation}
	
\item \emph{holomorphic} if 
	\begin{equation}\label{eq:holo}
		[\varepsilon,l] \subset l, \qquad [\varepsilon,\overline l] \subset \overline l. 
	\end{equation}

\end{enumerate}

\end{definition}

Note that, trivially, condition \eqref{eq:holo} implies \eqref{eq:isometryabs}. Observe also that, in the situation considered in Proposition \ref{prop:FDreal}, equation \eqref{eq:isometryabs} implies that $\varepsilon$ is an isometry in the sense of Definition \ref{def:isometry}. More abstractly, in the next result we study a salient feature of holomorphicity for the derived Lie subalgebras 
$$
[l,l] \subset \g^c,\qquad  [\overline{l},\overline{l}] \subset \g^c.
$$
The necessary condition \eqref{eq:orthogonalderived} below will play a key role in our first main Theorem \ref{th:N=2}.  The proof is a direct consequence of the invariance of the pairing $\qf{\cdot}{\cdot}$ and the isotropic condition on $l, \overline{l}$, and it is omitted.

\begin{lemma}\label{lem:orthogonalderived}
Let $l \oplus \overline{l} \in \cL$ and $\varepsilon \in \g^c$ holomorphic. Then $\varepsilon$ is orthogonal to the derived Lie subalgebras $[l,l],[\overline{l},\overline{l}] \subset \g^c$, that is,
 \begin{equation}\label{eq:orthogonalderived}
	\varepsilon \in [l,l]^\perp \cap [\overline{l},\overline{l}]^\perp. 
	\end{equation}
      \end{lemma}

\section{Embeddings of Superconformal vertex algebras from Killing spinors}\label{sec:SCVA}

\subsection{Background on supersymmetric vertex algebras}\label{ssec:background}

Here we review basic notions and examples of the theory of supersymmetric vertex algebras that will be needed in the rest of this article. 
We will use extensively the superfield formalism introduced by Heluani--Kac~\cite{SUSYVA}, where details can be found.
As we will work only with $N_K=1$ SUSY vertex algebras and $N_K=1$ SUSY Lie conformal algebras within this formalism~\cite[\S 4]{SUSYVA}, we will refer to these special cases simply as `SUSY vertex algebras' and `SUSY Lie conformal algebras'.
The case $N_K=1$ was studied earlier by Barron~\cite{Barron00}.
As in the rest of Section~\ref{sec:SCVA}, we work over an algebraically closed field $\CC$ of characteristic $0$. The adjective ``super'' applied to algebras, modules and vector spaces will mean $\ZZ/2\ZZ$-graded.

In the standard approach, a vertex algebra consists of a vector superspace $V$, endowed with an even non-vanishing vector $\left|0\right\rangle\in V$ (vacuum), an even endomorphism $T\colon V\to V$ (infinitesimal translation) and a parity preserving linear map $Y\colon V\to(\End V)[[z^{\pm}]]$ (the state-field correspondence) mapping each vector $v\in V$ into a field. By a field, we mean a formal sum
\[
a(z)\defeq Y(a,z)=\sum_{n\in\ZZ}a_{(n)}z^{-n-1}
\]
in an even variable $z$, with `Fourier modes' $a_{(n)}\in\End V$, such that $Y(a,z)b$ is a formal Laurent series for all $b\in V$, so that the operator product expansions are finite sums
\[
a(z)b(w)\sim\sum_{n\geq 0}\frac{(a_{(n)}b)(w)}{(z-w)^{n+1}}.
\]
This data should satisfy several conditions called vacuum axioms, translation invariance and locality (see, e.g.,~\cite[\S 4.1]{Kac} for details).

One is often interested in a vertex algebra that is conformal~\cite[\S 4.10]{Kac}. In particular, the $T$-action is enhanced to an action of the Virasoro algebra for some central charge $c\in\C$. When the Virasoro action admits a further enhancement to a superconformal symmetry, it is computationally more efficient to use the superfield formalism.
In the case of the (Neveu--Schwarz) $N=1$ superconformal algebra, this action includes an odd linear map $S\colon V\to V$ such that $S^2=T$.
Adding a formal odd variable $\theta$ commuting with $z$, we can assign a \emph{superfield}
\[
Y(a,Z)=Y(a,z)+\theta Y(Sa,z),
\]
to each $a\in V$, with $Z=(z,\theta)$, and then the enhanced translation-invariance property is
\begin{equation}\label{eq:enhanced-translation-invariance}
[S,Y(a,Z)]=(\partial_{\theta}-\theta\partial_z)Y(a,Z).
\end{equation}
This motivates the definition of a \emph{SUSY vertex algebra}~\cite{SUSYVA}, which we recall next.

Since $S^2=T$, $V$ is a supermodule over the translation algebra $\cH$, defined as the associative superalgebra with an odd generator $S$, an even generator $T$, and the relation $S^2=T$. This superalgebra can be identified with the parameter algebra, i.e., the associative superalgebra $\cL$ with an odd generator $-\chi$, an even generator $-\lambda$, and the relation $\chi^2=-\lambda$. The two pairs of generators are denoted $\nabla=(T,S)$ and $\Lambda=(\lambda,\chi)$. Expanding now the superfields
\[
Y(a,Z)=\sum_{n \in \Z \atop J=0,1} Z^{-1-n|1-J}a_{(n|J)},
\]
where $Z^{n|J}=z^n\theta^J$, the \emph{$\Lambda$-bracket} and the \emph{normally ordered product} are defined by
\begin{equation}\label{eq:Lambda-bracket-nop}
\left[a_\Lambda b\right] = \displaystyle\sum_{n\in\N \atop J=0,1} \frac{\Lambda^{n|J}}{n!}a_{(n|J)}b,
\qquad
:ab:=a_{(-1|1)}b,
\end{equation}
respectively, where $\Lambda^{n|J}=\lambda^n\chi^J$.
The properties of the first operation motivate the definition~\cite[Def. 4.10]{SUSYVA} of a \emph{SUSY Lie conformal algebra} as the data given by a supermodule $\cR$ over $\cH$ and a parity-reversing bilinear map
\begin{equation}\label{eq:Lambda-bracket.1}
\cR\times \cR\lto \cL\otimes \cR,\quad (a,b)\longmapsto [a_{\Lambda}b],
\end{equation}
called the $\Lambda$-bracket, that satisfies the following identities for all $a,b,c\in\cR$, respectively called sesquilinearity, skew-symmetry or commutativity, and the Jacobi identity:
\begin{gather}\label{eq:sesquiLambda}
\left[Sa_\Lambda b\right] = \chi\left[a_\Lambda b\right],\qquad
\left[a_\Lambda Sb\right] = -(-1)^{\left|a\right|} \left(S+\chi\right)\left[a_\Lambda b\right],  
\\\label{eq:comLambda}
\left[a_\Lambda b\right] = (-1)^{\left|a\right|\left|b\right|}\left[b_{-\Lambda-\nabla}a\right],
\\\label{eq:JacobiLambda}
\left[a_\Lambda\left[b_\Gamma c\right]\right] = (-1)^{\left|a\right|+1}\left[\left[a_\Lambda b\right]_{\Lambda+\Gamma}c\right] + (-1)^{\left(\left|a\right|+1\right)\left(\left|b\right|+1\right)}\left[b_\Gamma\left[a_\Lambda c\right]\right].
\end{gather}
The identities~\eqref{eq:sesquiLambda} and~\eqref{eq:comLambda} take place in $\cL\otimes \cR$, and the identity~\eqref{eq:JacobiLambda} takes place in $\cL\otimes\cL'\otimes\cR$, where $\cL'$ is a copy of $\cL$ with the formal variables $\Lambda=(\lambda,\chi)$ replaced by $\Gamma=(\gamma,\eta)$.
The degree of a homogeneous element $a$ is denoted $\left|a\right|$ (when we use this notation, we are implicitly assuming $a$ is homogeneous).
The precise meaning of the expressions in~\eqref{eq:sesquiLambda},~\eqref{eq:comLambda} and~\eqref{eq:JacobiLambda} is explained in Appendix~\ref{app:1}.

Adding the data given by the second operation in~\eqref{eq:Lambda-bracket-nop}, one obtains the following definition~\cite[Def. 4.19]{SUSYVA}.
A \emph{SUSY vertex algebra} is a SUSY Lie conformal algebra $V$, whose underlying vector superspace is a differential superalgebra, with the (odd) differential $S\colon V\to V$ induced by the $\cH$-module structure on $V$, that is unital (with identity element $\left|0\right\rangle\in V$), quasicommutative and quasiassociative, and such that the $\Lambda$-bracket and the multiplication of this superalgebra, called normally ordered product
\begin{equation}\label{eq:VA-equiv.1}
V\times V\lto V,\quad (a,b)\longmapsto :ab:,
\end{equation}
are related by the non-commutative Wick formula.
Quasicommutativity, quasiassociativity and the non-commutative Wick formula are respectively given by the following identities:
\begin{gather}\label{eq:cuasicon}
:ab:-(-1)^{\left|a\right|\left|b\right|}:ba: = \int_{-\nabla}^0 d\Lambda \left[a_\Lambda b\right],
\\\label{eq:cuasiaso}
::ab:c:-:a:bc:: = :\!\left(\int_0^\nabla d^r\!\Lambda\, a\right) \left[b_\Lambda c\right]\!:+(-1)^{\left|a\right|\left|b\right|} :\!\left(\int_0^\nabla d^r\!\Lambda\, b \right)\left[a_\Lambda c\right]\!:,
\\\label{eq:Wick}
\left[a_\Lambda:bc:\right] = :\left[a_\Lambda b\right]c:+(-1)^{\left(\left|a\right|+1\right)\left|b\right|} :b\left[a_\Lambda c\right]:+\int_0^\Lambda d\Gamma \left[\left[a_\Lambda b\right]_\Gamma c\right].
\end{gather}
The meaning of the integrals in these identities 
is explained in Appendix \ref{app:1}.

Any SUSY Lie conformal algebra $\cR$ determines a SUSY vertex algebra $V(\cR)$ together with an embedding of SUSY Lie conformal algebras $\cR\hra V(\cR)$, that is universal for morphisms $\cR\to V'$ into other SUSY vertex algebras $V'$.
One says that $V(\cR)$ is the SUSY vertex algebra \emph{generated} by $\cR$~\cite[\S 1.8]{SUSYVA}, or more precisely, $V(\cR)$ is the \emph{universal enveloping SUSY vertex algebra} of $\cR$~\cite[Theorem 3.4.2 and \S 4.21]{SUSYVA}. 
%

Next we describe several SUSY vertex algebras used in this paper. 

\begin{example}\label{exam:NS}
The $N=1$ superconformal vertex algebra of central charge $c$, also called the \emph{Neveu--Schwarz vertex algebra} (see, e.g.,~\cite[pp. 178-179]{Kac},~\cite[Example 2.4]{SUSYVA}), is generated by the SUSY Lie conformal algebra $\cR$ whose underlying $\cH$-module is freely generated by an odd vector $H$ and a scalar $c$, with $\Lambda$-bracket
\begin{equation}\label{eq:NS}
\left[{H}_\Lambda H\right] = \left(2T+\chi S+3\lambda\right)H+\frac{\chi\lambda^2}{3}c.
\end{equation}

Expanding in components the corresponding superfield $Y(H,Z)=G(z)+2\theta L(z)$, and the field components in Fourier modes
\begin{equation*}
L(z) = \sum_{n \in \Z} L_n z^{-2-n},\qquad G(z) = \sum_{n \in \frac12+\Z} G_n z^{-\frac32-n}, 
\end{equation*}
one recovers from~\eqref{eq:NS} the usual Virasoro and Neveu--Schwarz commutation relations with central charge $c$ (see~\cite[Example 5.5]{SUSYVA} for details), respectively given by 
\begin{subequations}\label{eq:NS-comm-relations.1}
\begin{gather}\label{eq:NS-comm-relations.1.a}
\left[L_m,L_n\right] = (m-n)L_{m+n}+\delta_{m,-n}\frac{m^3-m}{12}c,
\\\label{eq:NS-comm-relations.1.b}
\left[G_m,L_n\right] = \left(m-\frac{n}{2}\right)G_{m+n},\quad \left[G_m,G_n\right] = 2L_{m+n}+\frac{c}{3}\left(m^2-\frac{1}{4}\right)\delta_{m,-n}.
\end{gather}\end{subequations}
\end{example}

\begin{example}\label{exam:N2}
The \emph{$N=2$ superconformal vertex algebra of central charge $c$} is generated by the SUSY Lie conformal algebra $\cR$ whose underlying $\cH$-module is freely generated by two superfields, namely an odd vector $H$ (a `Neveu-Schwarz vector'), an even vector $J$ (a `current') and a scalar $c$, with $\Lambda$-brackets given by~\eqref{eq:NS} and 
\begin{equation}\label{eq:N2SCVA}
\left[{J}_\Lambda J\right] = -\left(H+\frac{\lambda\chi}{3}c\right),
\quad
\left[{H}_\Lambda J\right] = \left(2T+2\lambda+\chi S\right)J.
\end{equation}
Expanding in components the corresponding superfields
\begin{gather*}
Y(J,Z)=-i\left(J(z)+\theta\left(G^-(z)-G^+(z)\right)\right),
\quad 
Y(H,Z)=\left(G^+(z)+G^-(z)\right)+2\theta L(z),
\end{gather*}
and the field components in Fourier modes
\[
J(z) = \sum_{n \in \Z} J_n z^{-1-n},
\quad
G^{\pm}(z)=\sum_{n \in \frac12+\Z} G_n^{\pm}z^{-\frac32-n}, 
\]
one obtains from the above $\Lambda$-brackets the Virasoro commutation relations~\eqref{eq:NS-comm-relations.1.a} and the following ones (see~\cite[Theorem 5.10]{Kac} and~\cite[Example 3.13]{BZHS}):
\begin{gather*}
\left[J_m,J_n\right] = \frac{m}{3}\delta_{m,-n}c,
\quad
\left[J_m,G_n^{\pm}\right] = \pm G_{m+n}^{\pm},
\quad  
\left[G_m^{\pm},L_n\right] = \left(m-\frac{n}{2}\right)G^{\pm}_{m+n},
\\
\left[L_m,J_n\right] = -nJ_{m+n},
\quad
\left[G_m^+,G_n^-\right] = L_{m+n}+\frac{m-n}{2}J_{m+n}+\frac{c}{6}\left(m^2-\frac{1}{4}\right)\delta_{m,-n}.
\end{gather*}
\end{example}

\begin{remark}\label{rem:tech}
The Neveu--Schwarz vector $H$ can be recovered from the current $J$. More precisely, given an even vector $J$ satisfying the identities~\eqref{eq:N2SCVA} for some odd vector $H$ and an invariant central element $c$, the triple $(J,H,c)$ automatically satisfies~\eqref{eq:NS}, so it determines an $N=2$ superconformal vertex algebra. This follows by a direct application of the Jacobi identity for SUSY Lie conformal algebras (see, e.g.,~\cite[Lemma A.1]{GCYHeluani} for details).
\end{remark}

\begin{example}\label{exam:N4.1}
The (`small') \emph{$N=4$ superconformal vertex algebra of central charge $c$} is generated by the SUSY Lie conformal algebra $\cR$ whose underlying $\cH$-module is freely generated by an odd vector $H$, three even vectors $J^1,J^2,J^3$, and a scalar $c$, with $\Lambda$-brackets given by~\eqref{eq:NS},~\eqref{eq:N2SCVA} for $J=J^i$ with $i=1,2,3$ (i.e., $H$ and $J^i$ determine an $N=2$ superconformal vertex algebra of central charge $c$), and
\begin{equation}\label{eq:N4SCVA.1}
\left[{J^i}_\Lambda J^j\right] = - \varepsilon_{ijk}(S+2\chi)J^k,\qquad i\neq j,
\end{equation}
where $\varepsilon_{ijk}$ is the totally antisymmetric tensor. 
Superfield expansions similar to the ones in Examples~\ref{exam:NS} and~\ref{exam:N2} provide commutation relations for the corresponding Fourier-mode infinite-dimensional Lie superalgebra (see~\cite[Example 3.14]{BZHS}).
\end{example}

\begin{example}\label{exam:superafin}
Let $(\g,\qf{\cdot}{\cdot})$ be a quadratic Lie algebra and $k\in\C$ a scalar (one can start with a Lie superalgebra, but we will not consider this generality in this paper).
Let $\Pi\g$ be the corresponding purely odd vector superspace. Abusing notation,
\[
  [\cdot,\cdot]\colon\Pi\g\times\Pi\g\to\Pi\g,
  \qquad
  \qf{\cdot}{\cdot}\colon\Pi\g\times\Pi\g\to\C,
\]
will denote the bilinear maps corresponding to the Lie bracket and the quadratic form on $\g$, identifying elements $a$ of $\g$ with their corresponding odd copies $\Pi a$ in $\Pi\g$. 
The \emph{universal superaffine vertex algebra with level $k$} associated to $\g$ is the SUSY vertex algebra 
$V^k(\g_{\text{super}})$ generated by the supercurrent algebra or superaffinization $\mathfrak{SCur}\mathfrak{g}$. This is the SUSY Lie conformal algebra with underlying $\cH$-module freely generated by $\Pi\g$ and a scalar $k$, with the $\Lambda$-bracket 
\begin{equation}\label{eq:Lambdaaffine}
\left[{a}_\Lambda b\right]=\left[a,b\right]+\chi\qf{a}{b}k
\end{equation}
for all $a,b\in\Pi\g$. 

The above construction is often considered when the Lie algebra $\g$ is simple or abelian, taking the level $k+h^{\vee}$, where $2h^{\vee}$ is the eigenvalue for the Casimir operator on $\glg$ (see~\cite[Example 5.9]{SUSYVA} and~\cite[pp. 33, 115]{Kac}). Since we do not want to restrict to semisimple or abelian Lie algebras in this paper, we will not use this convention. 
\end{example}

\subsection{Generators of supersymmetry}

Let $\left(\mathfrak{g},\qf{\cdot}{\cdot}\right)$ be a finite-dimensional quadratic Lie algebra over an algebraically closed field $\CC$ of characteristic $0$. Following Section \ref{sec:FDterms}, throughout this section we fix a direct sum decomposition 
$$
\mathfrak{g} = l \oplus \overline{l} \oplus V_-,
$$
with $l,\overline{l}$ isotropic, $\qf{\cdot}{\cdot}_{| l \oplus \overline{l} }$ non-degenerate, and $V_- = ( l \oplus \overline{l})^\perp$. We will use the notation \eqref{eq:projections} for the associated orthogonal projections. We fix a basis $\left\lbrace\epsilon_j,\overline{\epsilon}_j\right\rbrace_{j=1}^n$ of $l \oplus \overline{l}$ satisfying \eqref{eq:isotropybasis-weaker} and an orthogonal basis $\left\lbrace v_\alpha\right\rbrace_{\alpha=1}^m$ of $V_-$. 

Let $V^k(\g_{\text{super}})$ be the universal superaffine vertex algebra associated to $\left(\mathfrak{g},\qf{\cdot}{\cdot}\right)$ with level $0 \neq k \in \CC$ (see Example \ref{exam:superafin}). We fix $\varepsilon \in l \oplus \overline{l}$ and consider the associated odd vectors
\begin{equation}
e := \Pi\varepsilon =  e_{l} + e_{\overline{l}}, \qquad	 u := \Pi (\varepsilon_{l}-\varepsilon_{\overline{l}}) = e_{l}-e_{\overline{l}}.
\end{equation}
We define odd vectors associated to the basis elements
$$
e_j = \Pi\epsilon_j, \qquad e^j = \Pi\overline{\epsilon}_j, \qquad w_\alpha = \Pi v_\alpha, 
$$
and use them to define even vectors on $V^k(\g_{\text{super}})$ by
\begin{subequations}\begin{align}\label{eq:Jmas}
J_0 & = \frac{i}{k}:e^je_j:,
\\\label{eq:Jmasdilaton}
J & = J_0 - \frac{2}{k}Siu.
\end{align}\end{subequations}
Here and in the sequel we use the Einstein summation convention for repeated indices. It is easy to see that the previous expresions are independent of the choice of basis. We aim to prove that $J_0$ and $J$ generate $N=2$ superconformal vertex algebra structures (see Example \ref{exam:N2}) under natural assumptions studied in Section \ref{sec:FDterms}. Our construction extends a construction by E. Getzler for Manin triples \cite{Getzler} (see Remark \ref{rem:Getzler}). 

Before we address this question in Section \ref{ssec:N=2}, in the present section we undertake the calculation of the $\Lambda$-brackets $[{J_0}_{\Lambda} J_0]$ and $[J_\Lambda J]$ without making any assumptions. Define
\begin{equation}\label{eq:w}
w = \left[e^j,e_j\right]_{l}-\left[e^j,e_j\right]_{\overline{l}} = \Pi(\left[\overline{\epsilon}_j,\epsilon_j\right]_{l}-\left[\overline{\epsilon}_j,\epsilon_j\right]_{\overline{l}}) \in \Pi \g.
\end{equation}
Observe that $w$ is independent of the choice of basis, and hence it provides an invariant of the subspace $l \oplus \overline{l} \subset \g$.

\begin{proposition}\label{prop:pasoprimero}
Define $c_0:=3\dim l$ and the vector
\begin{equation}\label{eq:NSmas1}
		H'= H_0+\frac{1}{k}Tw \in V^k(\g_{\text{super}}),
\end{equation}
where
\begin{equation}\label{eq:NSmas0}
		\begin{split}
			H_0 & = \frac{1}{k}\left( :e_j\left(Se^j\right):+:e^j\left(Se_j\right):\right)\\ 
			& +\frac{1}{k^2}\left(:e_j:e^k\left[e^j,e_k\right]::+:e^j:e_k\left[e_j,e^k\right]::\right. \\
			&\left.-:e_j:e_k\left[e^j,e^k\right]::-:e^j:e^k\left[e_j,e_k\right]::\right).
		\end{split}
\end{equation}
Then, one has
\begin{equation}\label{eq:N2SCVAparte1}
		\left[{J_0}_\Lambda{J_0}\right] = -\left(H'+\frac{\lambda\chi}{3}c_0\right).
\end{equation}
\end{proposition}

\begin{proof}
By the non-commutative Wick  formula
	\begin{equation}\label{eq:J0ejej}
		\left[{J_0}_\Lambda:e^je_j:\right] = :\left[{J_0}_\Lambda e^j\right]e_j:+(-1)^{\left(\left|J_0\right|+1\right)\left|e^j\right|}:e^j\left[{J_0}_\Lambda e_j\right]:+\int_0^\Lambda d\Gamma\left[{\left[{J_0}_\Lambda e^j\right]}_\Gamma e_j\right].
\end{equation}
We calculate $\left[{J_0}_\Lambda e^j\right]$ and $\left[{J_0}_\Lambda e_j\right]$ using the commutativity of the $\Lambda$-bracket
	\begin{equation*}
		\begin{split}
			\left[{J_0}_\Lambda e^k\right] & = \frac{i}{k} (-1)^{\left|J_0\right|\left|e^k\right|}\left[{e^k}_{-\Lambda-\nabla}:e^je_j:\right].
		\end{split}
	\end{equation*}
Applying the non-commutative Wick formula, it is easy to see that
	\begin{equation*}
		\begin{split}
			\left[{e^k}_\Lambda:e^je_j:\right] &= :\left[{e^k}_\Lambda e^j\right]e_j:+\left(-1\right)^{\left(\left|e^k\right|+1\right)\left|e^j\right|}:e^j\left[{e^k}_\Lambda e_j\right]:+\int_0^\Lambda d\Gamma \left[\left[{e^k}_\Lambda e^j\right]_\Gamma e_j\right] \\
			&=:\left[{e^k},e^j\right]e_j: + :e^j\left[{e^k},e_j\right]:-\chi ke^k+\lambda k\qf{e^k}{\left[e^j,e_j\right]};
		\end{split}
	\end{equation*}
which implies
	\begin{equation}\label{eq:Jintsuper}
		\begin{split}
			\left[{J_0}_\Lambda e^k\right] &= \frac{i}{k}\left(:\left[{e^k},e^j\right]e_j: + :e^j\left[{e^k},e_j\right]:\right)+i\left(\chi e^k+Se^k-\lambda\qf{e^k}{\left[e^j,e_j\right]}\right),
		\end{split}
	\end{equation}
	\begin{equation}\label{eq:Jintsub}
		\begin{split}
			\left[{J_0}_\Lambda e_k\right] & = \frac{i}{k}\left(:\left[e_k,e^j\right]e_j:+:e^j\left[e_k,e_j\right]:\right)-i\left(\chi e_k+Se_k+\lambda\qf{e_k}{\left[e^j,e_j\right]}\right).
		\end{split}
	\end{equation}
Applying now sesquilinearity in the $\Lambda$-bracket in the integrand of \eqref{eq:J0ejej}, we obtain that
	\begin{equation*}
		\begin{split}
			\left[{J_0}_\Lambda:e^je_j:\right] & = \frac{i}{k}\left(::\left[{e^j},e^k\right]e_k:e_j:+ ::e^k\left[{e^j},e_k\right]:e_j:\right.\\
			& -\left.:e^j:\left[e_j,e^k\right]e_k::-:e^j:e^k\left[e_j,e_k\right]::\right)\\
			&+i\left(:\left(Se^j\right)e_j:+:e^j\left(Se_j\right):+\lambda\left(\left[e^j,e_j\right]_{\overline{l}}-\left[e^j,e_j\right]_{l}\right)\right)\\
			&+\frac{i}{k} \int_0^\Lambda d\Gamma\left[\left(:\left[{e^j},e^k\right]e_k:+:e^k\left[{e^j},e_k\right]:\right)_\Gamma e_j\right]\\
			&+i\int_0^\Lambda d\Gamma\left(\eta-\chi\right)\left(\left[e^j,e_j\right]+\eta k\qf{e^j}{e_j}\right).
		\end{split}
	\end{equation*}
By non-commutative Wick formula and commutativity, it is easy to see that
	$$\partial_{\eta}\left[\left(:\left[{e^j},e^k\right]e_k:+:e^k\left[{e^j},e_k\right]:\right)_\Gamma e_j\right] = -2k\left[e^j,e_j\right]_{\overline{l}}.$$ 
Hence, computing the previous two integrals, we obtain the identity
	\begin{equation*}
		\begin{split}
			\left[{J_0}_\Lambda {J_0}\right] & = - \tilde H - \lambda\chi\dim l  = -\left(H'+\frac{\lambda\chi}{3}c_0\right),
		\end{split}
	\end{equation*}
where
\begin{equation*}
		\begin{split}
			\tilde H & = \frac{1}{k^2}\left(::\left[{e^j},e^k\right]e_k:e_j:+ ::e^k\left[{e^j},e_k\right]:e_j:\right.\\
			& + \left.:e^j:\left[e_j,e^k\right]e_k::-:e^j:e^k\left[e_j,e_k\right]::\right) + \frac{1}{k}\left(:\left(Se^j\right)e_j:+:e^j\left(Se_j\right):\right).
		\end{split}
	\end{equation*}
To finish the proof, it suffices to check that $\tilde H = H'$ as in \eqref{eq:NSmas1}, for which we apply basic properties in Lemma \ref{lem:tech1}.	
Applying \eqref{eq:cuasicon1} and \eqref{eq:cuasicon2} we obtain, respectively, 
	\begin{equation*}
	\begin{split}
-:e^j:\left[e_j,e^k\right]e_k:: & = :e^j:e_k\left[e_j,e^k\right]::,\\
:\left(Se^j\right)e_j: & = :e_j\left(Se^j\right):+T\left[e^j,e_j\right].
	\end{split}
	\end{equation*}
By \eqref{eq:cuasicon3}, we also have
	\begin{equation*}
		\begin{split}
			::e^k\left[{e^j},e_k\right]:e_j: & = :e_j:e^k\left[{e^j},e_k\right]::-k T\left(\left(e_j\left|e^k\right.\right)\left[e^j,e_k\right] -\left(e_j\left|\left[e^j,e_k\right]\right.\right)e^k\right) \\
			& = :e_j:e^k\left[{e^j},e_k\right]::-k\left(T\left[e^j,e_j\right]+T\left[e^j,e_j\right]_{\overline{l}}\right).
		\end{split}
	\end{equation*}
Finally, \eqref{eq:cuasicon1} and \eqref{eq:cuasicon3} imply that
	\begin{equation*}
		\begin{split}
			::\left[{e^j},e^k\right]e_k:e_j: & = -:e_j:e_k\left[{e^j},e^k\right]::+kT\left(\left(\left.e_j\right|e_k\right)\left[{e^j},e^k\right]-\left(e_j\left|\left[{e^j},e^k\right]\right.\right)e_k\right) \\
			&=-:e_j:e_k\left[{e^j},e^k\right]::+kT\left[e^j,e_j\right]_{l}
		\end{split}
	\end{equation*}
and therefore $\tilde H = H'$ as claimed.
\end{proof}

\begin{remark}
Our formula \eqref{eq:NSmas0} for $H_0$ shall be compared with the formula for the vector $H^+$ in \cite[Theorem 2]{StructuresHeluani}.
\end{remark}

We next calculate the $\Lambda$-bracket $\left[{J}_\Lambda J\right]$, where $J$ is defined in \eqref{eq:Jmasdilaton}.

\begin{proposition}\label{prop:pasoprimdil}
Define
	\begin{equation}\label{eq:centralchargedilatoncasi}
		c = 3\left(\dim l+\frac{4}{k}\qf{e}{w-e}\right).
	\end{equation}
and the vector
\begin{equation}\label{eq:NSmasdilatoncasi}
		H = H'-\frac{2}{k}Te+\frac{2}{k^2}S\left(:\left[u,e^j\right]e_j:+:e^j\left[u,e_j\right]:\right) \in V^k(\g_{\text{super}}).
	\end{equation}
Then, one has
	\begin{equation}\label{eq:N2SCVAparte1dil}
		\left[{J}_\Lambda J\right] = -\left(H+\frac{\lambda\chi}{3}c\right).
	\end{equation}
\end{proposition}
\begin{proof}
By sesquilinearity we have 
	\begin{equation*}
		\begin{split}
			\left[{J}_\Lambda J\right] & = \left[{J_0}_\Lambda {J_0}\right] - i\frac{2}{k}\left[{J_0}_\Lambda Su\right] - i\frac{2}{k}\left[{Su}_\Lambda{J_0}\right]-\frac{4}{k^2}\left[{Su}_\Lambda Su\right] \\
			& = \left[{J_0}_\Lambda{J_0}\right]+i\frac{2}{k}\left(\chi+S\right)\left[{J_0}_\Lambda u\right]-i\frac{2}{k}\chi\left[{u}_\Lambda{J_0}\right]-\frac{4}{k^2} \chi\left(S+\chi\right)\left[{u}_\Lambda u\right].
		\end{split}
\end{equation*}
We calculate the $\Lambda$-brackets in the different summands by steps. The term $\left[{J_0}_\Lambda{J_0}\right]$ was calculated in \eqref{eq:N2SCVAparte1}. Combining \eqref{eq:Jintsuper} with \eqref{eq:Jintsub} we have
	\begin{equation*}
		\begin{split}
			\left[{J_0}_\Lambda u\right] & = \qf{e}{e^k}\left[{J_0}_\Lambda e_k\right]-\qf{e}{e_k}\left[{J_0}_\Lambda e^k\right] \\
			& = \qf{e}{e^k}\left(\frac{i}{k}\left(:\left[e_k,e^j\right]e_j:+:e^j\left[e_k,e_j\right]:\right)-i\left(\left(\chi+S\right)e_k+\lambda\qf{e_k}{\left[e^j,e_j\right]}\right)\right) \\
			&-\qf{e}{e_k}\left(\frac{i}{k}\left(:\left[e^k,e^j\right]e_j:+:e^j\left[e^k,e_j\right]:\right)+i\left(\left(\chi+S\right)e^k-\lambda\qf{e^k}{\left[e^j,e_j\right]}\right)\right) \\ 
			& = -i\left(\left(\chi+S\right)e+\lambda\qf{u}{\left[e^j,e_j\right]}\right)+ \frac{i}{k}\left(:\left[u,e^j\right]e_j:+:e^j\left[u,e_j\right]:\right).
		\end{split}
	\end{equation*}
Applying commutativity of the $\Lambda$-bracket, we also obtain
	\begin{equation*}
		\begin{split}
			\left[{u}_\Lambda{J_0}\right] & = (-1)^{\left|u\right|\left|J\right|}\left[{J_0}_{-\Lambda-\Delta}{u}\right] \\
			& = i\left(\chi e + \lambda\qf{u}{\left[e^j,e_j\right]}\right) + \frac{i}{k} \left(:\left[u,e^j\right]e_j:+:e^j\left[u,e_j\right]:\right),
		\end{split}
	\end{equation*}
while application of sesquilinearity shows that
	\begin{equation*}
		\begin{split}
			\left[{u}_\Lambda u\right] = \left[u,u\right]+k\chi\qf{u}{u} = k\chi\qf{u}{u} = -k\chi\qf{e}{e}.
		\end{split}
	\end{equation*}
Adding the different terms above, we conclude that
\begin{align*}
\left[{J}_\Lambda J\right] & =  \left[{J_0}_\Lambda{J_0}\right]+i\frac{2}{k}\left(\chi+S\right)\left[{J_0}_\Lambda u\right]-i\frac{2}{k}\chi\left[{u}_\Lambda{J_0}\right]-\frac{4}{k^2} \chi\left(S+\chi\right)\left[{u}_\Lambda u\right]\\
			& = -\left(H'+\frac{\lambda\chi}{3}c\right)+i^2\frac{2}{k^2}S\left(:\left[u,e^j\right]e_j:+:e^j\left[u,e_j\right]:\right)+\frac{4}{k}\chi\left(S+\chi\right)\chi\qf{e}{e} \\
			& +i\frac{2}{k}\left(\chi+S\right)\left(-i\left(\left(\chi+S\right)e+\lambda\qf{u}{\left[e^j,e_j\right]}\right)\right)-i\frac{2}{k}\chi\left(i\left(\chi e + \lambda\qf{u}{\left[e^j,e_j\right]}\right)\right)\\
			&=-\left(\left(H'-\frac{2}{k}Te+\frac{2}{k^2}S\left(:\left[u,e^j\right]e_j:+:e^j\left[u,e_j\right]:\right)\right)+\frac{\lambda\chi}{3}\left(c_0+\frac{12}{k}\qf{e}{w-e}\right)\right) \\
			& = -\left(H+\frac{\lambda\chi}{3}c\right).
\qedhere
\end{align*}
\end{proof}


\begin{remark}\label{rem:Jflip}
In the definition of $J_0$ and $J$ in \eqref{eq:Jmas} and \eqref{eq:Jmasdilaton}, we consider $l \oplus \overline{l}$ as an ordered pair. Note that, if we denote by $\overline{J_0}$ the vector associated to $\overline{l} \oplus l$, by \eqref{eq:cuasicon1} we have
	\begin{equation*}
\overline{J_0} = \frac{i}{k}:e_je^j: = - \frac{i}{k}:e^je_j:	= -	J_0.
	\end{equation*}
Similarly, if we denote $\overline{u} = e_{\overline{l}} - e_l$, we have 
$$
\overline{J} = \overline{J_0} - \frac{2}{k}Si\overline{u}  = - J_0 + \frac{2}{k}Si u = - J.
$$
Note that $\overline{J_0}$ (respectively $\overline{J}$) satisfies the relation \eqref{eq:N2SCVAparte1} (resp. \eqref{eq:N2SCVAparte1dil}) for the same vector $H'$ (resp. $H$).
\end{remark}

\subsection{$N=2$ supersymmetry from the $F$-term and $D$-term equations}\label{ssec:N=2}

Our main goal in this section is to prove that, under natural conditions studied in Section \ref{sec:FDterms}, the vectors $J_0$ and $J$ (see \eqref{eq:Jmas} and \eqref{eq:Jmasdilaton}) induce embeddings of $N=2$ superconformal algebras on the universal superaffine vertex algebra associated to $\left(\mathfrak{g},\qf{\cdot}{\cdot}\right)$. As a consequence of Proposition \ref{prop:pasoprimero} and Proposition \ref{prop:pasoprimdil}, this reduces to proving the identities (see Remark \ref{rem:tech})
\begin{subequations}
		\begin{align}
			\left[{H'}_\Lambda{J_0}\right] & = \left(2\lambda+2T+\chi S\right)J_0,\label{eq:N2SCVAparte2}\\
			\left[{H}_\Lambda{J}\right] & = \left(2\lambda+2T+\chi S\right)J.\label{eq:N2SCVAparte2dil}
		\end{align}
	\end{subequations}
We start by proving the identity \eqref{eq:N2SCVAparte2}. Firstly, we find a more convenient expression for the vector $H_0$ defined in \eqref{eq:NSmas0}.


\begin{lemma}\label{lem:HKSE}
Assume that $l \oplus \overline{l}$ satisfies the $F$-term equation \eqref{eq:Fterm}. Then the vector $H_0$ defined by \eqref{eq:NSmas0} can be written as follows:
	\begin{equation}\label{eq:NSmas0KSE}
		\begin{split}
			H_0 & = \frac{1}{k}\left(:e_j\left(Se^j\right):+:e^j\left(Se_j\right):\right)\\ 
			&  + \frac{1}{k^2}\left(2:e_j:e^k\left[e^j,e_k\right]_-::+:e^j:e^k\left[e_j,e_k\right]_{l}::+:e_j:e_k\left[e^j,e^k\right]_{\overline{l}}::\right).
		\end{split}
	\end{equation}
\end{lemma}
\begin{proof}
Without any hypothesis,	it follows from \eqref{eq:cuasiconaso1} that
	$$
	:e_j:e^k\left[e^j,e_k\right]:: = :e^j:e_k\left[e_j,e^k\right]::.
	$$
Using now \eqref{eq:cuasiconaso2}, we conclude the identity~\eqref{eq:NSmas0KSE} imposing the involutivity of $l$ and $\overline{l}$.
\end{proof}

We will also need the following technical identity.

\begin{lemma}\label{lem:tech3}
Assume that $l \oplus \overline{l}$ satisfies the $F$-term equation \eqref{eq:Fterm}. Then, for all $i \in \left\lbrace1,\ldots,n\right\rbrace$, we have
	\begin{equation*}\label{eq:Ups2}
		\begin{split}
2:e_j:e_k\left[\left[{e_i},e^j\right]_-,e^k\right]_{\overline{l}}:: & = :\left[{e_i},e^j\right]_+:e^k\left[e_j,e_k\right]_{l}::+:\left[{e_i},e_j\right]_l:e_k\left[e^j,e^k\right]_{\overline{l}}::\\
& + :e_j:\left[{e_i},e_k\right]_l\left[e^j,e^k\right]_{\overline{l}}::+:e^j:\left[{e_i},e^k\right]_+ \left[e_j,e_k\right]_{l}::\\ 
& +:e^j:e^k\left[{e_i},\left[e_j,e_k\right]_{l}\right]_l::+:e_j:e_k\left[{e_i},\left[e^j,e^k\right]_{\overline{l}}\right]_+::.\\
\end{split}
\end{equation*}
\end{lemma}
\begin{proof}
Applying the Jacobi identity for the Lie bracket of $\g$ and the involutivity of $l,\overline{l}$, we obtain
\begin{equation*}
\begin{split}
		:e^j:e^k\left[{e_i},\left[e_j,e_k\right]_{l}\right]_{l}:: &= :e^j:e^k\left[\left[{e_i},e_j\right]_{l},e_k\right]_{l}::+:e^j:e^k\left[{e_j},\left[e_i,e_k\right]_{l}\right]_{l}::, \\
			:e_j:e_k\left[{e_i},\left[e^j,e^k\right]_{\overline{l}}\right]_+:: & = :e_j:e_k\left[{e_i},\left[e^j,e^k\right]\right]_+:: \\
			& =:e_j:e_k\left[\left[{e_i},e^j\right],e^k\right]_+::+:e_j:e_k\left[{e^j},\left[e_i,e^k\right]\right]_+::  \\
			& = :e_j:e_k\left[\left[{e_i},e^j\right]_+,e^k\right]_+::+:e_j:e_k\left[{e^j},\left[e_i,e^k\right]_+\right]_+::\\
			& + :e_j:e_k\left[\left[{e_i},e^j\right]_-,e^k\right]_+::+:e_j:e_k\left[{e^j},\left[e_i,e^k\right]_-\right]_+::.
		\end{split}
	\end{equation*}
	\noindent
By \eqref{eq:cuasicon1} and \eqref{eq:cuasiconaso1} combined with the invariance of $\qf{\cdot}{\cdot}$ we also obtain
	\begin{equation*}
		\begin{split}
			:e^j:e^k\left[\left[{e_i},e_j\right]_{l},e_k\right]_{l}:: & = -:\left[e_i,e^j\right]_{\overline{l}}:e^k\left[e_j,e_k\right]_{l}::,\\
			:e^j:e^k\left[{e_j},\left[e_i,e_k\right]_{l}\right]_{l}:: & = -:e^j:\left[e_i,e^k\right]_{\overline{l}} \left[e_j,e_k\right]_{l}::,\\
			:e_j:e_k\left[\left[{e_i},e^j\right]_{l},e^k\right]_{\overline{l}}:: & = -:\left[e_i,e^j\right]_{l}:e^k\left[e_j,e_k\right]_{l}::,\\
			:e_j:e_k\left[\left[{e_i},e^j\right]_{\overline{l}},e^k\right]_+:: & =:e_j:e_k\left[\left[{e_i},e^j\right]_{\overline{l}},e^k\right]_{\overline{l}}::\\
			&= - :\left[e_i,e_j\right]_{l}:e_k\left[e^j,e^k\right]_{\overline{l}}::,\\
			:e_j:e_k\left[{e^j},\left[e_i,e^k\right]_{l}\right]_{\overline{l}}:: & =  -:e^j:\left[e_i,e^k\right]_{l}\left[e_j,e_k\right]_{l}::,\\
			:e_j:e_k\left[{e^j},\left[e_i,e^k\right]_{\overline{l}}\right]_+:: & = :e_j:e_k\left[{e^j},\left[e_i,e^k\right]_{\overline{l}}\right]_{\overline{l}}::\\
			& = -:e_j:\left[e_i,e_k\right]_{l}\left[{e^j},e^k\right]_{\overline{l}}::.
		\end{split}
	\end{equation*}
	\noindent
Substituting these expressions on the right hand side of \eqref{eq:Ups2}, it suffices to prove that
	$$
	:e_j:e_k\left[\left[e_i,e^j\right]_{l},e^k\right]_{l}::+:e_j:e_k\left[e^j,\left[e_i,e^k\right]_{l}\right]_{l}::=0
	$$
	and
	$$:e_j:e_k\left[\left[{e_i},e^j\right]_-,e^k\right]_+::=:e_j:e_k\left[{e^j},\left[e_i,e^k\right]_-\right]_+::.$$
	By \eqref{eq:cuasiconaso1}, the last identity is trivial using antisymmetry and invariance of the quadratic Lie algebra. Moreover, using the sames properties, one can easily conclude that
	$$
	:e_j:e_k\left[\left[e_i,e^j\right]_{l},e^k\right]_{l}::+:e_j:e_k\left[e^j,\left[e_i,e^k\right]_{l}\right]_{l}:: = 2:e_j:\left[e^k,e_i\right]_{l}\left[e^j,e_k\right]_{l}::.
	$$
	Thus it suffices to see that
	$a := :e_j:\left[e^k,e_i\right]_{l}\left[e^j,e_k\right]_{l}:: = 0.$
	But using the Jacobi identity in the left hand side of the previous equality, we conclude $2a =-a$ by \eqref{eq:cuasicon1}, as required.
\end{proof}

The next step is to calculate $\left[{H_0}_\Lambda a\right]$ for arbitrary $a \in l \oplus \overline{l}$.

\begin{lemma}\label{lem:pasointermedio}
Assume that $l \oplus \overline{l}$ satisfies the $F$-term equation \eqref{eq:Fterm}. Then for any $a = a_l + a_{\overline{l}} \in l \oplus \overline{l}$, we have
	\begin{subequations}
		\begin{align}
			\left[{H_0}_\Lambda a_l \right] & = - \frac{2}{k^2}\left(:e_j:e_k\left[\left[a_l,e^j\right]_-,e^k\right]_{\overline{l}}::\right. \nonumber \\
			& + :e_j:e_k\left[\left[a_l,e^j\right]_-,e^k\right]_-::+:\left[a_l,e^j\right]_-:e_k\left[e_j,e^k\right]_-:: \nonumber \\
			& + \left.:e_j:e^k\left[a_l,\left[e^j,e_k\right]_-\right]_{l}::+:e^j:e_k\left[\left[e^k,a_l\right]_-,e_j\right]_-::\right) \nonumber\\
			& + \frac{1}{k}\left(\chi:e_j\left[a_l,e^j\right]_-:-2:e_j\left(S\left[a_l,e^j\right]_-\right): + 2T\left[e_j,\left[a_l,e^j\right]_-\right]_{l}\right.\nonumber\\
			& + \left.\lambda\left(\left[e_j,\left[a_l,e^j\right]_-\right]_{l}+\left[\left[a_l,e^j\right]_-,e_j\right]_-\right)\right)+\left(\lambda+2T+\chi S\right)a_l,\label{eq:NSbasis10}\\
			\left[{H_0}_\Lambda a_{\overline{l}}\right] & = -\frac{2}{k^2}\left(:e^j:e^k\left[\left[a_{\overline{l}},e_j\right]_-,e_k\right]_{l}::\right. \nonumber\\
			&+:e^j:e^k\left[\left[a_{\overline{l}},e_j\right]_-,e_k\right]_-::+:\left[a_{\overline{l}},e_j\right]_-:e^k\left[e^j,e_k\right]_-:: \nonumber\\
			&+\left.:e^j:e_k\left[a_{\overline{l}},\left[e_j,e^k\right]_-\right]_{\overline{l}}::+:e_j:e^k\left[\left[e_k,a_{\overline{l}}\right]_-,e^j\right]_-::\right) \nonumber\\
			&+\frac{1}{k}\left(\chi:e^j\left[a_{\overline{l}},e_j\right]_-:-2:e^j\left(S\left[a_{\overline{l}},e_j\right]_-\right):+2T\left[e^j,\left[a_{\overline{l}},e_j\right]_-\right]_{\overline{l}} \right.\nonumber\\
			&+\left.\lambda\left(\left[e^j,\left[a_{\overline{l}},e_j\right]_-\right]_{\overline{l}}+\left[\left[a_{\overline{l}},e_j\right]_-,e^j\right]_-\right)\right)+\left(\lambda+2T+\chi S\right)a_{\overline{l}};\label{eq:NSbasis20}\\
			\left[{H'}_\Lambda a_l\right] & = \left[{H_0}_\Lambda a_l\right] + \frac{\lambda}{k}\left[a_l,w\right]-\lambda\chi\qf{a_l}{w}.\label{eq:NSbasis1}\\
			\left[{H'}_\Lambda a_{\overline{l}}\right] & = \left[{H_0}_\Lambda a_{\overline{l}}\right] + \frac{\lambda}{k}\left[a_{\overline{l}},w\right]-\lambda\chi\qf{a_{\overline{l}}}{w}.\label{eq:NSbasis2}
		\end{align}
	\end{subequations}
\end{lemma}

\begin{proof}
Since we can exchange the roles played by $l$ and $\overline{l}$, it suffices to check \eqref{eq:NSbasis10} and \eqref{eq:NSbasis1} for $a_l = e_i$, an element in the basis of $l$. By commutativity of the $\Lambda$-bracket we have
	\begin{equation*}
		\begin{split}
			\left[{H_0}_\Lambda e_i\right] = (-1)^{\left|H_0\right|\left|e_i\right|}\left[{e_i}_{-\Lambda-\nabla}H_0\right].
		\end{split}
	\end{equation*}
Hence we need to calculate
$$
\left[{e_i}_{\Lambda}H_0\right]=\frac{1}{k}\Upsilon_i^1+\frac{1}{k^2}\Upsilon_i^2,
$$
for which we use the expression for $H_0$ in Lemma \ref{lem:HKSE}. We compute first 
$$
\Upsilon_i^1 = \left[{e_i}_\Lambda \left(:e_j\left(Se^j\right):+:e^j\left(Se_j\right):\right)\right] = \Upsilon_i^{1,1}+\Upsilon_i^{1,2}.
$$ 
Applying the non-commutative Wick formula once on each summand, we have:
	\begin{equation*}
		\begin{split}
			\Upsilon_i^{1,1} & = \left[{e_i}_\Lambda :e_j\left(Se^j\right):\right] \\
			& =  :\left[{e_i},e_j\right]\left(Se^j\right):+:e_j\left(S\left[{e_i},e^j\right]\right):-\chi:e_j\left[{e_i},e^j\right]:+\lambda ke_i + \lambda\left[\left[{e_i},e_j\right],e^j\right],\\
			\Upsilon_i^{1,2} & = \left[{e_i}_\Lambda :e^j\left(Se_j\right)\right] \\ 
			& = :\left[{e_i},e^j\right]\left(Se_j\right): + :e^j\left(S\left[{e_i},e_j\right]\right):-\chi:e^j\left[{e_i},e_j\right]: + \chi S k e_i+\lambda\left[\left[{e_i},e^j\right],e_j\right],
		\end{split}
	\end{equation*}
and by the involutivity of $l$,
	\begin{equation*}
		\begin{split}
			\Upsilon_i^1 & = \Upsilon_i^{1,1}+\Upsilon_i^{1,2} \\
			& = :e_j\left(S\left[{e_i},e^j\right]_-\right): + :\left[{e_i},e^j\right]_-\left(Se_j\right):+\left(\lambda+\chi S\right)ke_i\\
			& -\chi\left(:e_j\left[{e_i},e^j\right]:+:e^j\left[{e_i},e_j\right]_{l}:\right)-\lambda\left(\left[\left[{e_j},e_i\right]_{l},e^j\right]+\left[\left[{e^j},e_i\right],e_j\right]\right).
		\end{split}
	\end{equation*} 
Here we have used the identity (which follows from antisymmetry and invariance on $\g$)
	\begin{equation*}\label{eq:Ups1}
		\begin{split}
			:\left[{e_i},e_j\right]_{l}\left(Se^j\right):+:e_j\left(S\left[{e_i},e^j\right]_+\right): + :\left[{e_i},e^j\right]_+\left(Se_j\right): + :e^j\left(S\left[{e_i},e_j\right]_{l}\right): & = 0.
		\end{split}
	\end{equation*}
Next, we compute
	\begin{equation*}
		\begin{split}
			\Upsilon_i^2 & = \left[{e_i}_\Lambda\left(:e^j:e^k\left[e_j,e_k\right]_{l}::+:e_j:e_k\left[e^j,e^k\right]_{\overline{l}}::+2:e_j:e^k\left[e^j,e_k\right]_-::\right)\right] \\
			& = \Upsilon_i^{2,1}+\Upsilon_i^{2,2}+2\Upsilon_i^{2,3}.
		\end{split}
\end{equation*}
First of all, applying twice the non-commutative Wick formula, we obtain 
\begin{equation*}
		\begin{split}
			\Upsilon_i^{2,1} & = \left[{e_i}_\Lambda:e^j\left(:e^k\left[e_j,e_k\right]_{l}:\right):\right] \\
			& = :\left[{e_i},e^j\right]:e^k\left[e_j,e_k\right]_{l}::+:\left(\chi k\delta_i^j\right):e^k\left[e_j,e_k\right]_{l}:: \\
			& + :e^j\left(:\left[{e_i}_\Lambda e^k\right] \left[e_j,e_k\right]_{l}:+:e^k\left[{e_i}_\Lambda\left[e_j,e_k\right]_{l}\right]:+\int_0^\Lambda d\Gamma\left[\left[{e_i}_\Lambda e^k\right]_\Gamma\left[e_j,e_k\right]_{l}\right]\right): \\
			&+\int_0^\Lambda d\Gamma \left(:\left[\left[{e_i},e^j\right]_\Gamma e^k\right]\left[e_j,e_k\right]_{l}:+:e^k\left[\left[{e_i},e^j\right]_\Gamma\left[e_j,e_k\right]_{l}\right]:\right) \\
			&=:\left[{e_i},e^j\right]:e^k\left[e_j,e_k\right]_{l}::+:e^j:\left[{e_i},e^k\right] \left[e_j,e_k\right]_{l}::+:e^j:e^k\left[{e_i},\left[e_j,e_k\right]_{l}\right]:: \\
			&+2\chi k:e^j\left[e_i,e_j\right]_{l}: + \lambda k\left(2\left[\left[e^j,e_i\right]_{\overline{l}},e_j\right]_{\overline{l}}+\left[\left[{e^j},e_i\right]_{l},e_j\right]_{l}\right).
		\end{split}
	\end{equation*} 
In the same way, we obtain
	\begin{equation*}
		\begin{split}
			\Upsilon_i^{2,2} & = \left[{e_i}_\Lambda:e_j\left(:e_k\left[e^j,e^k\right]_{\overline{l}}:\right):\right] \\
			& = :\left[{e_i},e_j\right]:e_k\left[e^j,e^k\right]_{\overline{l}}:: + :e_j:\left[{e_i},e_k\right]\left[e^j,e^k\right]_{\overline{l}}:: + :e_j:e_k\left[{e_i},\left[e^j,e^k\right]_{\overline{l}}\right]:: \\
			& + \chi k:e_j\left[{e_i},e^j\right]_{l}:+\lambda k\left(2\left[\left[e_j,e_i\right]_{l},e^j\right]_{l}+\left[\left[e_j,e_i\right]_{\overline{l}},e^j\right]_{\overline{l}}\right),\\
			\Upsilon_i^{2,3} & = \left[{e_i}_\Lambda:e_j:e^k\left[e^j,e_k\right]_-::\right] \\
			& = :\left[{e_i},e_j\right]:e^k\left[e^j,e_k\right]_-:: + :e_j:\left[{e_i},e^k\right]\left[e^j,e_k\right]_-:: + :e_j:e^k\left[{e_i},\left[e^j,e_k\right]_-\right]:: \\
			& + \chi k:e_j\left[e_i,e^j\right]_-:+\lambda k\left(\left[\left[e^j,e_i\right]_-,e_j\right]_{l}+\left[\left[e^j,e_i\right]_{\overline{l}},e_j\right]_-+\left[\left[e_j,e_i\right]_-,e^j\right]_{\overline{l}}\right).
		\end{split}
	\end{equation*}
Applying now the involutivity of $l$ and $\overline{l}$, we have that
	\begin{equation*}
		\begin{split}
			\Upsilon_i^2 & = \Upsilon_i^{2,1}+\Upsilon_i^{2,2}+2\Upsilon_i^{2,3} \\
			& = 2\left(:e_j:e_k\left[{e^j},\left[e_i,e^k\right]_-\right]_-::+:e_j:e_k\left[\left[{e_i},e^j\right]_-,e^k\right]_{\overline{l}}::\right. \\
			&+\left. :e_j:\left[{e_i},e^k\right]_-\left[e^j,e_k\right]_-::+:e_j:e^k\left[e_k,\left[{e^j},e_i\right]_-\right]_-::+:e_j:e^k\left[{e_i},\left[e^j,e_k\right]_-\right]_{l}::\right) \\
			& +\chi k\left(2:e^j\left[e_i,e_j\right]_{l}:+:e_j\left[{e_i},e^j\right]_{l}:+2:e_j\left[e_i,e^j\right]_-:\right) \\
			& + \lambda k\left(2\left[\left[e^j,e_i\right]_{\overline{l}},e_j\right]_{\overline{l}}+\left[\left[{e^j},e_i\right]_{l},e_j\right]_{l}+2\left[\left[e_j,e_i\right]_{l},e^j\right]_{l}+2\left[\left[e^j,e_i\right]_-,e_j\right]_{l}+2\left[\left[e^j,e_i\right]_{\overline{l}},e_j\right]_-\right),
		\end{split}
	\end{equation*}
where we have used the identities \eqref{eq:cuasiconaso1} and \eqref{eq:cuasiconaso2}, combined with invariance, antisymmetry and the Jacobi identity on $\g$, to add together all the double normally ordered products. Moreover, we have used the identity given in Lemma \ref{lem:tech3} to simplify all the terms that do not have any part in $V_-$.
	
The terms of $\left[{e_i}_\Lambda{H_0}\right]-\left(\lambda+\chi S\right)ke_i$ that are linear in $\chi$ and $\lambda$ can be simplified as follows. Using antisymmetry and invariance on $\g$, we have
\begin{equation*}
		\begin{split}
			k\partial_{\chi}\left(\left[{e_i}_\Lambda{H_0}\right]-\left(\lambda+\chi S\right)e_i\right) & = -:e_j\left[{e_i},e^j\right]:-:e^j\left[{e_i},e_j\right]_{l}: \\
			&+2:e^j\left[e_i,e_j\right]_{l}:+:e_j\left[{e_i},e^j\right]_{l}:+2:e_j\left[e_i,e^j\right]_-: \\
			& = :e_j\left[e_i,e^j\right]_-:,\\
			k\partial_{\lambda}\left(\left[{e_i}_\Lambda{H_0}\right]-\left(\lambda+\chi S\right)e_i\right) & = -\left[\left[{e_j},e_i\right]_{l},e^j\right]-\left[\left[{e^j},e_i\right],e_j\right]+2\left[\left[e^j,e_i\right]_-,e_j\right]_{l}\\
			& + 2\left[\left[e^j,e_i\right]_{\overline{l}},e_j\right]_- + 2\left[\left[e^j,e_i\right]_{\overline{l}},e_j\right]_{\overline{l}}+\left[\left[{e^j},e_i\right]_{l},e_j\right]_l+2\left[\left[e_j,e_i\right]_{l},e^j\right]_{l} \\
			& = \left[\left[e^j,e_i\right]_-,e_j\right]_{l}+\left[\left[{e_i},e^j\right]_-,e_j\right]_-,
		\end{split}
	\end{equation*}
where in the last expression we have used the involutivity of $l$ and $\overline{l}$. From this, we obtain
	\begin{equation*}
		\begin{split}
			\left[{e_i}_\Lambda{H_0}\right] & = \frac{1}{k}\Upsilon_i^1+\frac{1}{k^2}\Upsilon_i^2 \\
			& = \left(\lambda+\chi S\right)e_i + \frac{1}{k}\left(:e_j\left(S\left[{e_i},e^j\right]_-\right): + :\left[{e_i},e^j\right]_-\left(Se_j\right):\right.\\
			&+\left.\chi:e_j\left[e_i,e^j\right]_-:+\lambda\left(\left[\left[e^j,e_i\right]_-,e_j\right]_{l}+\left[\left[{e_i},e^j\right]_-,e_j\right]_-\right)\right)\\
			&+\frac{2}{k^2}\left(:e_j:e_k\left[{e^j},\left[e_i,e^k\right]_-\right]_-::+:e_j:e_k\left[\left[{e_i},e^j\right]_-,e^k\right]_{\overline{l}}::\right.\\
			&+\left.:e_j:\left[{e_i},e^k\right]_-\left[e^j,e_k\right]_-::+ :e_j:e^k\left[e_k,\left[{e^j},e_i\right]_-\right]_-:: + :e_j:e^k\left[{e_i},\left[e^j,e_k\right]_-\right]_{l}::\right).
		\end{split}
	\end{equation*}
Finally, to prove \eqref{eq:NSbasis10}, we use that $S$ is an antiderivation of the product $::$ and \eqref{eq:cuasicon2}:
	\begin{equation*}
		\begin{split}
			\left[{H_0}_\Lambda {e_i}\right] & = -\left[{e_i}_{-\Lambda-\nabla}{H_0}\right] \\
			& = -\frac{2}{k^2}\left(:e_j:e_k\left[\left[e_i,e^j\right]_-,e^k\right]_{\overline{l}}::\right.\\
			&+:e_j:e_k\left[\left[e_i,e^j\right]_-,e^k\right]_-::+:\left[e_i,e^j\right]_-:e_k\left[e_j,e^k\right]_-::\\
			&+\left.:e_j:e^k\left[e_i,\left[e^j,e_k\right]_-\right]_{l}::+:e^j:e_k\left[\left[e^k,e_i\right]_-,e_j\right]_-::\right)\\
			&+\frac{1}{k}\left(\chi:e_j\left[e_i,e^j\right]_-:-2:e_j\left(S\left[e_i,e^j\right]_-\right):+2T\left[e_j,\left[e_i,e^j\right]_-\right]_{l}\right.\\
			&+\left.\lambda\left(\left[e_j,\left[e_i,e^j\right]_-\right]_{l}+\left[\left[e_i,e^j\right]_-,e_j\right]_-\right)\right)+\left(\lambda+2T+\chi S\right)e_i.
		\end{split}
	\end{equation*}
To conclude, we compute $\left[{H'}_\Lambda {e_i}\right]$ calculating $\left[\left({H'-H_0}\right)_\Lambda {e_i}\right]$ by sesquilinearity. Indeed,
	\begin{equation*}
		\begin{split}
			\left[{Tw}_\Lambda e_i\right] & = -\lambda\left[{w}_\Lambda e_i\right]=\lambda\left(\left[e_i,w\right]-\chi k\qf{e_i}{w}\right).
		\end{split}
	\end{equation*}
which implies \eqref{eq:NSbasis1}.
\end{proof}

We are ready to state the first main result of our paper.

\begin{theorem}\label{th:N=2}\label{th:pasofin}
Assume that $l \oplus \overline{l}$ satisfies the $F$-term equation \eqref{eq:Fterm} and the $D$-term gravitino equation \eqref{eq:Dtermgrav}, and that 
$$
w \in [l,l]^\perp \cap [\overline{l},\overline{l}]^\perp.
$$
Then the vectors 
\begin{align*}
			J_0 & = \frac{i}{k}:e^je_j:,\\
			H' & = \frac{1}{k}\left( :e_j\left(Se^j\right):+:e^j\left(Se_j\right):\right)\\
			& +\frac{1}{k^2}\left(:e_j:e^k\left[e^j,e_k\right]::+:e^j:e_k\left[e_j,e^k\right]::\right. \\
			&\left.-:e_j:e_k\left[e^j,e^k\right]::-:e^j:e^k\left[e_j,e_k\right]::\right) + \frac{1}{k}Tw,
\end{align*}
induce an embbeding of the $N=2$ superconformal vertex algebra with central charge $c = 3 \dim l$ into the universal superaffine vertex algebra $V^k(\g_{\text{super}})$ with level $0 \neq k \in \mathbb{C}$.
\end{theorem}

\begin{proof}
By the non-commutative Wick formula, it suffices to calculate
	\begin{equation*}
		\begin{split}
			\left[{H'}_\Lambda:e^ie_i:\right] & = :\left[{H'}_\Lambda e^i\right]e_i:+(-1)^{\left(\left|H'\right|+1\right)\left|e^i\right|}:e^i\left[{H'}_\Lambda e_i\right]: + \int_0^\Lambda d\Gamma \left[\left[{H'}_\Lambda e^i\right]_\Gamma e_i\right] : \\
			& = P + I,
		\end{split}
	\end{equation*}
where $I$ is the third summand of the right-hand side in the first line. Firstly, we compute
	\begin{equation*}
		\begin{split}
			I = \int_0^\lambda \left(I_1+I_2+I_3\right)d\lambda,
		\end{split}
	\end{equation*}
	where
	\begin{equation*}
		\begin{split}
			I_1 & = \partial_{\eta}\left[{A^i}_\Gamma e_i\right],\\
			I_2 & = \partial_{\eta}\left[{B^i}_\Gamma{e_i}\right],\\
			I_2 & = \partial_{\eta}\left[{C^i}_\Gamma{e_i}\right],
		\end{split}
	\end{equation*}
	are given by
	\begin{equation*}
		\begin{split}
			A^i & = \frac{1}{k}\left(\chi:e^j\left[e^i,e_j\right]_-:-2:e^j\left(S\left[e^i,e_j\right]_-\right):\right)-\frac{2}{k^2}\left(:e^j:e^k\left[\left[e^i,e_j\right]_-,e_k\right]_{l}::\right.\\
			&+:e^j:e^k\left[\left[e^i,e_j\right]_-,e_k\right]_-::+:\left[e^i,e_j\right]_-:e^k\left[e^j,e_k\right]_-::\\
			&+\left.:e^j:e_k\left[e^i,\left[e_j,e^k\right]_-\right]_{\overline{l}}::+:e_j:e^k\left[\left[e_k,e^i\right]_-,e^j\right]_-::\right),\\
			B^i & = \frac{1}{k}\left(\lambda\left(\left[e^j,\left[e^i,e_j\right]_-\right]_{\overline{l}}+\left[\left[e^i,e_j\right]_-,e^j\right]_-\right)+2T\left[e^j,\left[e^i,e_j\right]_-\right]_{\overline{l}}\right),\\
			C^i & = \frac{\lambda}{k}\left[e^i,w\right]+\left(\lambda+2T+\chi S\right)e^i.
		\end{split}
	\end{equation*}
We proceed as follows. By commutativity of the $\Lambda$-bracket, we can write
	\begin{equation*}
		\begin{split}
			I_1 & = (-1)^{\left|A^i\right|\left|e_i\right|} \partial_{\eta} \left[{e_i}_{-\Gamma-\nabla}A^i\right] = \partial_{\eta}\left[{e_i}_\Gamma A^i\right]\\
			& = \frac{1}{k}\left(I_1^1-2I_1^2\right)-\frac{2}{k^2}\left(I_1^3+I_1^4+I_1^5+I_1^6+I_1^7\right),
		\end{split}
	\end{equation*}
	where
	\begin{equation*}
		\begin{split}
			I_1^1 & = \partial_{\eta}\left[{e_i}_\Gamma\left(\chi:e^j\left[e^i,e_j\right]_-:\right)\right] = \partial_{\eta}\left(\chi\left[{e_i}_\Gamma:e^j\left[e^i,e_j\right]_-:\right]\right) = - \chi k\left[e^j,e_j\right]_- = 0,\\
			I_1^2 & = \partial_{\eta}\left[{e_i}_\Gamma:e^j\left(S\left[e^i,e_j\right]_-\right):\right] = kS\left[e^j,e_j\right]_--:e^j\left[{e_i},\left[e^i,e_j\right]_-\right]: = -:e^j\left[{e_k},\left[e^k,e_j\right]_-\right]:,\\
			I_1^3 & = \partial_{\eta}\left[{e_i}_\Gamma:e^j:e^k\left[e_j,\left[e^i,e_k\right]_-\right]_{l}::\right] = k\left(:e^k\left[e_j,\left[e^j,e_k\right]_-\right]_{l}:-:e^j\left[e_j,\left[e^k,e_k\right]_-\right]_{l}:\right) \\ 
			& =  k:e^j\left[e_k,\left[e^k,e_j\right]_-\right]_{l}:,\\
			I_1^4 & = \partial_{\eta}\left[{e_i}_\Gamma:e^j:e^k\left[\left[e^i,e_j\right]_-,e_k\right]_-::\right] = k\left(:e^k\left[\left[e^j,e_j\right]_-,e_k\right]_-:-:e^j\left[\left[e^k,e_j\right]_-,e_k\right]_-:\right) \\ 
			& = k:e^j\left[e_k,\left[e^k,e_j\right]_-\right]_-:,\\
			I_1^5 & = \partial_{\eta}\left[{e_i}_\Gamma:\left[e^i,e_j\right]_-:e^k\left[e^j,e_k\right]_-::\right] = -k:\left[e^k,e_j\right]_-\left[e^j,e_k\right]_-:\\
			& = k:\left[e^j,e_k\right]_-\left[e^k,e_j\right]_-: = 0,\\
			I_1^6 & = \partial_{\eta}\left[{e_i}_\Gamma:e^j:e_k\left[e^i,\left[e_j,e^k\right]_-\right]_{\overline{l}}::\right] = k\left(:e_k\left[e^j,\left[e_j,e^k\right]_-\right]_{\overline{l}}:- :e^j\left[\left[e_k,e^k\right]_-,e_j\right]_{l}:\right) \\ & = k:e_j\left[e^k,\left[e_k,e^j\right]_-\right]_{\overline{l}}:,\\
			I_1^7 & = \partial_{\eta}\left[{e_i}_\Gamma:e_j:e^k\left[\left[e_k,e^i\right]_-,e^j\right]_-::\right] = k:e_j\left[\left[e_k,e^k\right]_-,e^j\right]_-: = 0.
		\end{split}
	\end{equation*}
These identities are derived combining the non-commutative Wick formula, the $D$-term gravitino equation \eqref{eq:Dtermgrav}, \eqref{eq:cuasicon1}, and the antisymmetry of the Lie bracket $\g$. Combining the previous expressions, we have proved that
	\begin{equation*}
		\begin{split}
			I_1 & = - \frac{2}{k}:e_j\left[e^k,\left[e_k,e^j\right]_-\right]_{\overline{l}}:.
		\end{split}
	\end{equation*}
By sesquilinearity and the $D$-term gravitino equation \eqref{eq:Dtermgrav}, it is straightforward to show 
	\begin{equation*}
		\begin{split}
			I_2 & = \left(\lambda-2\gamma\right)\qf{\left[e^j,\left[e^i,e_j\right]_-\right]_{\overline{l}}}{e_i} = \left(\lambda-2\gamma\right)\qf{\left[e_j,e^k\right]_-}{\left[e^j,e_k\right]_-},\\
			I_3 & = \frac{\lambda}{k}\left(\left(\left.\left[e_j,e^j\right]_{\overline{l}}\right|\left[e^i,e_i\right]_{l}\right)-\left(\left.\left[e_j,e^j\right]_{l}\right|\left[e^i,e_i\right]_{\overline{l}}\right)\right)+\partial_{\eta}\left(\left(\lambda-2\gamma-\chi\eta\right)\left[{e^i}_\Gamma e_i\right]\right) \\
			& =  k\left(\lambda-2\gamma\right)\qf{e^i}{e_i}+\chi\left[e^i,e_i\right]_+,
		\end{split}
	\end{equation*}
which lead us to the formula
	\begin{equation*}
		\begin{split}
			I & = \lambda\chi\left[e^i,e_i\right]_+-2\frac{\lambda}{k}:e_j\left[e^i,\left[e_i,e^j\right]_-\right]_{\overline{l}}:.
		\end{split}
	\end{equation*}
Next, we compute	
$$
	P=P_1+P_2+P_3
$$
in several steps, where
	\begin{equation*}
		\begin{split}
			P_1 & = :\left(\left[{H_0}_\Lambda e^j\right]-\left(\lambda+2T+\chi S\right)e^j\right)e_j: +:e^j\left(\left[{H_0}_\Lambda e_j\right]-\left(\lambda+2T+\chi S\right)e_j\right):,\\
			P_2 & = :\left[\left({H'-H_0}\right)_\Lambda e^j\right]e_j:+:e^j\left[\left({H'-H_0}\right)_\Lambda e_j\right]e_j:,\\
			P_3 & = :\left(\left(\lambda+2T+\chi S\right)e^j\right)e_j:+:e^j\left(\left(\lambda+2T+\chi S\right)e_j\right):.
		\end{split}
\end{equation*}
First, we can compute directly
\begin{equation*}
		\begin{split}
			P_1	& = -\frac{2}{k^2}\left(:e^i\left(:e_j:e_k\left[\left[e_i,e^j\right]_-,e^k\right]_{\overline{l}}::+:e_j:e_k\left[\left[e_i,e^j\right]_-,e^k\right]_-::\right.\right. \\
			&+:\left[e_i,e^j\right]_-:e_k\left[e_j,e^k\right]_-::+:e_j:e^k\left[e_i,\left[e^j,e_k\right]_-\right]_{l}:: \\
			&+\left.:e^j:e_k\left[\left[e^k,e_i\right]_-,e_j\right]_-::\right):+:\left(:e^j:e^k\left[\left[e^i,e_j\right]_-,e_k\right]_{l}::\right. \\
			&+:e^j:e^k\left[\left[e^i,e_j\right]_-,e_k\right]_-::+:\left[e^i,e_j\right]_-:e^k\left[e^j,e_k\right]_-:: \\
			&+\left.\left.:e^j:e_k\left[e^i,\left[e_j,e^k\right]_-\right]_{\overline{l}}::+:e_j:e^k\left[\left[e_k,e^i\right]_-,e^j\right]_-::\right)e_i:\right)  \\
			&+\frac{1}{k}\left(:e^i\left(\chi:e_j\left[e_i,e^j\right]_-:-2:e_j\left(S\left[e_i,e^j\right]_-\right):\right.\right.\\
			&+\left.\lambda\left(\left[e_j,\left[e_i,e^j\right]_-\right]_{l}+\left[\left[e_i,e^j\right]_-,e_j\right]_-\right)+2T\left[e_j,\left[e_i,e^j\right]_-\right]_{l}\right): \\
			&+:\left(\chi:e^j\left[e^i,e_j\right]_-:-2:e^j\left(S\left[e^i,e_j\right]_-\right):\right.\\
			&+\left.\left.\lambda\left(\left[e^j,\left[e^i,e_j\right]_-\right]_{\overline{l}}+\left[\left[e^i,e_j\right]_-,e^j\right]_-\right)+2T\left[e^j,\left[e^i,e_j\right]_-\right]_{\overline{l}}\right)e_i:\right) \\
			&  = -\frac{2}{k^2}Q+\frac{1}{k}R,
		\end{split}
\end{equation*}
where, after reordering the terms and changing indices, we have
\begin{equation*}
		\begin{split}
			Q &= :e^i:e_j:e_k\left[e^j,\left[e_i,e^k\right]_-\right]_{\overline{l}}:::+ :e^i:e_j:e_k\left[\left[e_i,e^j\right]_-,e^k\right]_-:::\\
			& + :e^i:\left[e^j,e_i\right]_-:e_k\left[e^k,e_j\right]_-:::+:e^i:e_j:e^k\left[e_i,\left[e^j,e_k\right]_-\right]_{l}:::\\
			& + :e^i:e^j:e_k\left[\left[e^k,e_i\right]_-,e_j\right]_-::: + ::e^j:e^k\left[e_j,\left[e^i,e_k\right]_-\right]_{l}::e_i:\\
			& + ::e^j:e^k\left[\left[e^i,e_j\right]_-,e_k\right]_-::e_i: + ::\left[e^i,e_j\right]_-:e^k\left[e^j,e_k\right]_-::e_i:\\
			& + ::e^j:e_k\left[e^i,\left[e_j,e^k\right]_-\right]_{\overline{l}}::e_i: + ::e^j:e_k\left[e^k,\left[e_j,e^i\right]_-\right]_-::e_i:\\
			R &= :\left(\chi:e^j\left[e^k,e_j\right]_-:-2:e^j\left(S\left[e^k,e_j\right]_-\right):+\lambda\left(\left[e^j,\left[e^k,e_j\right]_-\right]_{\overline{l}}+\left[\left[e^k,e_j\right]_-,e^j\right]_-\right) \right.\\
			&\left.+2T\left[e^j,\left[e^k,e_j\right]_-\right]_{\overline{l}}\right)e_k:+:e^k\left(\chi:e_j\left[e_k,e^j\right]_-:-2:e_j\left(S\left[e_k,e^j\right]_-\right): \right.\\
			&+\left.\lambda\left(\left[e_j,\left[e_k,e^j\right]_-\right]_{l}+\left[\left[e_k,e^j\right]_-,e_j\right]_-\right)+2T\left[e_j,\left[e_k,e^j\right]_-\right]_{l}\right):.
		\end{split}
	\end{equation*}
Next, we expand $Q$ using the special identities given in Lemma \ref{lem:tech4} (matching the summands in pairs corresponding to the same line). Applying the antisymmetry and invariance on $\g$, combined with \eqref{eq:cuasicon1}, \eqref{eq:cuasiconaso1} and \eqref{eq:cuasiconaso2}, we obtain the following decomposition:
	\begin{equation*}
		\begin{split}
			Q & = Q_1+Q_2+Q_3+Q_4+Q_5,
		\end{split}
	\end{equation*}
where
	\begin{equation*}
		\begin{split}
			Q_1 & = ::e_j:e^k\left[e^i,\left[e^j,e_k\right]_-\right]_{\overline{l}}::e_i:+::e^k:\left[e^i,\left[e^j,e_k\right]_-\right]_{\overline{l}}e_i::e_j:,\\
			Q_2 & = :e^i:e_j:e^k\left[e_i,\left[e^j,e_k\right]_-\right]_{l}:::+ :e^k:\left[e_i,\left[e^j,e_k\right]_-\right]_{l}:e_je^i:::,\\
			Q_3 & = ::e^j:e^k\left[\left[e^i,e_j\right]_-,e_k\right]_-::e_i: + :e^j:e^k:e_i\left[\left[e^i,e_j\right]_-,e_k\right]_-:::,\\
			Q_4 & = ::\left[e^i,e_j\right]_-:e^k\left[e^j,e_k\right]_-::e_i:+:e^k:\left[e^j,e_k\right]_-:e_i\left[e^i,e_j\right]_-:::,\\
			Q_5 & = ::e^j:e_k\left[\left[e^i,e_j\right]_-,e^k\right]_-::e_i: + :e^j:e_i:e_k\left[\left[e_j,e^i\right]_-,e^k\right]_-:::.
		\end{split}
	\end{equation*}
Applying \eqref{eq:cuasiaso3}, \eqref{eq:cuasicon4}, \eqref{eq:cuasiaso2} and \eqref{eq:cuasicon11} to the first summand in this order, we obtain $Q_1 = 0$, while applying \eqref{eq:cuasiconaso1}, \eqref{eq:cuasicon4}, \eqref{eq:cuasiconaso1}, \eqref{eq:cuasiaso3}, \eqref{eq:cuasicon4}, \eqref{eq:cuasiaso2}, \eqref{eq:cuasiaso3}, \eqref{eq:cuasiaso2} and \eqref{eq:cuasicon1} to the first summand in this order, we obtain $Q_2 = 0$. By \eqref{eq:cuasiaso3}, \eqref{eq:cuasiaso2} and \eqref{eq:cuasicon1} to the first summand in this order, we obtain
	\begin{equation*}
		\begin{split}
			Q_3	& = -kT\left(:e^j\left[\left[e^k,e_j\right]_-,e_k\right]_-:\right).
		\end{split}
	\end{equation*}
Similarly, applying \eqref{eq:cuasicon3}, \eqref{eq:cuasiaso2}, \eqref{eq:cuasiaso3}, \eqref{eq:cuasiaso2} and \eqref{eq:cuasicon1} to the first summand in this order, we obtain $Q_4 = 0$, while applying \eqref{eq:cuasiaso3}, \eqref{eq:cuasiaso2}, \eqref{eq:cuasiconaso2} and antisymmetry of the bracket to the first summand in this order, we obtain $Q_5 = 0$. Putting all together, we have that
	\begin{equation*}
		\begin{split}
			Q & = Q_3 = -kT\left(:e^j\left[\left[e^k,e_j\right]_-,e_k\right]_-:\right).
		\end{split}
	\end{equation*}
The next step is to decompose $R$ as follows:
	\begin{equation*}
		\begin{split}
			R & = \chi R_1-2R_2+\lambda R_3+2R_4,
		\end{split}
	\end{equation*}
where
	\begin{equation*}
		\begin{split}
			R_1 & = ::e^j\left[e^k,e_j\right]_-:e_k:-:e^k:e_j\left[e_k,e^j\right]_-::,\\
			R_2 & = ::e^j\left(S\left[e^k,e_j\right]_-\right):e_k:+:e^k:e_j\left(S\left[e_k,e^j\right]_-\right)::,\\
			R_3 & = :\left[e^j,\left[e^k,e_j\right]_-\right]_{\overline{l}}e_k:+:e^k\left[e_j,\left[e_k,e^j\right]_-\right]_{l}:+:\left[\left[e^k,e_j\right]_-,e^j\right]_-e_k:+:e^k\left[\left[e_k,e^j\right]_-,e_j\right]_-:,\\
			R_4 & = :\left(T\left[e^j,\left[e^k,e_j\right]_-\right]_{\overline{l}}\right)e_j:+:e^k\left(T\left[e_j,\left[e_k,e^j\right]_-\right]_{l}\right):.
		\end{split}
	\end{equation*}
Applying \eqref{eq:cuasiaso2} and \eqref{eq:cuasicon1} we have that $R_1 = 0$, while applying \eqref{eq:cuasiaso1} and \eqref{eq:cuasicon2}, we obtain
	\begin{equation*}
		\begin{split}
			R_2 & = T\left(:e^j\left[\left[e^k,e_j\right]_-,e_k\right]:\right).
		\end{split}
	\end{equation*}
By antisymmetry and invariance, we trivially have that
	\begin{equation*}
		\begin{split}
			R_3 & = 2:\left[e^j,\left[e^k,e_j\right]_-\right]_{\overline{l}}e_k:+:\left[\left[e^j,e_k\right]_-,e^k\right]_-e_j:+:e^j\left[\left[e_j,e^k\right]_-,e_k\right]_-:,\\
			R_4
			& = T\left(:e^j\left[\left[e^k,e_j\right]_-,e_k\right]_{l}:\right).
		\end{split}
	\end{equation*}
Putting all together, by Remark \ref{rem:involutivity}, we have that
	\begin{equation*}
		\begin{split}
			R & = \lambda\left(2:\left[e^j,\left[e^k,e_j\right]_-\right]_{\overline{l}}e_k:+:\left[\left[e^j,e_k\right]_-,e^k\right]_-e_j:+:e^j\left[\left[e_j,e^k\right]_-,e_k\right]_-:\right) \\
			& - 2T\left(:e^j\left[\left[e^k,e_j\right]_-,e_k\right]_-\right),
		\end{split}
	\end{equation*}
and we conclude
	\begin{equation*}
		\begin{split}
			P_1 & = \frac{\lambda}{k}\left(2:\left[e^j,\left[e^k,e_j\right]_-\right]_{\overline{l}}e_k:+:\left[\left[e^j,e_k\right]_-,e^k\right]_-e_j:+:e^j\left[\left[e_j,e^k\right]_-,e_k\right]_-:\right).
		\end{split}
	\end{equation*}
Using now the hypothesis $w \in \left[l,l\right]^\perp \cap\left[\overline{l},\overline{l}\right]^\perp$ combined with the properties of invariance and antisymmetry on $\g$, we see that
	\begin{equation*}
		\begin{split}
			P_2 & = \frac{\lambda}{k}\left(:e^j\left[e_j,\left[e^k,e_k\right]_{l}-\left[e^k,e_k\right]_{\overline{l}}\right]:+:\left[e^j,\left[e^k,e_k\right]_{l}-\left[e^k,e_k\right]_{\overline{l}}\right]e_j:\right) \\
			&+ :e^j\left(\lambda\chi\left(e_j\left|\left[e^k,e_k\right]\right.\right)\right):-:\left(\lambda\chi\left(e^j\left|\left[e^k,e_k\right]\right.\right)\right)e_j:\\
			& = -\lambda\chi\left[e^j,e_j\right]_++\frac{\lambda}{k}\left(:\left[e^j,\left[e^k,e_k\right]\right]_-e_j:-:e^j\left[e_j,\left[e^k,e_k\right]\right]_-:\right),
		\end{split}
	\end{equation*}
where, for the last two terms, we have used the $F$-term and $D$-term gravitino equations \eqref{eq:Ftermgrav} and \eqref{eq:Dtermgrav} in order to apply the Jacobi identity for the Lie bracket.
	
At last, using that $T$ is a derivation and $S$ is an antiderivation for the normally ordered product, it is easy to check that
	\begin{equation*}
	\begin{split}
	P_3 & = \left(2\lambda+2T+\chi S\right):e^je_j:.
	\end{split}
	\end{equation*}
To conclude, by the Jacobi identity, we have that
	\begin{equation*}
		\begin{split}
			P & = \left(2\lambda+2T+\chi S\right):e^je_j:-\lambda\chi\left[e^j,e_j\right]_++\frac{2\lambda}{k}:\left[e^j,\left[e^k,e_j\right]_-\right]_{\overline{l}}e_k:\\
			& + \frac{\lambda}{k}\left(:\left[\left[e^j,e_k\right]_-,e^k\right]_-e_j:+:e^j\left[\left[e_j,e^k\right]_-,e_k\right]_-:+:\left[e^j,\left[e^k,e_k\right]\right]_-e_j:-:e^j\left[e_j,\left[e^k,e_k\right]\right]_-:\right)\\
			& = \left(2\lambda+2T+\chi S\right):e^je_j:-\lambda\chi\left[e^j,e_j\right]_++\frac{2\lambda}{k}:\left[e^j,\left[e^k,e_j\right]_-\right]_{\overline{l}}e_k:.
		\end{split}
	\end{equation*}
which, combined with our formula for the integral $I$, now implies
\begin{equation*}
		\begin{split}
			\left[{H'}_\Lambda:e^j e_j:\right] & = \left(2\lambda+2T+\chi S\right):e^je_j:.
		\end{split}
\end{equation*}
This is equivalent to \eqref{eq:N2SCVAparte2}, and hence the statement follows from Proposition \ref{prop:pasoprimero} and Remark \ref{rem:tech}.
\end{proof}

\begin{remark}\label{rem:Getzler}
Under the hypothesis of Theorem \ref{th:N=2}, if we assume further that $l \oplus \overline{l} = \g$, then $(\g,l,\overline{l})$ is a Manin triple and we recover a construction by E. Getzler \cite{Getzler}. Explicit examples where Theorem \ref{th:N=2} applies and such that $l \oplus \overline{l}$ is not a Manin triple will be provided in Section \ref{sec:HS}.
\end{remark}

To prove our second main theorem, we need the following technical lemma. 

\begin{lemma}\label{lem:previothm2}
Assume that $l \oplus \overline{l}$ satisfies the $F$-term equation \eqref{eq:Fterm} and the $D$-term dilatino equation \eqref{eq:Dtermdil} and that $\varepsilon \in l \oplus \overline{l}$ is holomorphic (Definition \ref{def:holoiso}). Then
$$
H = H_0,
$$
where the vectors $H$ and $H_0$ are defined in \eqref{eq:NSmasdilatoncasi} and \eqref{eq:NSmas0}, respectively. Furthermore, the constant $c$ in \eqref{eq:centralchargedilatoncasi} equals $c = 3\dim l+\tfrac{12}{k}\qf{e}{e}$.
\end{lemma}
\begin{proof}
The holomorphicity of $\varepsilon \in l \oplus \overline{l}$ implies that
	\begin{equation}\label{eq:NSmasdilaton}
		H = H'-\frac{2}{k}Te = H_0+\frac{1}{k}T(w -2e).
	\end{equation}
This follows from the basic identity $:\left[a,e^j\right]_{\overline{l}}e_j:+:e^j\left[a,e_j\right]_l:=0$ for all $a \in \Pi \mathfrak{g}$ (by invariance of $\qf{\cdot}{\cdot}$ and antisymmetry of $\left[\cdot,\cdot\right]$). Note also that the $D$-term dilatino equation \eqref{eq:Dtermdil} is equivalent to  $\left[e^j,e_j\right]_+ = 2u$ and, equivalently, to
\begin{equation}\label{eq:Dtermnewnot}
w = 2e.
\end{equation}
The lemma follows from \eqref{eq:NSmasdilaton} and \eqref{eq:centralchargedilatoncasi}.
\end{proof}



\begin{theorem}\label{th:N=2dil}
Assume that $l \oplus \overline{l}$ satisfies the $F$-term equation \eqref{eq:Fterm} and the $D$-term equation \eqref{eq:Dterm}, and that $\varepsilon \in l \oplus \overline{l}$ is holomorphic. Then the vectors 
\begin{align*}
			J & = \frac{i}{k}:e^je_j: - \frac{2}{k}Siu,\\
			H & = \frac{1}{k}\left( :e_j\left(Se^j\right):+:e^j\left(Se_j\right):\right)\\ 
			& +\frac{1}{k^2}\left(:e_j:e^k\left[e^j,e_k\right]::+:e^j:e_k\left[e_j,e^k\right]::\right. \\
			&\left.-:e_j:e_k\left[e^j,e^k\right]::-:e^j:e^k\left[e_j,e_k\right]::\right),
\end{align*}
induce an embbeding of the $N=2$ superconformal vertex algebra with central charge 
\begin{equation}\label{eq:centralchargedilaton}
		c = 3\left(\dim l+\frac{4}{k}\qf{e}{e}\right)
	\end{equation}
into the universal superaffine vertex algebra $V^k(\g_{\text{super}})$ with level $0 \neq k \in \mathbb{C}$.
\end{theorem}

\begin{proof}
The proof reduces to checking the identity \eqref{eq:N2SCVAparte2dil}. By sesquilinearity, we have
	\begin{equation*}
		\begin{split}
			\left[{H}_\Lambda{J}\right] & = \left[{H'}_\Lambda{J_0}\right]-i\frac{2}{k}\left[{H'}_\Lambda Su\right]-\frac{2}{k}\left[{Te}_\Lambda{J_0}\right]+i\frac{4}{k^2}\left[{Te}_\Lambda Su\right] \\
			& = \left[{H'}_\Lambda{J_0}\right]-i\frac{2}{k}\left(S+\chi\right)\left[{H'}_\Lambda u\right]+\frac{2}{k}\lambda\left[{e}_\Lambda{J_0}\right]-i\frac{4}{k^2}\lambda\left(\chi+S\right)\left[{e}_\Lambda u\right].
		\end{split}	
	\end{equation*}
The first summand corresponds to the identity \eqref{eq:N2SCVAparte2}, which we proved in Theorem \ref{th:N=2}. We next compute the other summands in a series of steps. 

By hypothesis $\varepsilon$ is an infinitesimal isometry (Definition \ref{def:holoiso}), and therefore
\begin{equation}\label{eq:twoholomorphic}
		\left[e_l,\overline{l}\right] \subset l \oplus \overline{l} \quad \text{and} \quad \left[e_{\overline{l}},l\right] \subset l \oplus \overline{l}.
	\end{equation}
By sesquilinearity and Lemma \ref{lem:pasointermedio}, using the equations \eqref{eq:twoholomorphic} above, we have
	\begin{equation*}
		\begin{split}
			\left[{H'}_\Lambda u\right] & = \qf{e}{e^i}\left[{H'}_\Lambda e_i\right]-\qf{e}{e_i}\left[{H'}_\Lambda e^i\right] = \left(\lambda+2T+\chi S\right)u+ \frac{\lambda}{k}\left[u,w\right]_++\lambda\chi\qf{e}{\left[e^j,e_j\right]}.
		\end{split}
	\end{equation*}
By \eqref{eq:Jintsub} and \eqref{eq:Jintsuper}, and using that  $\varepsilon$ is an infinitesimal isometry, we have the following:
	\begin{equation*}
		\begin{split}
			\left[{e}_\Lambda{J_0}\right] & = \qf{e}{e^k}\left[{e_k}_\Lambda{J_0}\right]+\qf{e}{e_k}\left[{e^k}_\Lambda{J_0}\right] \\
			&= \frac{i}{k}\left(\qf{e}{e^k}\left(:\left[e_k,e^j\right]e_j:+:e^j\left[e_k,e_j\right]:+ k\chi e_k+k\lambda\qf{e_k}{\left[e^j,e_j\right]}\right)\right. \\
			& +\left. \qf{e}{e_k}\left(:\left[e^k,e^j\right]e_j:+:e^j\left[e^k,e_j\right]:-k\chi e^k+k\lambda\qf{e^k}{\left[e^j,e_j\right]}\right)\right) \\
			& = i\chi u+i\lambda\qf{e}{\left[e^j,e_j\right]}.
		\end{split}
	\end{equation*} 
This follows from the basic identity $:\left[a,e^j\right]_{\overline{l}}e_j:+:e^j\left[a,e_j\right]_l:=0$ for all $a \in \mathfrak{g}$ (see the proof of Lemma \ref{lem:previothm2}). Using again the infinitesimal isometry condition \eqref{eq:isometry}, a simple computation shows that
	\begin{equation*}
		\begin{split}
			\left[{e}_\Lambda u\right] = \left[e,u\right]+k\chi\qf{e}{u} = \left[e,u\right]_+.
		\end{split}
	\end{equation*}
	\noindent
Applying now the $D$-term equation \eqref{eq:Dtermnewnot}, we obtain the required identity \eqref{eq:N2SCVAparte2dil}:
	\begin{equation*}
		\begin{split}
			\left[{H}_\Lambda J\right] & = \left[{H'}_\Lambda{J_0}\right]-i\frac{2}{k}\left(S+\chi\right)\left[{H'}_\Lambda u\right]+\frac{2}{k}\lambda\left[{e}_\Lambda{J_0}\right]-i\frac{4}{k^2}\lambda\left(\chi+S\right)\left[{e}_\Lambda u\right] \\
			& = -i\frac{4}{k^2}\lambda\left(S+\chi\right)\left(\left[u,e\right]_++\left[e,u\right]_+\right) + \left(2\lambda+2T+\chi S\right)J \\
			& = \left(2\lambda+2T+\chi S\right)J.
		\end{split}
	\end{equation*}
The formulae for $H$ and the central charge follow from Lemma \ref{lem:previothm2}.
\end{proof}
\noindent

\begin{remark}
The embeddings of Theorems~\ref{th:N=2} and~\ref{th:N=2dil} do not neccesarily induce on $V^k(\mathfrak{g}_{\rm super})$ an `$N=2$ superconformal structure', i.e., a structure of $N=2$ superconformal vertex algebra in the sense of~\cite[Def. 5.6]{SUSYVA}. In fact, defining the Fourier modes $L_n$ associated to $H'$ (for Theorem~\ref{th:N=2}) or $H$ (for Theorem~\ref{th:N=2dil}) as in Example~\ref{exam:NS}, $L_{-1}$ may not be the translation operator $T$ of $V^k(\mathfrak{g}_{\rm super})$, while $L_0$ may not act semisimply on $V^k(\mathfrak{g}_{\rm super})$.
\end{remark}

\begin{remark}
Assume that $l\oplus\overline{l}$ satisfies the $F$-term equation \eqref{eq:Fterm} and the $D$-term gravitino equation \eqref{eq:Dtermgrav} and that $\varepsilon \in l\oplus\overline{l}$ is holomorphic. Then, by the proof of the previous theorem, there is an embedding of the $N=2$ superconformal vertex algebra generated by $(J,H)$ with central charge \eqref{eq:centralchargedilatoncasi}, where $H$ is defined by \eqref{eq:NSmasdilaton}, on the universal superaffine vertex algebra $V^k(\g_{\text{super}})$ with level $0 \neq k \in \mathbb{C}$.
\end{remark}

\begin{remark}\label{rem:HZ}
It should be noticed that the expressions for the elements $J$ and $H$ in Theorem \ref{th:N=2dil} are formally very similar to the corresponding ones in \cite[Theorem 2]{StructuresHeluani}, but their interpretation is different. On the one hand, our expression for the fields uses a global frame for $l$ and $\overline{l}$, which is not required in \cite{StructuresHeluani}. On the other hand, the result here is much more general in that it does not need two Killing spinors for the construction of the $N=2$ structure, neither a globally defined \emph{dilaton function} (cf. Remark \ref{rem:Zquotient}). We thank an anonymous referee for this comment.
\end{remark}

\section{$(0,2)$ mirror symmetry}
\label{sec:Tdual}

\subsection{Killing spinors on exact Courant algebroids}\label{ssec:KsCourant}

In this section we study the Killing spinor equations on an exact Courant algebroid \cite{GF3,grt}, for Riemannian generalized metrics and a special class of divergences called \emph{closed} \cite{GFStreets}. As we will see in Proposition \ref{prop:geomalg}, restricting to invariant solutions on a homogeneous manifold, we obtain Killing spinors on a quadratic Lie algebra in the sense of Definition \ref{def:killing}. These are the basic ingredients of our construction of (0,2) mirrors in Section \ref{02example}.

Let $M$ be a smooth oriented spin manifold endowed with an exact Courant algebroid (see, e.g., \cite{GFStreets})
\begin{equation}\label{eq:exactE}
0 \longrightarrow T^* \overset{\pi^*}{\longrightarrow} E \overset{\pi}{\longrightarrow} T \longrightarrow 0.
\end{equation}
We will denote the indefinite pairing and the Dorfman bracket on sections of $E$ by $\langle , \rangle$ and $[,]$, respectively. Recall that a choice of isotropic splitting $s \colon T \to E$ of \eqref{eq:exactE} induces a vector bundle isomorphism $E \cong T \oplus T^*$, such that  
\begin{equation}\label{eq:pairingbracketexp}
\langle X + \xi, X + \xi \rangle = \xi(X), \qquad [X+ \xi,Y + \eta] = [X,Y] + L_X \eta - i_Y d\xi + i_Yi_X H,
\end{equation}
where $H \in \Omega^3(M)$ is a closed three-form determined by $s$. The Killing spinor equations on $E$ are defined in terms of a pair $(V_+,\operatorname{div})$, as given in the following definition.

\begin{definition}\label{def:GmetricE}
\hfill

\begin{enumerate}
\item A (Riemannian) generalized metric on $E$ is an orthogonal decomposition $E = V_+ \oplus V_-$, so that the restriction of $\langle \cdot,\cdot\rangle$ to $V_+$ is positive definite.
	
\item A divergence operator on $E$ is a map $\operatorname{div} \colon \Gamma(E) \to C^\infty(M)$ satisfying
$$
\operatorname{div}(fa) = f \operatorname{div}(a) + \pi(a)(f)
$$
for any $f \in C^\infty(M)$.

\end{enumerate}
	
\end{definition}

A choice of generalized metric $V_+ \subset E$ determines an isotropic splitting $s \colon T \to E$, and hence a preferred presentation of the Courant algebroid structure as in \eqref{eq:pairingbracketexp}. In particular, in this splitting one has
$$
V_\pm = \{ X \pm g(X) \; | \; X \in T\}
$$
for a uniquely determined Riemannian metric $g$ on $M$. Consequently, the generalized metric has an associated \emph{Riemannian divergence} defined by
\begin{equation}\label{eq:div0}
\operatorname{div}_0(X + \xi) = \frac{L_X \mu_g}{\mu_g}
\end{equation}
where $\mu_g$ is the volume element of $g$. Following \cite{GFStreets}, we introduce next a compatibility condition for pairs $(V_+,\operatorname{div})$. Recall that the space of divergence operators on $E$ is affine, modelled on the space of sections of $E^* \cong E$.

\begin{definition}\label{def:compatibleE}
A pair $(V_+,\operatorname{div})$ is \emph{compatible} if $\la e, \cdot \ra:= \operatorname{div}_0 - \operatorname{div}$ is an infinitesimal isometry of $V_+$, that is,
$$
[e,V_+] \subset V_+.
$$ 
Furthermore, we say that $(V_+,\operatorname{div})$ is closed if $\pi(e) = 0$.	
\end{definition}

Given a compatible pair $(V_+,\operatorname{div})$, we have that $e = X + \varphi$ in the splitting determined by $V_+$ and the infinitesimal isometry condition reads (see \cite[Lemma 2.50]{GFStreets})
\begin{equation}\label{eq:infisometry}
L_Xg = 0, \qquad d\varphi = i_XH .
\end{equation}
Thus the condition of being closed corresponds to $X = 0$ and $d \varphi = 0$.

To introduce the Killing spinor equations on $E$, let us fix a pair $(V_+,\operatorname{div})$ as above. The fixed spin structure on $M$ combined with the isometries 
$$
\sigma_\pm \colon (T,\pm g) \to V_\pm \colon X \to X \pm g(X).
$$ 
determine complex spinor bundles $S_\pm$ for $V_\pm$. Associated to the pair $(V_+,\operatorname{div})$ there are canonical first order differential operators (see \cite{GF3})
\begin{equation}\label{eq:LCspinpuregeom}
\begin{split}
	D^{S_+}_- & \colon \Gamma(S_+) \to \Gamma(V_-^* \otimes S_+), \qquad \slashed D^+ \colon \Gamma(S_+) \to \Gamma(S_+),\\
	D^{S_-}_+ & \colon \Gamma(S_-) \to \Gamma(V_+^* \otimes S_-), \qquad \slashed D^- \colon \Gamma(S_-) \to \Gamma(S_-).
\end{split}
\end{equation}
The operators $D^{S_\pm}_\mp$ correspond to the unique lifts to $S_\pm$ of the metric-preserving operators 
$$
D_{a_-} b_ + = [a_-,b_+]_+, \qquad D_{a_+} b_ - = [a_+,b_-]_-,
$$
acting on sections $a_\pm, b_\pm \in \Gamma(V_\pm)$. The Dirac type operators $\slashed D^\pm$ are more difficult to construct, as they involve torsion-free generalized connections (much like in Section \ref{sec:KSeq}). In the situation of our interest, they can be described 
in terms of the following affine metric connections with totally skew-symmetric torsion (where $\nabla$ is the Levi--Civita connection): 
$$
\nabla^\pm = \nabla \pm \tfrac{1}{2} H, \qquad \nabla^{\pm \tfrac{1}{3}} = \nabla \pm \tfrac{1}{6} H.
$$

\begin{lemma}[\cite{GF3}]\label{lem:DpmexpE}
Let $(V_+,\operatorname{div})$ be a pair given by a generalized metric and a divergence operator. Denote $\operatorname{div}_0 - \operatorname{div} = \la e, \cdot \ra $, and set $\varphi_\pm = g(\pi e_\pm, \cdot) \in T^*$. Then, for any spinor $\eta \in \Gamma(S_\pm)$ one has
\begin{align*}
D_{\sigma_\mp X}^{S_\pm} \eta & = \nabla^\pm_X \eta,\\
\slashed D^\pm \eta & = \slashed \nabla^{\pm \tfrac{1}{3}} \eta - \tfrac{1}{2}\varphi_\pm \cdot \eta.
\end{align*}

\end{lemma} 

Observe that $D_\mp^{S_\pm}$ are independent of the choice of divergence, while the dependence of $\slashed D^\pm$ on $\operatorname{div}$ is only through the \emph{partial divergence operator}
$$
\operatorname{div}_\pm:= \operatorname{div}_{|V_\pm} \colon \Gamma(V_\pm) \to C^\infty(M).
$$

We are ready to introduce the equations of our interest. 

\begin{definition}\label{def:killingE}
Let $M$ be a smooth oriented spin manifold endowed with an exact Courant algebroid $E$. A triple $(V_+,\operatorname{div}_\pm,\eta)$, given by a generalized metric $V_+$, a partial divergence operator $\operatorname{div}_\pm$, and a non-vanishing spinor $\eta \in S_\pm$, is a solution of the \emph{Killing spinor equations}, if
	\begin{equation}\label{eq:killingE}
		\begin{split}
			D^{S_\pm}_\mp \eta &= 0,\\
			\slashed D^\pm \eta &= 0.
		\end{split}
	\end{equation}
\end{definition}

As we show in our next result, there is a perfect match between the geometric Definition \ref{def:killingE} and the algebraic Definition \ref{def:killing}, provided that we consider invariant solutions of \eqref{eq:killingE} on a homogeneous manifold. 
Assume that $M$ is endowed with the action of a Lie group $K$. Then, we say that $E$ is equivariant if the $K$-action on $M$ lifts to $E$ preserving the Courant algebroid structure.

\begin{proposition}\label{prop:geomalg}
Let $M$ be a smooth oriented spin manifold endowed with a left transitive action of a Lie group $K$. Let $E$ be an exact equivariant Courant algebroid over $M$. Then, the space of invariant sections
$$
\g  = \Gamma(E)^K
$$
of $E$, endowed with the induced bracket and pairing, defines a quadratic Lie algebra. Furthermore, there is a one-to-one correspondence between invariant solutions of \eqref{eq:killingE} and solutions to the Killing spinor equations \eqref{eq:killing} on $\g$.
\end{proposition} 

\begin{proof}
The first part of the statement is straightforward from the axioms of a Courant algebroid, transitivity of the action, and the equivariance of $E$. As for the second part, it follows from the natural construction of the operators $D^{S_\pm}_\mp$ and $\slashed D^\pm$ using torsion-free generalized connections \cite{GF3}. Observe here that an invariant generalized connection on $E$ corresponds to an element $D \in \g^* \otimes \Lambda^2 \g$ as in Definition \ref{def:Gconnection}. Similarly, for an invariant pair $(V_+,\operatorname{div}_\pm)$ one has 
\[
\operatorname{div}_{0|V_\pm} - \operatorname{div}_\pm = \la e_\pm, \cdot \ra \in V_\pm^*.
\qedhere\]
\end{proof}

We finish this section with a more explicit description of the solutions of the Killing spinor equations on an even-dimensional manifold for closed pairs (Definition \ref{def:compatibleE}). This will be useful for understanding the salient geometric implications of (0,2) mirror symmetry in the following sections. Recall that, in even dimensions, our spinor bundle decomposes as
$$
S_+ = S_+^+ \oplus S_+^-,
$$
where the factors correspond to irreducible spin representations (cf. Section \ref{sec:KSeq}). The resulting system of equations \eqref{eq:twistedStrom} below was first considered in \cite{grst} in relation to holomorphic Courant algebroids. The next result is a particular case of \cite[Proposition 3.1]{grst2}, but we include a sketch of the proof here for the convenience of the reader.

\begin{proposition}\label{lemma:KillingevenE}
Let $M$ be a smooth oriented spin manifold of dimension $2n$ endowed with an exact Courant algebroid $E$. Then, a solution $(V_+,\operatorname{div}_+,\eta)$ of the Killing spinor equations \eqref{eq:killingE}, with $\operatorname{div}_+ = \operatorname{div}_{|V_+}$ for $(V_+,\operatorname{div})$ closed and $\eta \in \Gamma(S_+^+)$ pure, is equivalent to a $SU(n)$-structure $(\Psi,\omega)$ with integrable almost complex structure $I$ and metric $g = \omega(\cdot, I\cdot)$, such that
\begin{equation}
\label{eq:twistedStrom}
\begin{split}
d \Psi - \theta_\omega \wedge \Psi & = 0,\\
d \theta_\omega & = 0,\\
dd^c \omega  & = 0,
\end{split}
\end{equation}
where $H = - d^c \omega$ and $\tfrac{1}{2}(\operatorname{div}_0 - \operatorname{div}) = \theta_\omega := I d^*\omega$ is the Lee form of $(\Psi,\omega)$.
\end{proposition}

\begin{proof}
Given a solution of \eqref{eq:killingE}, $\eta$ determines a reduction of the orthonormal frame bundle of $(M,g)$ to $SU(n)$. This reduction is equivalent to a pair $(\omega,\Psi)$, where $\omega$ is a non-degenerate real two-form and $\Psi$ is a locally decomposable complex $n$-form satisfying:
\begin{equation}\label{eq:SU(n)}
\omega \wedge \Psi = 0, \qquad i^{n(n-1)/2} (-1)^{n-1} \Psi \wedge \overline{\Psi} = \frac{\omega^n}{n!}\, .
\end{equation}
In terms of the almost complex structure $I$ determined by $\eta$, $\omega$ is of type $(1,1)$, $g = \omega(\cdot, I\cdot)$, and we have
\begin{equation*}
T^{0,1} M = \{X \in T \otimes \mathbb{C}: \iota_X \Psi = 0 \}.
\end{equation*} 
Using Lemma \ref{lem:DpmexpE}, the first equation in \eqref{eq:killingE} implies that the triple $(\omega,I,\Psi)$ is parallel with respect to $\nabla^{+}$. The second equation in \eqref{eq:killingE} implies, as it happens in \cite[Theorem 5.1]{grt}, that the almost complex structure is integrable and furthermore:
\begin{equation}\label{eq:Htheta}
H = - d^c \omega, \qquad \theta_\omega = \varphi_+.
\end{equation}
Using that $\operatorname{div}_0 - \operatorname{div} = \varphi = 2\varphi_+ \in \Omega^1(M)$ and $d\varphi = 0$, we obtain $d \theta_\omega = 0$. The rest of the proof follows as in \cite[Theorem 5.1]{grt}, using Gauduchon's formula for the Bismut connection $\nabla^+$ on the canonical bundle.
\end{proof}

\subsection{Killing spinors on homogeneous Hopf surfaces}\label{sec:KillingHopf}

We present a family of solutions of the Killing spinor equations \eqref{eq:killingE} on a four-dimensional compact Lie group, such that the triple $(V_+,\operatorname{div}_+,\eta)$ is left-invariant, and Proposition \ref{prop:geomalg} and Proposition \ref{lemma:KillingevenE} apply. 

Consider the compact four-dimensional manifold $M = S^3 \times S^1$ endowed with the canonical orientation. We use the Lie group structure given by identifying
$$
M \cong K = \SU(2) \times \U(1).
$$
Consider generators for the Lie algebra of left-invariant vector fields
$$
\mathfrak{k} = \mathfrak{su}(2) \oplus \mathbb{R} = \langle v_1, v_2, v_3, v_4 \rangle, \qquad  
$$
with relations
\begin{equation}\label{eq:brackets}
\left[v_2,v_3\right] = -v_1, \quad \left[v_3,v_1\right] = -v_2, \quad \left[v_1,v_2\right] = -v_3, \quad \left[v_4,\cdot\right] = 0.
\end{equation}
Equivalently,
$$
dv^1 = v^{23}, \quad dv^2 = v^{31}, \quad dv^3 = v^{12}, \quad dv^4 = 0,
$$
for $\{v^j\}$ the dual basis satisfying $v^j(v_k) = \delta_{jk}$. Here, we use the notation $v^{ij}=v^i\wedge v^j$. We take $\ell \in \mathbb{R}$ and define a left-invariant three-form
\begin{equation}\label{eq:Hell}
H_\ell = \ell v^{123}. 
\end{equation}
This corresponds to a constant multiple of the Cartan three-form on the $SU(2)$ factor, and hence it is bi-invariant and closed. Thus it defines an equivariant Courant algebroid $E _\ell = T \oplus T^*$, with pairing and Dorfman bracket as in \eqref{eq:pairingbracketexp} (with $H = H_\ell$).

Our next goal is to define a one-parameter family of left-invariant solutions of the Killing spinor equations on $E_\ell$. Given $x,a >0$ positive real numbers, consider the bi-invariant metric on $K$ defined by
$$
g_{x,a} = \frac{a}{x}(v^1 \otimes v^1 + v^2 \otimes v^2 + v^3 \otimes v^3 + x^2 v^4 \otimes v^4).
$$
We define a bi-invariant generalized metric on $E_\ell$
$$
V_\pm^{x,a} = \{v \pm g_{x,a}(v) \; | \; v \in \mathfrak{k}\}
$$
and a bi-invariant divergence $\operatorname{div}_{x,a} = \operatorname{div}_0 - \langle \varepsilon^x , \rangle$, for $\operatorname{div}_0$ the Riemannian divergence of $g_{x,a}$ (see \eqref{eq:div0}) and
$$
\varepsilon^x = - x v^4.
$$

\begin{lemma}\label{lem:examiso}
For any $x,a >0$ and $\ell \in \RR$, the pair $(V_\pm^{x,a},\varepsilon^x)$ is closed, that is,
$$
[\varepsilon^x,V_\pm^{x,a}] \subset V_{\pm}^{x,a}, \qquad \pi(\varepsilon^x) = 0.
$$
\end{lemma}
\begin{proof}
The statement follows trivially from \eqref{eq:infisometry}.
\end{proof}

It remains to specify a left-invariant spinor line $\langle \eta \rangle \subset S_+^+$ corresponding to an almost complex structure on $V_+^{x,a}$. For this, note that the anchor map $\pi \colon E_\ell \to T$ induces an isomorphism
$$
V_+^{x,a} \cong T,
$$
and we set
\begin{equation}\label{eq:Ix}
I_x v_4 = x v_1, \qquad I_x v_2 = v_3.
\end{equation}

For $\operatorname{div}^{x,a}_+:= \operatorname{div}_{x,a|V_+^{x,a}}$, we shall prove that $(V_+^{x,a},\operatorname{div}^{x,a}_+,I_x)$ defines a solution of the Killing spinor equations, provided that $\ell = a/x$. For this, we adopt an algebraic approach. By Proposition \ref{prop:geomalg}, the space of left-invariant section of $\Gamma(E_\ell)^{K}$ inherits a structure of quadratic Lie algebra
$$
\g_\ell := \Gamma(E_\ell)^{K}
$$
with underlying vector space $\g_\ell \cong \mathfrak{k} \oplus \mathfrak{k}^*$, pairing
$$
\qf{v + \alpha}{v + \alpha} = \alpha(v) 
$$
and Lie bracket
$$
[v + \alpha,w + \beta] = [v,w] - \beta([v,]) + \alpha([w,]) + \ell i_wi_v v^{123}.
$$
For fixed $x,a >0$, $(V_+^{x,a},\operatorname{div}_{x,a})$ can be regarded as the pair $(V_+^{x,a},\varepsilon^x)$ of generalized metric and divergence on the quadratic Lie algebra $\g_\ell$ (see Section \ref{sec:gmetrics}). Observe that, being $g_{x,a}$ bi-invariant, the Riemannian divergence $\operatorname{div}_0$ corresponds to the zero element in $\g_\ell^*$.

\begin{lemma}\label{lem:solutionsKsEq}
Denote by $\eta_x \in S_+^+$ a non-vanishing pure spinor in the spinor line determined by $I_x$. Then, $(V_+^{x,a},\varepsilon^x_+,\eta_x)$ is a solution of the Killing spinor equations \eqref{eq:killing} on $\g_\ell$ if and only if $\ell = a/x$. Consequently,  $(V_+^{x,a},\operatorname{div}^{x,a}_+,\eta_x)$ is a left-invariant solution of the Killing spinor equations \eqref{eq:killingE} on $E_\ell$ if and only if $\ell = a/x$.
\end{lemma}

\begin{proof}
By Proposition \ref{prop:Killingeven}, it suffices to prove that \eqref{eq:FtermDterm} holds if and only if $\ell = a/x$. We have that
$$
I_x(v+g_{x,a}(v)) := I_xv+g_{x,a}\left(I_xv\right).
$$
Hence, an oriented orthonormal basis for $V_+^{x,a}$ is given by
\begin{equation}
\left\lbrace \sqrt{\tfrac{x}{a}} v_2 + \sqrt{\tfrac{a}{x}} v^2, \sqrt{\tfrac{x}{a}} v_3 + \sqrt{\tfrac{a}{x}} v^3, \tfrac{1}{\sqrt{xa}} v_4 + \sqrt{xa} v^4, \sqrt{\tfrac{x}{a}} v_1 + \sqrt{\tfrac{a}{x}} v^1 \right\rbrace,
\end{equation}
with associated basis of $V_+^{1,0}$ and $V_+^{0,1}$ satisfying \eqref{eq:isotropybasis} given by
\begin{equation}\label{eq:epsbasisexam}
\begin{array}{rlrl}
\epsilon_1^+ & = \frac{1}{\sqrt{2}}\left(\left(\sqrt{\tfrac{x}{a}} v_2 + \sqrt{\tfrac{a}{x}} v^2\right) -i\left(\sqrt{\tfrac{x}{a}} v_3 + \sqrt{\tfrac{a}{x}} v^3\right)\right), &\overline{\epsilon}_1^+ & = \overline{\epsilon_1^+},\\
\epsilon_2^+ & = \frac{1}{\sqrt{2}}\left(\left(\tfrac{1}{\sqrt{xa}} v_4 + \sqrt{xa} v^4\right) -i\left(\sqrt{\tfrac{x}{a}} v_1 + \sqrt{\tfrac{a}{x}} v^1\right)\right), &\overline{\epsilon}_2^+ & = \overline{\epsilon_2^+}.
\end{array}
\end{equation}
Now, a direct calculation shows that
\begin{equation*}
\left[\epsilon_1^+,\epsilon_2^+\right] = \tfrac{1}{2}\left(\tfrac{x}{a}v_2 +\left(2 - \ell \tfrac{x}{a}\right) v^2 \right) - \tfrac{i}{2}\left(\tfrac{x}{a}v_3 +\left(2  - \ell \tfrac{x}{a}\right) v^3 \right).
\end{equation*}
Hence, $\left[\epsilon_1^+,\epsilon_2^+\right] \in V_+^\CC$ if and only if $\ell = a/x$ and, assuming this condition, it follows that $[V_+^{1,0},V_+^{1,0}] \in V_+^{1,0}$, where $V_+^{1,0} \subset V_+^{x,a} \otimes \CC$ is the $i$-eigenbundle of $I_x$. Similarly,
$$
v:= \sum_{j=1}^2 \left[\epsilon_j^+,\overline{\epsilon}_j^+\right] = i\left(-\tfrac{x}{a}v_1 - \left(2 - \ell \tfrac{x}{a}\right) v^1 \right),
$$
and therefore $v \in V_+^\CC$ if and only if $\ell = a/x$. Finally, assuming this condition and using that
$$
I_x(\varepsilon^x_+) = -\tfrac{1}{2}I_x\left(\tfrac{1}{a}v_4 + x v^4 \right) = -\tfrac{1}{2}\left(\tfrac{x}{a}v_1 + v^1 \right),
$$
it follows that the second equation in \eqref{eq:FtermDterm} holds. The last part of the statement follows by direct application of Proposition \ref{prop:geomalg}.
\end{proof}

\begin{remark}\label{rem:Zquotient}
Using Lemma \ref{lem:DpmexpE} and Proposition \ref{prop:geomalg}, the solutions of the Killing spinor equations in the previous lemma can be identified, via pull-back to the universal cover 
$$
\widetilde{M} = \mathbb{R}^4 \backslash \{0\},
$$
with the \emph{heterotic string soliton} constructed by Callan, Harvey, and Strominger \cite{CHS}. As pointed out by a referee, the \emph{dilaton} of this physical solution (given by the radius in $\mathbb{R}^4$) is not invariant under the radial $\mathbb{Z}$-action on $\widetilde{M}$, and hence the physical meaning of the quotient geometry is unclear. Nonetheless, the Killing spinor equations only depend on the total derivative of the dilaton, which is $\mathbb{Z}$-invariant and coincides with the divergence $\varepsilon_+^x$ via the natural isomorphism $V_+^{x,a} \cong T^*$.
\end{remark}

Our solutions of the Killing spinor equations in the previous lemma are such that $\varepsilon^x$ is an infinitesimal isometry for $V_+^{x,a}$ in the sense of Definition \ref{def:isometry} (see Lemma \ref{lem:examiso}). In the next result we prove that the dual vector field $\pi \varepsilon_+^x$ is holomorphic with respect to $I_x$.

\begin{lemma}\label{lem:examhol}
The left-invariant vector field $\pi \varepsilon^x_+$ is $I_x$-holomorphic, that is,
$$
[\varepsilon^x_+,V_+^{1,0}] \subset V_{+}^{1,0},
$$
where $V_+^{1,0} \subset V_+^{x,a} \otimes \CC$ is the $i$-eigenbundle of $I_x$.
\end{lemma}
\begin{proof}
The statement follows from the fact that $\varepsilon^x_+ = -\tfrac{1}{2}\left(\tfrac{1}{a}v_4 + x v^4 \right)$ is in the center of the Lie algebra $\g_\ell$.
\end{proof}

We analyse next our family of Killing spinors in terms of complex geometry by means of Proposition \ref{lemma:KillingevenE}. Assuming that $\ell = a/x$, the family of solutions of \eqref{eq:twistedStrom} induced by 
\begin{equation}\label{eq:V+x}
(V_+^{x},\operatorname{div}^{x}_+,\eta_x) := (V_+^{x,\ell x},\operatorname{div}^{x,\ell x}_+,\eta_x)
\end{equation}
is given by
\begin{equation}\label{eq:SU2structure}
\begin{split}
\omega_{x} & =  \ell x v^{41} + \ell v^{23},\\
\Psi_{x} & = \tfrac{\ell}{2}(iv^1 + x v^4)\wedge (v^2 + i v^3),
\end{split}
\end{equation}
with Lee form $\theta_{x} = -xv^4$. A classification of solutions of \eqref{eq:twistedStrom} on compact four-manifolds was obtained in \cite{grst}: either $(M,g,I)$ is a flat torus or a K3 surface with a K\"ahler Ricci-flat metric (and hence $\theta_\omega = 0$), or $(M,g,I)$ is a quaternionic Hopf surface and $\theta_\omega \neq 0$. To see this more explicitly, we note that the complex manifold $(K,I_x)$ is biholomorphic to the diagonal Hopf surface 
$$
X_x =  (\CC^2 \backslash \{0\})/\langle \gamma_x \rangle
$$
where $\langle \gamma_x \rangle \cong \mathbb{Z}$ is generated by $\gamma_x(z_1,z_2) = (e^xz_1,e^xz_2)$. 
With this identification, there is an isomorphism (see \cite[Thm. 1.2]{AngellaBC})
$$
H_A^{1,1}(X_x) \cong \mathbb{C}\langle [\eta^1_x \wedge \overline{\eta^1_x}] \rangle \cong \mathbb{C}\langle [v^{41}] \rangle,
$$
where $H_A^{1,1}(X_x)$ stands for the Aeppli cohomology group of $X_x$ and $\eta^1_x = iv^1 + x v^4$. Therefore, we obtain an interpretation of the parameter $a = \ell x > 0$ as the Aeppli class of the pluriclosed metric $\omega_{x}$ (see Proposition \ref{lemma:KillingevenE})
\begin{equation}\label{eq:Aeppli}
[\omega_{x}] = a [v^{41}] \in H_A^{1,1}(X_x).
\end{equation}

\subsection{T-dual Killing spinors}

In this section we apply T-duality to the family of solutions of the Killing spinor equations \eqref{eq:killingE} found in Lemma \ref{lem:solutionsKsEq}. By \cite[Theorem 6.5]{GF3}, solutions of \eqref{eq:killingE} are preserved under T-duality, and we shall prove that the T-dual of $(V_+^{x},\operatorname{div}^{x}_+,\eta_x)$ (see \eqref{eq:V+x}) is a different element in the same family.

We start recalling some background on topological T-duality following \cite{BEM,CaGu}. Let $T^k$ be a $k$-dimensional torus acting freely and properly on a smooth compact manifold $M$, so that $M$ is a principal $T^k$-bundle over the smooth manifold $B := M/T^k$. We endow $M$ with a choice of $T^k$-invariant cohomology class
$$
\tau \in H^3(M,\mathbb{R})^{T^k}.
$$
Fix another pair $(\hat M,\hat \tau)$ consisting of a smooth compact manifold $\hat M$ with a proper and free $T^k$-action such that $B=\hat M/T^k$, and $\hat\tau \in H^3(\hat M,\mathbb{R})^{T^k}$. 
Consider the fibre product $\overline M = M \times_B \hat M$ and the diagram
\begin{equation*}
  \xymatrix{
 & \ar[ld]_{q} \overline M \ar[rd]^{\hat{q}} & \\
 M \ar[rd]_{p} &  & \hat{M} \ar[ld]^{\hat{p}} \\
  & B & \\
  }
\end{equation*}

\begin{definition}\label{def:toptdual}
We say that two pairs $(M,\tau)$ and $(\hat M,\hat \tau)$ as above are T-dual if there exist representatives $H \in \Omega^3(M)^{T^k}$ and $\hat H \in \Omega^3(\hat M)^{\hat T^k}$ of $\tau$ and $\hat \tau$, respectively, such that
\begin{equation}\label{eq:relationTdual}
q^*H - \hat{q}^* \hat H = d \overline{B},
\end{equation}
where $\overline{B} \in \Omega^2(\overline M)^{T^k \times \hat T^k}$ is such that
\begin{equation}\label{eq:Bpairing}
\overline B \colon \operatorname{ker} d q \otimes \operatorname{ker} d \hat{q} \to \mathbb{R}
\end{equation}
is non-degenerate. 
\end{definition}


Given $M$ as above, the natural exact sequence
\begin{equation}\label{eq:2.1.3}
    0 \rightarrow \mathfrak{t} \rightarrow \frac{TM}{T^{k}} \rightarrow TB \rightarrow 0,
\end{equation}
where $\mathfrak{t} = B \times \operatorname{Lie} \; T^k$, induces a filtration \begin{equation}\label{eq:filtration}
\Omega^*(B) \cong \mathcal{F}^{0}\subset \mathcal{F}^{1} \subset \dots \subset \mathcal{F}^{\bullet}= \Omega^*(M)^{T^{k}},
\end{equation}
where $\mathcal{F}^{i}=\mathrm{Ann}(\wedge^{i+1} \mathfrak{t})$. 
Given now T-dual pairs $(M,\tau)$ and $(\hat M,\hat \tau)$, there exist representatives $H \in \tau$ and $\hat H \in \hat \tau$ in $\mathcal{F}^{1}$ of $M$ and $\hat M$, respectively, which satisfy the conditions in Definition \ref{def:toptdual} (see \cite[Lemma 10.5]{GFStreets}). 


Let $E$ be an exact Courant algebroid over $M$ which is equivariant with respect to the $T^k$-action.  Recall that such an $E$ has an associated \v Severa class
$$
[E] \in H^3(M,\RR)^{T^k}.
$$
Consider the vector bundle 
$$
E/T^k \to B,
$$
whose sheaf of sections is given by the invariant section of $E$, that is, $\Gamma(E/T^k) = \Gamma(E)^{T^k}$. We can endow $E/T^k$ with a natural structure of Courant algebroid, with pairing and Dorfman bracket given by the restriction of the neutral pairing and the Dorfman bracket on $E$ to $\Gamma(E)^{T^k}$. We will call $E/T^k$ the \emph{simple reduction} of $E$ by $T^k$. 

\begin{theorem}[\cite{CaGu}]\label{th:Tduality}
Let $E \to M$ and $\hat E \to \hat M$ be equivariant exact Courant algebroids. Assume that $(M,[E])$ is T-dual to $(\hat M,[\hat E])$. Then there exists a canonical isomorphism of Courant algebroids between the simple reductions
\begin{equation}\label{eq:psiisomorphism}
\psi \colon E/T^k \to \hat E/\hat T^k.
\end{equation}
\end{theorem}

We briefly describe the construction of the isomorphism $\psi$, which we will need. We choose equivariant isotropic splittings of $E$ and $\hat E$ such that the corresponding three-forms $H$ and $\hat H$ are in $\mathcal{F}^1$ of their respective fibrations. Given $X + \xi \in \Gamma(TM \oplus T^*M)^{T^k}$, choose the unique lift $\overline{X}$ of $X$ to $\Gamma(T\bar{M})^{T^k \times \hat T^k}$ such that
\begin{equation}\label{eq:liftvector}
q^* \xi(Y) - \overline{B}(\overline{X}, Y) = 0, \quad \mbox{ for all } Y \in \mathfrak t.
\end{equation}
Due to this condition the form $q^* \xi - \overline{B}(\overline{X}, \cdot)$ is basic for the
bundle determined by $\hat{q}$, and can therefore be pushed forward to
$\hat{M}$. Then, $\psi$ is defined by the explicit formula
\begin{align*}
\psi(X + \xi) = \hat{q}_*(\overline{X} + q^* \xi - \overline{B}(\overline{X}, \cdot)).
\end{align*}


To apply T-duality to the situation of our interest, we regard $K$ as in Section \ref{sec:KillingHopf} as a $T^1$-principal bundle over $S^3 \cong \SU(2)$, via the natural right action of the central subgroup
$$
T^1 = \U(1) \subset \SU(2) \times \U(1)
$$
given by the second factor. Given $\ell \in \mathbb{R}$ we consider the closed three-form $H_\ell$ in \eqref{eq:Hell}.

\begin{lemma}\label{lem:selfTdual}
For any $\ell \in \mathbb{R}$, the pair $(K,[H_\ell])$ is self-T-dual.
\end{lemma}

\begin{proof}
We regard the correspondence space $\overline{K}$ in Definition \ref{def:toptdual} inside the Lie group
$$
\iota \colon \overline{K} \hookrightarrow K \times \hat K,
$$
where $\hat K$ denotes another copy of $K$. Then, define $\overline{B}$ as the pull-back to $\overline{K}$ of the bi-invariant two-form 
\begin{equation}\label{eq:B}
\overline{B}  = - \iota^*(v^4 \wedge \hat v^4).
\end{equation}
Then we have
$$
d \overline{B} = 0 = p^{*}H_{\ell} - \hat{p}^{*}\hat{H}_{\ell},
$$
where we used that $H_{\ell}$ and $\hat{H}_{\ell}$ are both pull-back of the same three-form on $B =  \SU(2)$. The non-degeneracy condition on $\overline{B}$ follows from the fact that this two-form is bi-invariant. 
\end{proof}
 
We are ready to construct our pairs of T-dual solutions of the Killing spinor equations on the compact Lie group $K$. 

\begin{proposition}\label{prop:Tdual}
Given $0 < \ell \in \RR$, consider the equivariant exact Courant algebroid $E_\ell$ over $K$. Then two solutions $(V^{x}_+,\operatorname{div}^{x}_+,\eta_x)$ and $(V^{\hat x}_+,\operatorname{div}^{\hat x}_+,\eta_{\hat x})$ of the Killing spinor equations \eqref{eq:killingE} on $E_\ell$ in Lemma \ref{lem:solutionsKsEq} are exchanged under T-duality, provided that
\begin{equation}\label{eq:Tdualparameters}
\hat x = \frac{1}{\ell x}.   
\end{equation}

\end{proposition}

\begin{proof}
We calculate first the isomorphism \eqref{eq:psiisomorphism} in Theorem \ref{th:Tduality} corresponding to  the two-form \eqref{eq:B}. Firstly, notice that, since $T^1 \subset K$ is central, our global left-invariant frame $v_j,v^j \in \Gamma(E_\ell)^K$ is also invariant by the right $T^1$-action. Then, by \eqref{eq:liftvector} we have
$$
\psi(v_j) = \hat v_j, \qquad \psi(v^j) = \hat v^j, \textrm{ for } j= 1,2,3.
$$
A direct calculation also shows that
$$
\psi(v_4) = \hat v^4, \quad \psi(v^4) = \hat v_4.
$$
By definition of $V_+^{x}$ (see \eqref{eq:V+x}) we have
$$
V_+^{x} = \langle v_2 + \ell v^2, v_3 + \ell v^3, v_1 + \ell v^1, v_4 + \ell x^2 v^4  \rangle \subset T \oplus T^*,
$$
and therefore
$$
\hat V_+^x := \psi(V_+^{x}) = \langle \hat v_2 + \ell \hat v^2, \hat v_3 + \ell \hat v^3, \hat v_1 + \ell \hat v^1, \hat v^4 + \ell x^2 \hat v_4\rangle \subset T \oplus T^*.
$$
From this, the T-dual metric is
$$
\hat g_x = \ell(\hat v^1 \otimes \hat v^1 + \hat v^2 \otimes \hat v^2 + \hat v^3 \otimes \hat v^3 + (\ell x)^{-2} \hat v^4 \otimes \hat v^4) = g_{\hat x},
$$
where $\hat x$ is defined as in \eqref{eq:Tdualparameters}. Similarly, we have that
$$
\psi(\varepsilon^x_+) = -\tfrac{1}{2}\psi\left(\tfrac{1}{\ell x}v_4 + x v^4 \right) = -\tfrac{1}{2}\left(\tfrac{1}{\ell x}\hat v^4 + x \hat v_4 \right) = \varepsilon^{\hat x}_+,
$$
where $\varepsilon^{\hat x}_+$ denotes the orthogonal projection of $\varepsilon^{\hat x}$ onto $\hat V_+^x = V_+^{\hat x}$. Finally, the T-dual complex structure $\hat I_x := \psi I_x \psi^{-1}_{|V_+^{x}}$ is given by
\[
\hat I_x \hat v_2 = \hat v_3, \qquad \hat I_x \hat v_4 = \tfrac{1}{\ell x}\hat v_1.
\qedhere\]
\end{proof}

\begin{remark}
Observe that the volume element along the fibres of $K \to \SU(2)$ is constant on the base, and hence there is no \emph{dilaton shift} (see \cite[Proposition 6.8]{GF3}). This implies
$$
\psi_* \operatorname{div}_{x} = \operatorname{div}_{\hat x} + \langle x \hat v_4 - (\ell x)^{-1}\hat v_4, \cdot \rangle \neq \operatorname{div}_{\hat x},
$$
unless $\ell x^2  = 1$. Nonetheless, the Killing spinor equations \eqref{eq:killingE} only depend on $\operatorname{div}_+^{x}$, and this is precisely the quantity exchanged under T-duality in the previous result.
\end{remark}

Geometrically, Proposition \ref{prop:Tdual} implies that T-duality exchanges the complex structure parameter $x$ with the Aeppli class $a = \ell x$ of the T-dual solution (see \eqref{eq:Aeppli}), and vice versa: 
\begin{equation}\label{eq:Tdualparameters.bis}
\hat x = \frac{1}{a}, \qquad \hat a = \frac{1}{x}.
\end{equation}
This phenomenon is strongly reminiscent of the rotation of the Hodge diamond on mirror symmetry for algebraic Calabi--Yau manifolds. The goal of the next section is to turn this observation into a precise statement on (0,2) mirror symmetry, using vertex algebras.

\subsection{An example of $(0,2)$ mirror symmetry}\label{02example}

We start recalling the coordinate independent description of the chiral de
Rham complex using Courant algebroids \cite{Heluani09,GCYHeluani}. Let $E$ be 
a Courant algebroid over a smooth manifold $M$. 
Let $\Pi E$ be the corresponding purely odd super vector bundle. We will abuse
notation and denote by $\langle, \rangle$ the corresponding
super-skew-symmetric bilinear form, and by $[,]$ the Dorfman bracket on $\Pi E$. Similarly, we obtain an odd
differential operator 
$$
\mathcal{D}: C^\infty(M) \rightarrow \Gamma(\Pi E)
$$
defined by the composition of the exterior differential $d$, the map $\pi^* \colon T^* \to E^*$, and the isomorphism $E^* \cong E$ provided by $\langle, \rangle$, and the parity change operator $\Pi$. 

As in Section \ref{ssec:background}, we denote by $\cH$ the translation algebra. We denote by $\underline{\CC}$ the sheaf of locally constant functions on $M$.
The next result provides a coordinate-free description of the chiral de Rham complex on the smooth manifold $M$.

\begin{proposition}[\cite{Heluani09,GCYHeluani}]\label{prop:CDRE}
Let $E$ be a Courant algebroid over $M$. Then, there exists a unique sheaf of SUSY vertex algebras $\Omega^\mathrm{ch}_M(E)$ over $M$ endowed with embeddings of sheaves of $\underline{\CC}$-modules
$$
\iota \colon C^\infty(M) \hookrightarrow \Omega^\mathrm{ch}_M(E), \qquad j \colon \Pi E \hookrightarrow \Omega^\mathrm{ch}_M(E),
$$
satisfying the following properties:

\begin{enumerate}

\item $\iota$ is an isomorphism of unital commutative algebras onto its image
$$
\iota(fg) = :\iota(f)\iota(g):,
$$

\item $\iota$ and $j$ are compatible with the $C^\infty(M)$-module structure of $\Pi E$ and the $\cH$-module structure of $\Omega^\mathrm{ch}_M(E)$ 
$$
j(fA) = :\iota(f)j(A):, \qquad 2 S \iota(f) = j(\mathcal{D}f),
$$

\item $\iota$ and $j$ are compatible with the Dorfman bracket and pairing
$$
{[j(A)}_\Lambda j(B) ] = j([A,B]) + 2 \chi \iota (\langle A,B\rangle),
$$

\item $\iota$ and $j$ are compatible with the action of $\Gamma(\Pi E)$ on $C^\infty(M)$
$$
{[j(A)}_\Lambda \iota(f)] = \iota(\pi(A)(f))
$$

\item  $\Omega^\mathrm{ch}_M(E)$ is universal with these properties,

\end{enumerate}
	
for all $f,g \in C^\infty(M)$, $A,B \in \Gamma(\Pi E)$. Furthermore, when $E$ is the standard Courant algebroid $TM \oplus T^*M$, $\Omega^\mathrm{ch}_M(E)$ is the chiral de Rham complex of $M$.
	\label{prop:universal}
\end{proposition}

To relate this with our results from the previous section, assume now that our manifold $M$ is a compact Lie group $K$, and let $E$ be a left-equivariant Courant algebroid over $K$. By Proposition \ref{prop:geomalg}, we can associate to $E$ a quadratic Lie algebra given by the invariant sections of $E$
$$
\g = \Gamma(E)^K.
$$
Applying the universal construction in Proposition \ref{prop:universal}, we obtain an embedding of the universal superaffine vertex algebra $V^2(\g_{\text{super}})$ of level $k =2$.

\begin{proposition}\label{prop:superaffineembed}
Let $K$ be a compact Lie group and $E$ a left-equivariant Courant algebroid over $K$. Then there is an embedding
$$
V^2(\g_{\text{super}}) \hookrightarrow H^0(K,\Omega^\mathrm{ch}_K(E))
$$
of the universal superaffine vertex algebra $V^2(\g_{super})$ of level $k =2$ on the space of global sections $H^0(K,\Omega^\mathrm{ch}_K(E))$ of $\Omega^\mathrm{ch}_K(E)$.
\end{proposition}

\begin{proof}
Observe that $H^0(K,\Omega^\mathrm{ch}_K(E))$ inherits a natural structure of SUSY vertex algebra. We denote by $\mathcal{R}$ the underlying SUSY Lie conformal algebra.
By Example \ref{exam:superafin} and Proposition \ref{prop:CDRE}, we have an embedding of the Lie conformal algebra $\mathfrak{SCur}\mathfrak{g}$ localized at $k=2$ into $\mathcal{R}$. Now, being $V^2(\g_{\text{super}})$ the universal enveloping SUSY vertex algebra of $\mathfrak{SCur}\mathfrak{g}$ of level $k = 2$, this induces an embedding as in the statement. Indeed, this follows because any morphism from a SUSY Lie conformal algebra $\tilde{\mathcal{R}}$ into a SUSY vertex algebra can be extended to a unique SUSY vertex algebra morphism from the universal enveloping SUSY vertex algebra of $\tilde{\mathcal{R}}$.
\end{proof}

We are ready to prove our main result, which gives our examples of $(0,2)$ mirror pairs on compact non-K\"ahler manifolds. Our result is based on two basic observations combined with a theorem by Linshaw and Mathai which provides a stronger version of Theorem \ref{th:Tduality} in terms of the chiral de Rham complex \cite{LinshawMathai}.

We fix $0 < \ell \in \RR$ and identify a left-invariant solution $(V_+,\operatorname{div}_+,\eta)$ of the Killing spinor equations \eqref{eq:killingE} on the equivariant Courant algebroid $E_\ell$ over $K = \SU(2)\times \U(1)$, as in Lemma \ref{lem:solutionsKsEq}, with a solution $(V_+,\varepsilon_+,I)$ of the F-term and D-term equations \eqref{eq:FtermDterm} on $\g_\ell$ (see Proposition \ref{prop:Killingeven} and Proposition \ref{prop:geomalg}). Firstly, if $(V_+,\varepsilon_+,I)$ is a solution of \eqref{eq:FtermDterm} with $\varepsilon_+$ holomorphic, then so is $(V_+,\varepsilon_+,-I)$ (see Remark \ref{rem:flipI}). Secondly, if we denote
\begin{align*}
J := J(V_+,\varepsilon_+,I),\qquad  H :=  H(V_+,\varepsilon_+,I) \in V^2(\g_{\text{super}})
\end{align*}
the generators of the $N=2$ superconformal vertex algebra 
constructed in Theorem \ref{th:N=2dil}, then (see Remark  \ref{rem:Jflip})
\begin{equation}\label{eq:JSUSYflip}
J(V_+,\varepsilon_+,-I)=-J, \qquad H(V_+,\varepsilon_+,-I)=H.
\end{equation}
Finally, by a theorem of Linshaw and Mathai \cite{LinshawMathai}, the T-duality isomorphism $\psi$ in Theorem \ref{th:Tduality} induces an isomorphism
\begin{equation}\label{eq:LinMat}
\psi^{ch} \colon p_* \Omega^\mathrm{ch}_K(E_\ell)^{T^1} \to \hat p_* \Omega^\mathrm{ch}_K(E_\ell)^{T^1}.
\end{equation}
where $\Omega^\mathrm{ch}_K(E_\ell)^{T^1}$ is the sheaf of SUSY vertex algebras on $K$ generated by $T^1$-invariant sections of $E_\ell$ and functions on $K/T^1 \cong \SU(2)$ (see \cite[Section 5.2]{LinshawMathai} for a precise definition). Observe also that we have an embedding 
$$
V^2(\g_{\text{super}}) \hookrightarrow H^0(\SU(2),p_* \Omega^\mathrm{ch}_K(E_\ell)^{T^1}),
$$
by Proposition \ref{prop:superaffineembed}. The existence of the induced isomorphism $\psi^{ch}$ relies on the fact that $(M,[H_\ell])$ is self-T-dual (see Lemma \ref{lem:selfTdual} and \cite[Remark 6.3]{LinshawMathai}).

\begin{theorem}\label{thm:02}
Given $0 < \ell \in \RR$, consider the one-parameter family of solutions of the Killing spinor equations  $(V_+^{x},\operatorname{div}^{x}_+,I_x)$ on $E_\ell$ in Lemma \ref{lem:solutionsKsEq} parametrized by $x >0$. Then $(V_+^{x},\operatorname{div}^{x}_+,I_x)$ and $(V^{\hat x}_+,\operatorname{div}^{\hat x}_+,-I_{\hat x})$ are related by $(0,2)$ mirror symmetry provided that $\hat x = 1 /\ell x$. More precisely, if we denote
\begin{equation*}
\begin{split}
J & = J(V_+^{x},\operatorname{div}^{x}_+,I_x), \qquad \; \; \; H =  H(V_+^{x},\operatorname{div}^{x}_+,I_x)\\
\hat J & = J(V^{\hat x}_+,\operatorname{div}^{\hat x}_+,-I_{\hat x}), \qquad \hat H =  H(V^{\hat x}_+,\operatorname{div}^{\hat x}_+,-I_{\hat x})
\end{split}
\end{equation*}
the generators of the $N=2$ superconformal vertex algebras with central charge $c = 6 + 6/\ell$ constructed in Theorem \ref{th:N=2dil}, then the Linshaw--Mathai isomorphism \eqref{eq:LinMat} realises the mirror involution
\begin{equation}\label{eq:mirrorinv}
\psi^{ch}(J) = - \hat J, \qquad \psi^{ch}(H) = \hat H.
\end{equation}
\end{theorem}

\begin{proof}
The solutions of the Killing spinor equations $(V_+^{x},\operatorname{div}^{x}_+,I_x)$ constructed in Lemma \ref{lem:solutionsKsEq} (see \eqref{eq:V+x}) are such that the corresponding divergence $\varepsilon_+ \in V_+^{x}$ is holomorphic in the sense of Definition \ref{def:holoiso} (see Lemma \ref{lem:examhol}), and hence Theorem \ref{th:N=2dil} applies. The formula for the central charge follows from \eqref{eq:centralchargedilaton}
$$
c = 3\left(2+2\qf{\varepsilon_+^x}{\varepsilon_+^x}\right) = 6 + 6 x^2 \frac{1}{x^2 \ell} = 6 + 6/\ell.
$$
Observe that the expansion in Fourier modes of \eqref{eq:mirrorinv} recovers \eqref{eq:mirrorinvintro} (see Example \ref{exam:N2}). By \eqref{eq:JSUSYflip} it suffices to prove the identity 
$$
\psi^{ch}(J) = \tilde J := J(V^{\hat x}_+,\operatorname{div}^{\hat x}_+,I_{\hat x}).
$$
To see this, we write $w_j = \Pi v_j$ and $w^j = \Pi v^j$ for all $j \in \left\lbrace1,2,3,4\right\rbrace$, where the $v_j,v^j$ are as in Section \ref{ssec:KsCourant}.  Applying \eqref{eq:cuasicon1} and setting $a = \ell x$, a simple calculation shows that
\begin{equation*}
\begin{split}
J 
& = \frac{1}{2}\left(\frac{1}{\ell}:w_2w_3:+:w_2w^3:+:w^2w_3:+\frac{a}{x}:w^2w^3: \right.\\
& +\left.\frac{1}{a}:w_4w_1:+\frac{1}{x}:w_4w^1:+x:w^4w_1:+a:w^4w^1:\right) - \frac{1}{2}S\left(\frac{1}{\ell}w_1+w^1\right), \\
\tilde J 
& = \frac{1}{2}\left(\frac{1}{\ell}:\widehat{w}_2\widehat{w}_3:+:\widehat{w}_2\widehat{w}^3:+:\widehat{w}^2\widehat{w}_3:+\frac{a}{x}:\widehat{w}^2\widehat{w}^3: \right.\\
& +\left.x:\widehat{w}_4\widehat{w}_1:+a:\widehat{w}_4\widehat{w}^1:+\frac{1}{a}:\widehat{w}^4\widehat{w}_1:+\frac{1}{x}:\widehat{w}^4\widehat{w}^1:\right) -\frac{1}{2}S\left(\frac{1}{\ell}\hat w_1 + \hat w^1\right).
\end{split}
\end{equation*}
Now, using that $\psi^{ch}$ is an isomorphism of SUSY vertex algebras (that is, in particular, is an homomorphism for the normally ordered product and $S \psi^{ch} = \psi^{ch} S$), we obtain that
\begin{equation*}
\begin{split}
\psi^{ch}\left(J\right)
& = \frac{1}{2}\left(\frac{1}{\ell}:\widehat{w}_2\widehat{w}_3:+:\widehat{w}_2\widehat{w}^3:+:\widehat{w}^2\widehat{w}_3:+\frac{a}{x}:\widehat{w}^2\widehat{w}^3:\right.\\
& +\left.\frac{1}{a}: \widehat{w}^4 \widehat{w}_1: + \frac{1}{x}:\widehat{w}^4\widehat{w}^1: + x :\widehat{w}_4\widehat{w}_1: + a:\widehat{w}_4 \widehat{w}^1:\right) - \frac{1}{2}S\left(\frac{1}{\ell}\hat w_1 + \hat w^1\right) = \tilde J.
\qedhere\end{split}
\end{equation*}
\end{proof}

To finish this section, we show that the $(0,2)$ mirrors constructed in Theorem \ref{thm:02} have isomorphic $T^1$-\emph{equivariant half-twisted models}. For the rigorous definition of the half-twisted model as a holomorphic sheaf of vertex algebras, we follow \cite[Section 8]{StructuresHeluani}. Consider the operators
\begin{equation}
Q_0 = \tfrac{1}{2}(H_{\qf{0}{1}} + i J_{\qf{0}{0}}), \qquad G_0 = \tfrac{1}{2}(H_{\qf{0}{1}} - i J_{\qf{0}{0}}),
\end{equation}
on $\Omega_K^{ch}(E_\ell)$  defined by the Fourier modes of $J$ and $H$ in Theorem \ref{thm:02} (see Example \ref{exam:N2}).

\begin{definition}\label{def:halftwisted}
Given $0 < \ell,x \in \RR$, consider the solution of the Killing spinor equations $(V_+^{x},\operatorname{div}^{x}_+,I_x)$ on $E_\ell$ in Lemma \ref{lem:solutionsKsEq}.
\begin{enumerate}
\item The \emph{equivariant $A$-half twisted model} of $(V_+^{x},\operatorname{div}^{x}_+,I_x)$ is the sheaf of SUSY vertex algebras given by the cohomology
$$
\Omega_{x,A}^{T^1} := H^*(\Omega_K^{ch}(E_\ell)^{T^1},Q_0).
$$
\item The \emph{equivariant $B$-half twisted model} of $(V_+^{x},\operatorname{div}^{x}_+,I_x)$ is the sheaf of SUSY vertex algebras given by the cohomology
$$
\Omega_{x,B}^{T^1} := H^*(\Omega_K^{ch}(E_\ell)^{T^1},G_0).
$$
\end{enumerate}
\end{definition}

The fact that the equivariant $A$-half and $B$-half twisted models are well-defined
follows as in \cite{StructuresHeluani} once we notice that $(g_x,I_x)$ is part of a generalized K\"ahler structure, where the second complex structure is given by the right invariant extension of $I_x$ to $K$ (see \cite[Example 4.8]{CaGu}). The conditions $Q_0^2 = G_0^2 = 0$ follow from the generalized K\"ahler identities. To give a more explicit description of the twisted models, recall that our solution of the Killing spinor equations $(V_+^{x},\operatorname{div}^{x}_+,I_x)$ has associated holomorphic Courant algebroids
$$
\mathcal{Q}_{x,\ell} \to X_x := (K,I_x), \qquad \overline{\mathcal{Q}}_{x,\ell} \to \overline{X}_x := (K,-I_x)
$$
defined by reduction (see \cite[Theorem 1.20]{G2})
$$
\mathcal{Q}_{x,\ell} := ((V_+^x)^{0,1})^\perp/(V_+^x)^{0,1}, \qquad \overline{\mathcal{Q}}_{x,\ell} := ((V_+^x)^{1,0})^\perp/(V_+^x)^{1,0}.
$$
Here, $((V_+^x)^{0,1})^\perp$ is the orthogonal complement of the $-i$-eigenbundle $(V_+^x)^{0,1}$ of $I_x$ on 
$$
V_+^x \otimes \CC \subset E_\ell \otimes \mathbb{C}.
$$
As in Proposition \ref{prop:universal}, $\mathcal{Q}_{x,\ell}$ (resp. $\overline{\mathcal{Q}}_{x,\ell}$) has an associated holomorphic sheaf of SUSY vertex algebras $\Omega^{ch}_{X_x}(\mathcal{Q}_{x,\ell})$ over $X_x$ (resp. $\Omega^{ch}_{\overline{X}_x}(\overline{\mathcal{Q}}_{x,\ell})$ over $\overline{X}_x$). As in \cite[Section 4.2]{LinshawMathai}, we can define subsheaves
$$
\Omega^{ch}_{X_x}(\mathcal{Q}_{x,\ell})^{T^1} \subset \Omega^{ch}_{X_x}(\mathcal{Q}_{x,\ell}), \qquad \Omega^{ch}_{\overline{X}_x}(\overline{\mathcal{Q}}_{x,\ell})^{T^1} \subset \Omega^{ch}_{\overline{X}_x}(\overline{\mathcal{Q}}_{x,\ell})
$$
as the commutants with the zero mode of $\Pi v_4$, where $v_4$ is the holomorphic vector field generating the $T^1$-action. The next result is a direct consequence of \cite[Proposition 13]{StructuresHeluani}.

\begin{proposition}[\cite{StructuresHeluani}]\label{prop:halftwist}
There are canonical isomorphisms of sheaves of SUSY vertex algebras 
$$
\Omega_{x,A}^{T^1} \cong \Omega^{ch}_{X_x}(\mathcal{Q}_{x,\ell})^{T^1}, \qquad \Omega_{x,B}^{T^1} \cong \Omega^{ch}_{\overline{X}_x}(\overline{\mathcal{Q}}_{x,\ell})^{T^1}.
$$
\end{proposition}

Our next result provides the desired identification between the equivariant half-twisted models for the pairs of $(0,2)$ mirrors in Theorem \ref{thm:02}.

\begin{theorem}\label{thm:halftwist}
Given $0 < \ell,x \in \RR$, set $\hat x = 1/\ell x$ and consider the pair of $(0,2)$ mirrors $(V_+^{x},\operatorname{div}^{x}_+,I_x)$ and $(V^{\hat x}_+,\operatorname{div}^{\hat x}_+,-I_{\hat x})$ in Theorem \ref{thm:02}. Then, the Linshaw--Mathai isomorphism \eqref{eq:LinMat} induces isomorphisms of sheaves of SUSY vertex algebras
$$
\Omega_{x,A}^{T^1} \cong \Omega_{\hat x,B}^{T^1}, \qquad \Omega_{x,B}^{T^1} \cong \Omega_{\hat x,A}^{T^1}.
$$
Consequently, there is an isomorphism of sheaves of SUSY vertex algebras
\begin{equation}\label{eq:halftwistmirror}
\Omega^{ch}_{X_x}(\mathcal{Q}_{x,\ell})^{T^1} \cong \Omega^{ch}_{X_{\hat x}}(\mathcal{Q}_{\hat x,\ell})^{T^1}.
\end{equation}
\end{theorem}
\begin{proof}
The first part of the statement follows from Theorem \ref{thm:02}. In particular, \eqref{eq:mirrorinv} implies
$$
\psi^{ch}(Q_0) = \hat G_0, \qquad \psi^{ch}(G_0) = \hat Q_0.
$$
The second part of the statement follows from Proposition \ref{prop:halftwist}.
\end{proof}

The identity \eqref{eq:halftwistmirror} is remarkable, as it gives an isomorphism of holomorphic sheaves of SUSY vertex algebras over different complex manifolds. It shall be compared with a fundamental result of Borisov and Libgober \cite{BorLib}, which matches the elliptic genera of a Calabi--Yau hypersurface on a Fano toric manifold with the one of its mirror. Here, the term holomorphic stands for the fact that $\Omega^{ch}_{X_x}(\mathcal{Q}_{x,\ell})^{T^1}$ is locally generated by the $T^1$-invariant holomorphic sections of $\mathcal{Q}_{x,\ell}$ (after parity change) and $T^1$-invariant holomorphic functions on $X_x$. Observe that, in the present setup, there is a natural elliptic fibration structure 
$$
X_x \to \mathbb{P}^1
$$
and any $T^1$-invariant holomorphic function on $X_x$ is necessarily pull-back from $\mathbb{P}^1$. Inspired by \cite{KontsevichHMS}, one can speculate that there should be a homological version of $(0,2)$ mirror symmetry for pairs $(X,\mathcal{Q})$, given by a compact complex manifold $X$ and a \emph{holomorphic string algebroid} $\mathcal{Q}$ \cite{grt2}, which replaces \eqref{eq:halftwistmirror} by an equivalence of suitable categories of vertex algebra modules. To start, it would be interesting to obtain an explicit calculation of $\Omega^{ch}_{X_x}(\mathcal{Q}_{x,\ell})^{T^1}$ in the present example.

\begin{remark}
The use of generalized K\"ahler geometry in Definition \ref{def:halftwisted} suggests that Theorem \ref{thm:02} can possibly be regarded as an example of $(2,2)$-mirror symmetry. However, we would like to stress that neither our methods nor the ones from \cite{StructuresHeluani} apply directly in the present situation to achieve such result. On the one hand, the generalized K\"ahler structure underlying $(g_x,I_x)$ is not left-invariant, and hence it cannot be regarded as a pair Killing spinors on the quadratic Lie algebra $\g_\ell$. On the other hand, \cite[Theorem 2]{StructuresHeluani} only applies to \emph{generalized Calabi--Yau metric structures}, which are always K\"ahler on a compact manifold.
\end{remark}

\section{Applications of Theorem \ref{th:N=2}}
\label{sec:furtherex}

\subsection{A family of $N=4$ algebras}\label{sec:N=4}

In this section, we discuss two applications of Theorem \ref{th:N=2}. In our first application we will show that each element of the family of $N=2$ superconformal vertex-algebra embeddings induced by Lemma \ref{lem:solutionsKsEq} and Theorem \ref{th:N=2} embeds in a $N=4$ superconformal vertex algebra with central charge $c = 6$. Our construction seems to be related to a $N=4$ superconformal algebra originally discovered by Sevrin, Troost and Van Proeyen \cite{STVan} (see also \cite{CHS}), but we have not been able to find a precise match.

We fix $\ell, x > 0$ and consider the solution of the Killing spinor equations $(V_+^{x},\operatorname{div}^{x}_+,I_x)$ (see \eqref{eq:V+x}) on the equivariant Courant algebroid $E_\ell$ over $K = \SU(2)\times \U(1)$ obtained in Lemma \ref{lem:solutionsKsEq}. We start with the observation that $g_{x}$ is compatible with a left-invariant hyperholomorphic structure $(I_x,J_x,K_x)$ on $K$, wiht $I_x$ defined by~\eqref{eq:Ix}, and $J_x$ and $K_x$ defined using~\eqref{eq:SU2structure} by
$$
2 \Psi_{x} = \omega_{J_x} + i \omega_{K_x},
$$
where $g_x=\omega_{J_x}(,J_x) = \omega_{K_x}(,K_x)$. More explicitly, we have
$$
J_x v_4 = x v_2, \qquad J_x v_3 = v_1, \qquad  K_x v_1 = v_2, \qquad K_x v_4 = x v_3,
$$
and it is straightforward to check that the standard Hamilton relations hold
$$
I_x^2 = J_x^2 = K_x^2 = I_xJ_xK_x = - \Id.
$$
In the next result, we show that  $(V_+^{x},\operatorname{div}^{x}_+)$ combined with either $I_x$, $J_x$ or $K_x$, gives a solution of the Killing spinor equations \eqref{eq:killingE}.

\begin{lemma}\label{lem:solutionsKsEqhyper}
For any $\ell, x > 0$, $(V_+^{x},\operatorname{div}^{x}_+,I_x)$, $(V_+^{x},\operatorname{div}^{x}_+,J_x)$ and $(V_+^{x},\operatorname{div}^{x}_+,K_x)$ are left-invariant solutions of the Killing spinor equations \eqref{eq:killingE} on $E_\ell$. Consequently, $(I_x,J_x,K_x)$ is a hyperholomorphic structure compatible with $g_{x}$ and fixed Lee form 
	$$
	\theta_x= - x v^4.
	$$
\end{lemma}
\begin{proof}
The claim about the complex structure $I_x$ has been checked in Lemma \ref{lem:solutionsKsEq}. We check that $(V_+^{x},\operatorname{div}^{x}_+,K_x)$ is a solution of $\eqref{eq:killingE}$ on $E_\ell$, and leave the other case for the reader. By Proposition \ref{lemma:KillingevenE}, it suffices to prove that
	$$
	\omega_{K_x} = \ell v^{12} + \ell x v^{43}, \qquad \Psi_{K_x} = (v^1 + i v^2) \wedge (iv^3 + x v^4)
	$$
	satisfies \eqref{eq:twistedStrom} with $d^c_{K_x}\omega_{K_x} = - H_\ell$ and $\theta_{\omega_{K_x}} = - x v^4$. We calculate
	\begin{align*}
	d \Psi_{K_x} & = i x v^{43} \wedge (v^1 + i v^2) = - x v^4 \wedge \Psi_{K_x},\\
	d \omega_{K_x} & = - \ell x v^{412} = - x v^4 \wedge \omega_{K_x},\\
	d^c_{K_x} \omega_{K_x} & = - d\omega_{K_x}(K_x,K_x,K_x) = \ell x v^{412}(K_x,K_x,K_x) = - \ell v^{123} = - H_\ell.
	\end{align*}
	The statement follows from the structure equation for the Lee form on a complex surface, given by $d \omega_{K_x} = \theta_{\omega_{K_x}} \wedge \omega_{K_x}$.
\end{proof}

In the next result, we construct the desired family of $N=4$ superconformal vertex algebras by application of Theorem \ref{th:N=2}. By Lemma \ref{lem:solutionsKsEqhyper} and Proposition \ref{prop:Killingeven}, it follows that
$$
w := w^{J_x} = w^{I_x} = w^{K_x} = \varepsilon_+^x,
$$
and hence condition \eqref{eq:orthogonalderived} holds by Lemma \ref{lem:examhol}.

\begin{proposition}\label{prop:hyperrotation}
The solutions of the Killing spinor equations in Lemma \ref{lem:solutionsKsEqhyper} induce an embedding of the $N=4$ superconformal vertex algebra of central charge $c = 6$ into the universal superaffine vertex algebra associated to $\mathfrak{g}_\ell$ with level $k=2$. More precisely, if we denote by 
\begin{alignat*}{3}
J_0^I & = J_0(V_+^{x},\operatorname{div}^{x}_+,I_x),
&\qquad&
H'_I &= H'(V_+^{x},\operatorname{div}^{x}_+,I_x),
\\
J_0^J & = J_0(V_+^{x},\operatorname{div}^{x}_+,J_x),
&\qquad&
H'_J &=  H'(V_+^{x},\operatorname{div}^{x}_+,J_x),
\\
J_0^{K} & = J_0(V^{x}_+,\operatorname{div}^{x}_+,K_x),
&\qquad&
H'_{K} &=  H'(V^{x}_+,\operatorname{div}^{x}_+,K_x),
\end{alignat*}
the generators of each $N=2$ superconformal vertex algebra constructed in Theorem \ref{th:N=2}, then $H'_I=H'_J=H'_K$ and furthermore
	$$
	\left[{J_0^I}_\Lambda J_0^J\right] = -\left(2\chi+S\right)J_0^K, \quad \left[{J_0^J}_\Lambda J_0^K\right] = -\left(2\chi+S\right)J_0^I, \quad \left[{J_0^K}_\Lambda J_0^I\right] = -\left(2\chi+S\right)J_0^J.
	$$
\end{proposition}
\begin{proof}
As in the proof of Theorem \ref{thm:02}, we obtain the explicit formulae
	\begin{equation*}
	\begin{split}
	J_0^J & = \frac{1}{2\ell}\left(:\left(w_3+\ell w^3\right)\left(w_1+\ell w^1\right):+ :\left(\frac{1}{x}\left(w_4+\ell x^2 w^4\right)\right)\left(w_2+\ell w^2\right):\right),\\
	J_0^K & = \frac{1}{2\ell }\left(:\left(w_1+\ell w^1\right)\left(w_2+\ell w^2\right):+ :\left(\frac{1}{x}\left(w_4+ \ell x^2 w^4\right)\right)\left(w_3+\ell w^3\right):\right).
	\end{split}
	\end{equation*}
We compute $\left[{J_0^I}_\Lambda J_0^J\right]$ and leave the rest of $\Lambda$-brackets as an exercise for the reader. By the non-commutative Wick formula, we have
	\begin{equation*}
	\begin{split}
	\left[{J_0^I}_\Lambda J_0^J\right] & = \frac12\left(\frac{1}{\ell}:\left[{J_0^I}_\Lambda\left(w_3+\ell w^3\right)\right]\left(w_1+\ell w^1\right): \right.\\
	&-\frac{1}{\ell}:\left(w_3+\ell w^3\right)\left[{J_0^I}_\Lambda\left(w_1+\ell w^1\right)\right]:\\
	&+\frac{1}{\ell}\int_0^\Lambda d\Gamma\left[\left[{J_0^I}_\Lambda\left(w_3+\ell w^3\right)\right]_\Gamma\left(w_1+\ell w^1\right)\right]\\
	&+\frac{1}{\ell x}:\left[{J_0^I}_\Lambda\left(w_4+\ell x^2 w^4\right)\right]\left(w_2+\ell w^2\right):\\
	&-\frac{1}{\ell x}:\left(w_4+\ell x^2w^4\right)\left[{J_0^I}_\Lambda\left(w_2+\ell w^2\right)\right]:\\
	&+\left.\frac{1}{\ell x}\int_0^\Lambda d\Gamma\left[\left[{J_0^I}_\Lambda\left(w_4+\ell x^2w^4\right)\right]_\Gamma\left(w_2+\ell w^2\right)\right]\right).
	\end{split}
	\end{equation*}
We compute the $\Lambda$-brackets:
	\begin{equation*}
	\begin{split}
	\left[\left(w_3+\ell w^3\right)_\Lambda{J_0^I}\right] & = -\chi\left(w_2+\ell w^2\right)+\frac{1}{2\ell}\left(:\left(w_1+\ell w^1\right)\left(w_3+\ell w^3\right):\right.\\
	&+\left.:\left(\frac{1}{x}\left(w_4+\ell x^2w^4\right)\right)\left(w_2+\ell w^2\right):\right),\\
	\left[{\left(w_1+\ell w^1\right)}_\Lambda{J_0^I}\right] & = -\lambda-\chi\frac{1}{x}\left(w_4+\ell x^2 w^4\right),\\
	\left[\left(w_2+\ell w^2\right)_\Lambda{J_0^I}\right] & = \chi\left(w_3+\ell w^3\right)-\frac{1}{2\ell}\left(:\left(w_2+\ell w^2\right)\left(w_1+\ell w^1\right):\right.\\
	&-\left.:\left(\frac{1}{x}\left(w_4+\ell x^2w^4\right)\right)\left(w_3+\ell w^3\right):\right),\\
	\left[\left(w_4+\ell x^2w^4\right)_\Lambda{J^I_0}\right] & = x\chi\left(w_1+\ell w^1\right).
	\end{split}
	\end{equation*}
Combining the non-commutative Wick formula, the antisymmetry of $\Lambda$-bracket, \eqref{eq:cuasicon1}, and \eqref{eq:cuasicon2}, we conclude
	\begin{equation*}
	\begin{split}
	\left[{J^I_0}_\Lambda J^J_0\right] & = \frac{1}{2\ell}\left(\chi:\left(w_2+\ell w^2\right)\left(w_1+\ell w^1\right):+:\left(S\left(w_2+\ell w^2\right)\right)\left(w_1+\ell w^1\right):\right.\\
	& +\frac{1}{x}\left(\chi:\left(w_3+\ell w^3\right)\left(w_4+\ell x^2w^4\right):-:\left(w_3+\ell w^3\right)\left(S\left(w_4+\ell x^2w^4\right)\right):\right)\\
	& -\left(\lambda+T\right)\left(w_3+\ell w^3\right)+\lambda\left(w_3+\ell w^3\right)-\\
	& - \chi:\left(w_1+\ell w^1\right)\left(w_2+\ell w^2\right):-:\left(S\left(w_1+\ell w^1\right)\right)\left(w_2+\ell w^2\right):\\
	& -\frac{1}{x}\left(\chi:\left(w_4+\ell x^2w^4\right)\left(w_3+\ell w^3\right):-:\left(w_4+\ell x^2w^4\right)\left(S\left(w_3+\ell w^3\right)\right):\right) \\
	& = -\left(2\chi+S\right)J^K_0,
	\end{split} 
	\end{equation*}
	because $S$ is an antiderivation for the normally ordered product. Finally, the identity $H'_I =H'_J = H'_K$ follows calculating a basis as in \eqref{eq:isotropybasis} in each case and substituting in the formula \eqref{eq:NSmas0KSE}.
\end{proof}

\begin{remark}
It is interesting to observe that the method of Proposition \ref{prop:hyperrotation} does not apply to the family of $N=2$ superconformal structures with central charge $c = 6 + 6/\ell$ in Theorem \ref{thm:02}. This seems to agree with the analysis of the Sevrin--Troost--Van Proeyen algebra made in \cite[Section 8]{CHS}. Note also that the value of the central charge $c = 6$ of our family of $N=4$ algebras coincides with the expected value from superstring theory. 

\end{remark}

Proposition \ref{prop:hyperrotation} suggests a generalization of Theorem \ref{th:N=2} for $N=4$ algebras, when $\dim V_+ = 4 k$ and $G_\eta = \operatorname{Sp}(k)$ in Lemma \ref{lem:gravitino} (cf. Remark \ref{rem:G2}). Here $\operatorname{Sp}(k)$ denotes the compact symplectic group. We hope to go back to this question in future work.

\subsection{$N=2$ algebras from the Hull-Strominger system}\label{sec:HS}

In this last section we discuss a different application of Theorem \ref{th:N=2}, related to invariant solutions of the Hull-Strominger system  \cite{HullTurin,Strom} on a Lie group. More precisely, we will show that, starting with such a solution, one can construct a quadratic Lie algebra endowed with a solution of the Killing spinor equations \eqref{eq:killing}, such that Theorem \ref{th:N=2} applies. Furthermore, we will provide an infinite family of examples for which $l \oplus \overline{l}$ is not a Manin triple (see Remark \ref{rem:Getzler}). 

For simplicity, we restrict to a family of solutions of the Hull-Strominger system recently studied in \cite{GFGM}. We will follow the abstract definition of the equations in \cite{grst}. The same arguments can be applied to other invariant solutions in the literature (see \cite{FeiYau,OUVi} and references therein). Consider the complex Heisenberg Lie group
$$
H_\CC =\left\{ \left(\begin{array}{c c c}
1 & z_2 & z_3\\
0 & 1 & z_1\\
0 & 0 & 1\\
\end{array}\right) \hspace{1mm} |\hspace{1mm} z_i \in \mathbb{C}\right\}.
$$
The following $1$-forms define a global left-invariant holomorphic frame of $T^*_{1,0}$ 
$$
\omega_1=dz_1 \hspace{2mm} , \hspace{2mm} \omega_2=dz_2 \hspace{2mm} , \hspace{2mm} \omega_3=dz_3-z_2 dz_1
$$
and satisfy the structure equations
$$
d\omega_1=d\omega_2=0 \hspace{3mm}, \hspace{3mm} d\omega_3=\omega_{12},
$$
from which all the exterior algebra relations can be derived. For any choice of
$$
(m,n,p)\in \mathbb{R}^{3}\backslash \{0\}
$$
we consider the following purely imaginary $(1,1)$-form
\begin{equation}\label{lbtorus}
F=\pi (m(\omega_{1\overline{1}}-\omega_{2\overline{2}})+n(\omega_{1\overline{2}}+\omega_{2\overline{1}})+ip(\omega_{1\overline{2}}-\omega_{2\overline{1}})).
\end{equation}
Consider the left-invariant $\SU(3)$ structure on $H_\CC$ defined by 
\begin{align}\label{Iwasu3}
    \Omega=\omega_{123} \hspace{3mm}, \hspace{3mm} \omega=\tfrac{i}{2}(\omega_{1\overline{1}}+\omega_{2\overline{2}}+\omega_{3\overline{3}}). 
\end{align}

\begin{proposition}[\cite{GFGM}]\label{prop:HSsol}
Let $\alpha = (2\pi^2(m^2+n^2+p^2))^{-1}  \in \RR$. Then, with the notation above, the pair $(\omega,F)$ is a solution of the Hull-Strominger system, that is,
\begin{equation}\label{eq:HSabstract}
\begin{split}
F^{0,2} = 0, \qquad F \wedge \omega^2 & =0,\\
d(\|\Omega\|_\omega \omega^2) & = 0,\\
dd^c \omega - \alpha F \wedge F & = 0.
\end{split}
\end{equation}
\end{proposition}

Starting from a solution of \eqref{eq:HSabstract}, it was proved in \cite{grt} that one can construct a solution of the Killing spinor equations on a transitive Courant algebroid over $H_\CC$ (cf. Proposition \ref{lemma:KillingevenE}). Taking left-invariant sections, one obtains a quadratic Lie algebra endowed with a solution of the Killing spinor equations \eqref{eq:killing}, similarly as in Proposition \ref{prop:geomalg}. Using that the solution in Proposition \ref{prop:HSsol} is left-invariant, it follows that $\|\Omega\|_\omega$ is constant and hence $\omega$ is balanced, that is, $\theta_\omega = 0$. From this, the induced solution of $\eqref{eq:killing}$ has zero divergence and Theorem \ref{th:N=2} applies. Rather than given the details of this general argument, we shall provide here an explicit direct proof.

Denote by $\mathfrak{h}_\CC$ the Lie algebra of $H_\CC$. For any choice of $(m,n,p)\in \mathbb{R}^{3}\backslash \{0\}$ one can define a real quadratic Lie algebra with underlying vector space 
$$
\g_{m,n,p} = \mathfrak{h}_\CC \oplus i \RR \oplus \mathfrak{h}_\CC^*,
$$ 
pairing
$$
\qf{v + r + \beta}{v + r + \beta} = \beta(v) - \alpha r^2
$$
and Lie bracket
\begin{align*}
[v + r + \beta,w + t + \eta] & = [v,w] - \eta([v,]) + \beta([w,]) + i_wi_v (d^c\omega)\\
& - F(v,w) + 2\alpha (r i_w F - t i_v F),
\end{align*}
where $F$ and $\alpha$ are defined as in Proposition \ref{prop:HSsol}. One can readily check that the definition of $\alpha$ is necessary for the Lie bracket to satisfy the Jacobi identity.  More explicitly, taking a real basis of $\mathfrak{h}_\CC^*$ defined by
$$
\omega_1 = v^1 + i v^2, \qquad \omega_2 = v^3 + i v^4, \qquad \omega_3 = v^4 + i v^6,
$$
one has relations
$$
dv^j = 0, \quad j = 1,2,3,4, \qquad dv^5 = v^{13} - v^{24}, \quad dv^6 = v^{14} + v^{23}
$$
and the complex structure on $\mathfrak{h}_\CC$ reads (for $v_j$ the dual basis)
$$
I v_1 = v_2, \quad I v_3 = v_4, \quad I v_5 = v_6.
$$
From this, it follows that
$$
d^c\omega = v^{135} + v^{236} + v^{146} - v^{245}   
$$
and also that
$$
F = 2 \pi i (m(v^{34} - v^{12})+n(v^{23} - v^{14})+p(v^{13} + v^{24})).
$$
We prove next that the previous data determines a solution of the Killing spinor equations on $\g_{m,n,p}$ with zero divergence. Consider the generalized metric on $\g_{m,n,p}$ defined by 
$$
V_+ = \{v + g(v) \; | \; v \in \mathfrak{h}_\CC\}, \qquad V_- = \{v + r - g(v) \; | \; v + r \in \mathfrak{h}_\CC \oplus i \RR \},
$$
where 
$$
g := \omega(,I) = v^1 \otimes v^1 + v^2 \otimes v^2 + v^3 \otimes v^3 + v^4 \otimes v^4 + v^5 \otimes v^5 + v^6 \otimes v^6.
$$
Consider the spinor line $\langle \eta \rangle \subset S_+^+$ corresponding to the complex structure on $V_+ \cong \mathfrak{h}_\CC$. 

\begin{lemma}\label{lem:solutionsKsEqHS}
The triple $(V_+,0,\eta)$ is a solution of the Killing spinor equations \eqref{eq:killing} on $\g_{m,n,p}$. Furthermore, $V_+^\CC \subset \g_{m,n,p}^\CC := \g_{m,n,p} \otimes_\RR \CC$ is a Lie subalgebra if and only if $m = 0$.
\end{lemma}

\begin{proof}
By Proposition \ref{prop:Killingeven}, it suffices to prove that \eqref{eq:FtermDterm} holds. A basis of $V_+^{1,0}$ and $V_+^{0,1}$  satisfying \eqref{eq:isotropybasis} is given by
\begin{equation}\label{eq:epsbasisexam}
\begin{array}{rlrl}
\epsilon_1^+ & = \frac{1}{\sqrt{2}}\left(\left(v_1 + v^1\right) -i\left(v_2 + v^2\right)\right), &\overline{\epsilon}_1^+ & = \overline{\epsilon_1^+},\\
\epsilon_2^+ & = \frac{1}{\sqrt{2}}\left(\left(v_3 + v^3\right) -i\left(v_3 + v^3\right)\right), &\overline{\epsilon}_2^+ & = \overline{\epsilon_2^+},\\
\epsilon_3^+ & = \frac{1}{\sqrt{2}}\left(\left(v_5 + v^5\right) -i\left(v_6 + v^6\right)\right), &\overline{\epsilon}_1^+ & = \overline{\epsilon_1^+}.
\end{array}
\end{equation}
Now, a direct calculation shows that
\begin{equation*}
\left[\epsilon_1^+,\epsilon_2^+\right] = - \sqrt{2}\epsilon_3^+, \qquad \left[\epsilon_1^+,\epsilon_3^+\right] = 0,\qquad \left[\epsilon_2^+,\epsilon_3^+\right] = 0,
\end{equation*}
and therefore $[V_+^{1,0},V_+^{1,0}] \in V_+^{1,0}$. Similarly,
\begin{equation*}
\left[\epsilon_1^+,\overline{\epsilon}_1^+\right] = - 2\pi m, \qquad \left[\epsilon_2^+,\overline{\epsilon}_2^+\right] = 2\pi m,\qquad \left[\epsilon_3^+,\overline{\epsilon}_3+\right] = 0,
\end{equation*}
and therefore the first part of the statement follows. From the previous formula we also conclude that $m \neq 0$ implies that $V_+^\CC \subset \g_{m,n,p}^\CC$ is not a Lie subalgebra. The other implication is left to the reader.
\end{proof}

By direct application of Theorem \ref{th:N=2}, we obtain an embedding of the $N=2$ superconformal vertex algebra of central charge $c = 9$ into the universal superaffine vertex algebra associated to $\g_{m,n,p}^\CC$, for any level $0 \neq k \in \CC$. When $m \neq 0$, these embeddings are not associated to Manin triples, showing that Theorem \ref{th:N=2} provides a strict generalization of Getzler's result in \cite{Getzler} (see Remark \ref{rem:Getzler}).

\begin{proposition}\label{prop:N=2HS}
For any choice $(m,n,p)\in \mathbb{R}^{3}\backslash \{0\}$, the solution of the Killing spinor equations in Lemma \ref{lem:solutionsKsEqHS} induces an embedding of the $N=2$ superconformal vertex algebra of central charge $c = 9$ into the universal superaffine vertex algebra $V^k((\g_{m,n,p}^\CC)_{super})$ with level $0 \neq k \in \CC$. The generators of each $N=2$ superconformal vertex algebra are as in Theorem \ref{th:N=2}. 
\end{proposition}

\begin{remark}\label{rem:CDRHS}
As proved in \cite{grt}, associated to a solution of the Hull-Strominger system there exists a transitive Courant algebroid $E$. Denote by $E_{m,n,p}$ the transitive Courant on $H_\CC$ determined by the solution in Proposition \ref{prop:HSsol}. Then, similarly as in Proposition \ref{prop:superaffineembed}, when $k =2$ there is an embedding
$$
V^2((\g_{m,n,p}^\CC)_{super}) \hookrightarrow H^0(H_\CC,\Omega^\mathrm{ch}_{H_\CC}(E_{m,n,p}))
$$
on the space of global sections of the sheaf of vertex algebras associated to $E_{m,n,p}$ (see Proposition \ref{prop:CDRE}). Applying now Proposition \ref{prop:N=2HS}, one obtains an embedding of the $N=2$ superconformal vertex algebra of central charge $c = 9$ into $H^0(H_\CC,\Omega^\mathrm{ch}_{H_\CC}(E_{m,n,p}))$. We expect that a similar result holds for general solutions of the Hull-Strominger system. We hope to go back to this question in future work.
\end{remark}

\appendix
\section{Rules and identities on supersymmetric vertex algebras}
\label{app:1}

In this appendix, we collect relevant remarks and identities about SUSY vertex algebras, especially in the case of the universal superaffine vertex algebras.


\subsection{SUSY Lie conformal algebras}

Given a SUSY Lie conformal algebra $\cR$, the terms in~\eqref{eq:sesquiLambda},~\eqref{eq:comLambda} and~\eqref{eq:JacobiLambda} are calculated as follows (see~\cite[Def. 4.10]{SUSYVA} for details):

\begin{itemize}
	\item \textit{Sesquilinearity}. To obtain $(S+\chi)[a_{\Lambda}b]$ as an element of $\mathcal{L} \otimes\cR$ in \eqref{eq:sesquiLambda}, we first calculate the $\Lambda$-bracket and then commute $S$ with $\chi$ and $\lambda$ to the right using the relations $\left[S,\lambda\right]=0$ and $\left[S,\chi\right]=2\lambda$.
	\item \textit{Skew-symmetry}. To obtain $\left[b_{-\Lambda-\nabla}a\right]$ as an element of $\mathcal{L} \otimes\cR$ in \eqref{eq:comLambda}, we first expand $\left[b_\Gamma a\right]=\sum_{n,J}\Gamma^{n|J}c_{n|J}$ in $\cL'\otimes\cR$ using the relations $\left[\gamma,\eta\right]=0$, $\left[\eta,\eta\right]=-2\gamma$, for a copy $\cL'$ of $\cL$ with the formal variables $\Lambda=(\lambda,\chi)$ replaced by $\Gamma=(\gamma,\eta)$, and then replace $\Gamma$ by $-\Lambda-\nabla=(-\lambda-T,-\chi-S)$ and apply $T$ and $S$ to the coefficients $c_{n|J}\in\cR$.
	\item \textit{Jacobi identity}. To obtain $\left[\left[a_\Lambda b\right]_{\Lambda+\Gamma}c\right]\in\mathcal{L} \otimes \mathcal{L}'\otimes\cR$ in \eqref{eq:JacobiLambda}, we first calculate 
$\left[\left[a_\Lambda b\right]_{\Psi}c\right]\in\mathcal{L} \otimes \mathcal{L}''\otimes V$, where $\cL''$ is another copy of $\cL$ with the formal variables $\Lambda$ replaced by $\Psi$, and then replace $\Psi$ by $\Lambda+\Gamma=(\lambda+\gamma,\chi+\eta)$ and use the relations $[\lambda,\gamma]=[\lambda,\eta]=[\chi,\gamma]=[\chi,\eta]=0$. 
  To calculate the other two terms in \eqref{eq:JacobiLambda} as elements of $\mathcal{L} \otimes \mathcal{L}'\otimes\cR$, we use the relations $[\gamma,\lambda]=[\gamma,\chi]=[\eta,\lambda]=[\eta,\chi]=0$ and 
\begin{equation}\label{eq:MUYIMPORTANTE}
[a_\Lambda q(\Gamma)b]=(-1)^{\left|q\right|(\left|a\right|+1)}q(\Gamma)[a_\Lambda b]
\end{equation}
for any homogeneous polynomial $q(\Gamma)$ in $\Gamma$ (this formula follows applying the Koszul rule to the \emph{odd} $\Lambda$-bracket~\cite[Remark 4.9]{SUSYVA}).
\end{itemize}

\subsection{SUSY vertex algebras}

Given a SUSY vertex algebra $V$, the terms in~\eqref{eq:cuasicon},~\eqref{eq:cuasiaso} and~\eqref{eq:Wick} are calculated as follows (see~\cite[(3.3.3.2); Theorem 3.3.14; (3.2.6.12)]{SUSYVA} for details):

\begin{itemize}
\item \textit{Quasicommutativity}. To compute the integral in \eqref{eq:cuasicon}, we use the expansion $\left[b_\Lambda a\right]=\sum_{n,J}\Lambda^{n|J}c_{n|J}$ in $\cL\otimes V$, and on each term $p(\Lambda)=\Lambda^{n|J}c_{n|J}\in\cL\otimes V$, apply the formula 
\begin{equation}\label{eq:def-definite-integral.1}
\int_{-\nabla}^0 d\Lambda\, p=\int_{-T}^0 d\lambda\left(\partial_{\chi}p\right),
\end{equation}
taking first the partial derivative with respect to the odd variable $\chi$, performing the indefinite integral in the even variable $\lambda$, and taking the difference of the values at the limits.
Here, the (left) partial derivative $\partial_{\chi}p$ of $p=\Lambda^{n|J}c_{n|J}\in\cL\otimes V$ is zero if $J=0$, and $\lambda^nc_{n|J}$ if $J=1$.

\item \textit{Quasiassociativity}.
The fist integral in \eqref{eq:cuasiaso} is computed by expanding the $\Lambda$-bracket $[b_\Lambda c]$ as in~\eqref{eq:cuasicon}, putting the powers of $\Lambda$ on the left under the integral sign and performing the definite integral 
\begin{equation}\label{eq:def-definite-integral.2}
\int^{\nabla}_0 d^r\!\Lambda\, q=\int^{T}_0d\lambda\left(\partial^r_{\chi}q\right)
\end{equation}
to each term $q(\Lambda)=a\Lambda^{n|J}$ inside the parenthesis,
where the \emph{right} partial derivative $\partial^r_{\chi}q$ of $q(\Lambda)=a\Lambda^{n|J}$ with respect to the odd variable $\chi$ is zero if $J=0$ and $\lambda^na$ if $J=1$. 
The second integral in \eqref{eq:cuasiaso} is calculated applying the same rules.

\item \textit{The non-commutative Wick formula}. To compute the integral in~\eqref{eq:Wick}, 
we expand
\[
[[a_\Lambda b]_\Gamma c]=\sum_{j,k\in\Z\atop J,K=0,1}\Lambda^{j|J}\Gamma^{k|K}c_{j|J,k|K}
\]
as an element of $\cL\otimes\cL'\otimes V$, using the relations
$[\lambda,\gamma]=[\lambda,\eta]=[\chi,\gamma]=[\chi,\eta]=0$ and~\eqref{eq:MUYIMPORTANTE}, 
and take the definite integral of each term as in~\eqref{eq:def-definite-integral.1}.
\end{itemize}

In the notation of Heluani--Kac~\cite[(3.3.3.2); Theorem 3.3.14; (3.2.6.12)]{SUSYVA}, the integrals~\eqref{eq:def-definite-integral.1} and~\eqref{eq:def-definite-integral.2} would be represented as $\int_{-\nabla}^0p\,d\Lambda $ and $\int^{\nabla}_0 d\Lambda\,q$, respectively. 
Our small notational changes will be used in our calculations of Sections~\ref{sec:SCVA} and~\ref{sec:Tdual}.
\begin{remark}
For simplicity, in the quasiassociativity identity \eqref{eq:cuasiaso} we have omited the parenthesis that determine the order of computing the normally ordered products, which is clear from the notation we are using. 
\end{remark}
\begin{remark}
	In the identities \eqref{eq:cuasiaso} and \eqref{eq:Wick}, we must use that
	\begin{equation}\label{eq:nopLambda}
	:\left(p(\Lambda)a\right)b: = p(\Lambda):ab: \text{ and } :a\left(p(\Lambda)b\right): = (-1)^{|a||p|}p(\Lambda):ab:, \quad \text{for } a,b \in \mathcal{R},
	\end{equation}
	for any homomegeneous polynomial $p(\Lambda)$ in $\Lambda$ (this follows after applying the Koszul rule to the \textit{even} normally ordered product).
\end{remark}
\subsection{Superaffine vertex algebras}

Let $V^k(\g_{\text{super}})$ be the universal superaffine vertex algebra with level $k$ associated to a quadratic Lie algebra $\left(\mathfrak{g},\qf{\cdot}{\cdot}\right)$, for a scalar $k\in\C$ (see Example \ref{exam:superafin}). 

\begin{lemma}\label{lem:tech1}
For all $a,b,c \in \Pi\mathfrak{g}$, the following identities hold:
\begin{align}
		:ab: & = -:ba:\label{eq:cuasicon1}\\
		:a\left(Tb\right): & = -:\left(Tb\right)a:\label{eq:cuasicon11}\\
		:a\left(Sb\right):&=:\left(Sb\right)a:+T\left[a,b\right]\label{eq:cuasicon2}\\
		:a:bc::&=::bc:a:+kT\left(\qf{a}{b}c-\qf{a}{c}b\right)\label{eq:cuasicon3}\\
		::a\left(Sb\right):c:&=:a:\left(Sb\right)c::+:\left(Ta\right)\left[b,c\right]:+kT\qf{a}{c}Sb\label{eq:cuasiaso1}\\
		::ab:c:&=:a:bc:: + kT\left(\qf{c}{b}a-\qf{c}{a}b\right)\label{eq:cuasiaso2}\\
         \begin{split}
		:a:bc::&=:b:ca:: \\
		& = -:b:ac::\label{eq:cuasiconaso1}
              \end{split}
  \\
         \begin{split}
          :a:bc::&=:c:ab::\\
		&=-:c:ba::\label{eq:cuasiconaso2}
        \end{split}
\end{align}
\end{lemma}
\begin{proof}
The identities \eqref{eq:cuasicon1}, \eqref{eq:cuasicon11}, \eqref{eq:cuasicon2} and \eqref{eq:cuasicon3} are immediate consequences of quasicommutativity, while \eqref{eq:cuasiaso1} and \eqref{eq:cuasiaso2} follow from quasiassociativity. The identities \eqref{eq:cuasiconaso1} and \eqref{eq:cuasiconaso2} are obtained applying the previous ones. 
\end{proof}

\begin{lemma}\label{lem:tech4}
For all $a,b,c,d \in \Pi\mathfrak{g}$, the following identities hold:
\begin{gather}
		:a::bc:d::  = - :::bc:d:a: + k\left(\qf{a}{b}T\left(:cd:\right)-\qf{a}{c}T\left(:bd:\right)+\qf{a}{d}T\left(:bc:\right)\right).\label{eq:cuasicon4}\\ 
          ::a:bc::d:  = :a::bc:d:: + k\left(:\left(Ta\right)\left(\qf{d}{c}b-\qf{d}{b}c\right):+\qf{a}{d}T\left(:bc:\right)\right).\label{eq:cuasiaso3}
\end{gather}
\end{lemma}
\begin{proof}
The first identity follows from quasicommutativity, and the second one follows from quasiassociativity. 
\end{proof}
\begin{remark}
Observe that the formula for the Weyl endomorphism \eqref{eq:chipm} appears naturally in several of the previous identities.
\end{remark}


\begin{thebibliography}{12}
\frenchspacing\smallbreak

\bibitem{ABS}
A. Adams, A. Basu \and S. Sethi, 
\emph{(0,2) Duality},  
Adv. Theor. Math. Phys. {\bf 7} (2004) 865--950.

\bibitem{AldiHel} 
M. Aldi \and R. Heluani,
\emph{On a Complex-Symplectic Mirror Pair}, 
Int. Math. Res. Not. {\bf 2018} (22) (2018) 6934--6960

\bibitem{AXu} 
A. Alekseev \and P. Xu,
\emph{Derived brackets and courant algebroids}.
Unpublished, available at \url{http://www.math.psu.edu/ping/anton-final.pdf} (2001).

\bibitem{AngellaBC} 
D. Angella, G. Dloussky \and A. Tomassini, 
\emph{On Bott-Chern cohomology of compact complex surfaces}, 
Ann. Mat. Pura Appl. \textbf{195} (2016) 199--217.



\bibitem{Barron00}
K. Barron,
\emph{$N=1$ Neveu--Schwarz vertex operator superalgebras over Grassmann algebras and with odd formal variables}.
In: \emph{Representations and Quantizations (Shanghai, 1998)}.
China High. Educ. Press, Beijing, 2000, pp. 9--35.

\bibitem{BZHS}
D. Ben-Zvi, R. Heluani \and M. Szczesny,
\emph{Supersymmetry of the chiral de Rham complex},
Compositio Math. \textbf{144} (2008) 503--521.

\bibitem{Borisov}
L. Borisov, 
\emph{Vertex algebras and mirror symmetry}, 
Comm. Math. Phys. {\bf 215} (2001) 517--557.

\bibitem{Borisov02}
L. Borisov \and R. Kaufmann, 
\emph{On CY-LG correspondence for (0,2) toric mirror models}, 
Adv. Math. {\bf 230} (2012) 531--551.

\bibitem{BorLib}
L. Borisov \and A. Libgober,
\emph{Elliptic genera of toric varieties and applications to mirror
symmetry}, 
Invent. Math. (2) {\bf 140} (2000) 453--485.

\bibitem{BEM}
P. Bouwknegt, J. Evslin \and V. Mathai,
\emph{T-duality: topology change from $H$-flux},
Comm. Math. Phys. {\bf 249} (2004) 383--415.
		
		
\bibitem{CHS} 
C. Callan, J. Harvey \and A. Strominger, 
\emph{Supersymmetric string solitons}, 
Trieste 1991 Proceedings, String theory and quantum gravity, 1991, pp. 208--244.
		
\bibitem{COGP}
P. Candelas, X. De la Ossa, P. Green \and L. Parkes, 
\emph{A pair of Calabi--Yau manifolds as an exactly soluble superconformal field theory}, 
Nucl. Phys. B {\bf 359} (1991) 21--74.		
		
\bibitem{CaGu}
G. Cavalcanti \and M. Gualtieri,
\emph{Generalized complex geometry and T-duality}. 
In: \emph{A Celebration of the Mathematical Legacy of Raoul Bott}.
CRM Proceedings \& Lecture Notes. Amer. Math. Soc. 2010, pp. 341--366.
        
\bibitem{CSW} A. Coimbra, C. Strickland-Constable \and D. Waldram, 
\emph{Supergravity as Generalised Geometry I: Type II Theories}, 
JHEP {\bf 11} (2011) 91.

  
\bibitem{Dine} 
M. Dine, I. Ichinose \and N. Seiberg, 
\emph{F Terms and D Terms in String Theory}, 
Nucl. Phys. B {\bf 293} (1987) 253--265.
  
\bibitem{Donagi}
R. Donagi, J. Guffin, S. Katz \and E. Sharpe, 
\emph{A mathematical theory of quantum sheaf cohomology}, 
Asian J. Math. {\bf 18} (2014) 387--418.

\bibitem{FeiYau} 
T. Fei \and S.-T.~Yau, 
\emph{Invariant Solutions to the Strominger System on Complex Lie Groups and Their Quotients}, 
Comm. Math. Phys. {\bf 338} (2015) 1183--1195.

\bibitem{GGG} 
A. Gadde, S. Gukov \and P. Putrov, 
{\em (0,2)-trialities}, 
JHEP {\bf 1403} (2014) 076.


\bibitem{GF3} 
M. Garcia-Fernandez, 
\emph{Ricci flow, Killing spinors, and T-duality in generalized geometry}, 
Adv. Math. {\bf 350} (2019) 1059--1108.
 
\bibitem{GF4}
\bysame, \emph{T-dual solutions of the Hull--Strominger system on non-K\"ahler threefolds}, J.
Reine Angew. Math. {\bf 776} (2020) 137--150.

\bibitem{GFGM} 
M. Garcia-Fernandez \and R. Gonz\'alez Molina, 
\emph{Harmonic metrics for the Hull-Strominger system and stability}, to appear (2022).

\bibitem{grt} 
M. Garcia-Fernandez, R. Rubio \and C. Tipler, 
\emph{Infinitesimal moduli for the Strominger system and Killing spinors in generalized geometry}, 
Math. Ann. {\bf 369} (2017) 539--595.

\bibitem{grt2} 
\bysame, 
\emph{Holomorphic string algebroids}, 
Trans. Amer. Math. Soc. (10) {\bf 373} (2020) 7347--7382.

\bibitem{grst} 
M. Garcia-Fernandez, R. Rubio, C. Shahbazi \and C. Tipler, 
\emph{Canonical metrics on holomorphic Courant algebroids}, Proc. London Math. Soc. {\bf 125} (3) (2022) 700--758.

\bibitem{grst2} 
\bysame, 
\emph{Heterotic supergravity and moduli stabilization}, to appear.

\bibitem{GFStreets} 
M. Garcia-Fernandez \and J. Streets, 
\emph{Generalized Ricci flow}, University Lecture Series {\bf 76}, 2021, 248 pp, American Mathematical Society.

\bibitem{Getzler}
E. Getzler,
\emph{Manin Pairs and Topological Conformal Field Theory}, 
Annals Phys. {\bf 237} (1995) 161--201.
		
\bibitem{GSV} 
V. Gorbounov, F. Malikov \and V. Schechtman, 
\emph{Gerbes of chiral differential operators. II. Vertex algebroids}, 
Invent. Math. {\bf 155} (2001) 605--680.

\bibitem{GGSharpe} 
W. Gu, J. Guo \and E. Sharpe,
\emph{A proposal for nonabelian $(0,2)$-mirrors}, 
\href{https://arxiv.org/abs/1908.06036}{\tt arXiv:1908.06036 [hep-th]}.

\bibitem{G3} 
M. Gualtieri, 
\emph{Branes on Poisson varieties}. 
In: \emph{The many facets of geometry}.
Oxford Univ. Press, Oxford, 2010, pp. 368--394.

    
\bibitem{G2}
\bysame, 
\emph{Generalized K\"ahler Geometry}, 
Comm. Math. Phys. (1) {\bf 331} (2014) 297--331.

    


\bibitem{Heluani09}
R. Heluani,
\emph{Supersymmetry of the chiral de Rham complex II: Commuting sectors}, 
Int. Math. Res. Not.  {\bf 2009} (6) (2009) 953--987.

\bibitem{HeluaniRev} 
\bysame,
 \emph{Recent advances \and open questions on the susy structure of the chiral de Rham Complex}, 
J. Phys. A: Math. Theor. {\bf 50} (2017) 423002.

\bibitem{SUSYVA}
R.~ Heluani \and V.~ G.~ Kac,
\emph{Supersymmetric vertex algebras},  
Comm. Math. Phys. {\bf 271} (2007) 103--178.

\bibitem{GCYHeluani}
R.~ Heluani \and M.~ Zabzine, 
\textit{Generalized Calabi--Yau manifolds and the chiral De Rham complex}. Adv. Math. 
{\bf 223} (2010) 1815--1844.

\bibitem{StructuresHeluani}
R.~ Heluani \and M.~ Zabzine,
\textit{Superconformal Structures on Generalized Calabi--Yau Metric Manifolds}, 
Comm. Math. Phys. {\bf 306} (2011) 333--364.

\bibitem{Hit1}
N. Hitchin,
\emph{Generalized Calabi--Yau manifolds}, 
Q. J. Math {\bf 54} (2003) 281--308.
		

\bibitem{HullTurin}
C. Hull, 
\emph{Superstring compactifications with torsion and space-time supersymmetry}, 
In: Turin 1985 Proceedings ``Superunification and Extra Dimensions'', 1986, pp. 347--375.

\bibitem{Kac}
V.~ G.~ Kac.
\emph{Vertex algebras for beginners}.
University Lecture series, Vol. {\bf 10}. Providence,
RI: Amer. Math. Soc. Second Edition, 1998.


\bibitem{Kapustin}
A. Kapustin, 
\emph{Chiral de Rham complex and the half-twisted sigma-model}, 
CALT-68-2547, \href{https://arxiv.org/abs/hep-th/0504074}{\tt arXiv:0504074} (2005).



\bibitem{KontsevichHMS}
M. Kontsevich, 
\emph{Homological Algebra of Mirror Symmetry}. 
In: S. D. Chatterji (ed.), \emph{Proceedings of the International Congress of Mathematicians, Z\"urich, 1994}. Birkh\"auser, 1995, pp. 120--139.

\bibitem{LTY}
S.-C. Lau, L.-S. Tseng \and S.-T. Yau, 
\emph{Non-K\"ahler SYZ mirror symmetry}, 
Comm. Math. Phys. \textbf{340} (2015) 145--170.

\bibitem{MicLaw}
H. Lawson \and M. Michelsohn,
\emph{Spin geometry}. 
Princeton Mathematical Series \textbf{38}. Princeton University Press, Princeton, N. J., 1989.

\bibitem{LinshawMathai}
A. Linshaw \and V. Mathai,
\emph{Twisted Chiral de Rham Complex, Generalized Geometry, and T-duality},  
Comm. Math. Phys. {\bf 339} (2015) 663--697.
		
\bibitem{MSV} 
F. Malikov, V. Schechtman \and A. Vaintrob, 
\emph{Chiral de Rham complex}, 
Comm. Math. Phys. {\bf 204} (1999) 439--473.
		
\bibitem{McOrist}
J. McOrist, 
\emph{The revival of (0,2) linear $\sigma$-models}, 
Int. J. Mod. Phys. A {\bf 26} (2011) 1--41.

\bibitem{MelBook}
I. V. Melnikov, 
\emph{An introduction to two-dimensional quantum field theory with (0,2) supersymmetry}. 
Lecture Notes in Physics {\bf 951}, Springer, Cham, 2019.

\bibitem{MelPle}
I. V. Melnikov \and M. R. Plesser, 
\emph{A $(0,2)$ mirror map}, JHEP {\bf 1102} (2011) 001.


\bibitem{MelSeShe}
I. Melnikov, S. Sethi \and E. Sharpe, 
\emph{Recent developments in (0,2) mirror symmetry}, 
SIGMA {\bf 8} (2012) 068.

\bibitem{OUVi} 
A. Otal, L. Ugarte \and R. Villacampa, 
\emph{Invariant solutions to the Strominger system and the heterotic equations of motion}, 
Nucl. Phys. B {\bf 920} (2017) 442--474.

\bibitem{PPZ}
D.-H. Phong, S. Picard \and X. Zhang,
\emph{Geometric flows and Strominger systems}, 
Math. Z. {\bf 288} (2018) 101--113.
        
\bibitem{Popovici} 
D. Popovici,
\emph{Non-K\"ahler Mirror Symmetry of the Iwasawa Manifold}, 
Int. Math. Res. Notices \textbf{2020}, no. 23, 9471--9538.

		
\bibitem{SV2}
P. \v Severa \and F. Valach,
\emph{Courant algebroids, Poisson-Lie T-duality, and type II supergravities}, 
Comm. Math. Phys. {\bf 375} (2020) 307--344.
		
\bibitem{STVan}
A. Sevrin, W. Troost \and A. van Proeyen,
\emph{Superconformal Algebras in Two Dimensions with $N=4$}, 
Phys. Lett. {\bf B208} (1988) 447--450.
		
\bibitem{Strom} A.~Strominger, \emph{Superstrings with torsion}, Nucl. Phys. B {\bf 274} (1986) 253--284.
	
\bibitem{SYZ} A.~Strominger, S.-T. Yau \and E. Zaslow, \emph{Mirror symmetry is T-Duality}, Nucl. Phys. B {\bf 479} (1996) 243--259.

\bibitem{Tan} M.-C. Tan, \emph{Two-dimensional twisted sigma models, the mirror chiral de Rham complex, and twisted generalized mirror symmetry}, JHEP {\bf 07} (2007) 013.

\bibitem{Ward} A. Ward, \emph{Homological mirror symmetry for elliptic Hopf surfaces}, \href{https://arxiv.org/abs/2101.11546}{\tt arXiv:2101.11546 [math.SG]}.

\bibitem{WittenMS} E. Witten, \emph{Mirror manifolds and topological field theory}.
In: S. T. Yau (ed.), \emph{Essays on Mirror Manifolds}, International Press, Hong Kong, 1992, 120--158.

\bibitem{WittenN2} \bysame, \emph{Phases of $N=2$ theories in two dimensions}, Nuclear Phys. {\bf B 403} (1993) 159--222.

\bibitem{Witten02} \bysame, \emph{Two dimensional models with (0,2) supersymmetry: perturbative aspects}, Adv. Theor. Math. Phys. \textbf{11} (2007) 1--63.
				
\end{thebibliography}
\end{document}